\documentclass{lmcs}
\pdfoutput=1

\usepackage{lastpage}
\lmcsdoi{15}{1}{5}
\lmcsheading{}{\pageref{LastPage}}{}{}%
{Jan.~20,~2016}{Jan.~29,~2019}{}

\usepackage[all]{xy}
\xyoption{2cell}
\xyoption{curve}
\UseAllTwocells

\SelectTips{cm}{}

\usepackage{amsfonts,amsbsy,mathrsfs,latexsym,amssymb,mathbbol,stmaryrd}
\usepackage[mathcal]{euscript}

\usepackage{color}

\renewcommand\thefootnote{\hbox to 6pt{\hss\arabic{footnote}}}

\renewcommand{\phi}{\varphi}
\renewcommand{\frak}[1]{\mathfrak{#1}}

\def\op{{\mathit{op}}}

\def\Id{{\mathsf{Id}}}

\def\ol#1{\overline{#1}}

\def\Colim#1#2{{#1}\mathop{*}{#2}}

\def\Lan#1#2{{\mathrm{Lan}}_{#1} #2}

\def\One{{\mathbb{1}}}
\def\Two{{\mathbb{2}}}

\def\Sum{\bullet}
\def\tensor{\otimes}

\def\extra{{{\sf o}}}
\def\hs{\heartsuit}

\def\kat#1{{\mathscr{#1}}}

\def\V{\kat{V}}

\def\A{\kat{A}}
\def\B{\kat{B}}
\def\C{\kat{C}}

\def\K{\kat{K}}
\def\X{\kat{X}}
\def\Y{\kat{Y}}
\def\Z{\kat{Z}}

\def\Sum{\bullet}
\def\tensor{\otimes}

\newcommand{\ccal}{\C}

\def\XX{{\mathbb{X}}}
\def\YY{{\mathbb{Y}}}

\def\AA{{\mathbb{A}}}
\def\NN{{\mathbb{N}}}

\def\Set{{\mathsf{Set}}}
\def\Fin{{\mathsf{Fin}}}
\def\Pos{{\mathsf{Pos}}}
\def\Pre{{\mathsf{Preord}}}

\newcommand{\Coalg}{\mathsf{Coalg}}
\newcommand{\Endo}{\mathsf{Endo}}

\def\Vcat{\V\hspace{-1.5pt}\mbox{-}\hspace{.9pt}{\mathsf{cat}}}
\def\Twocat{\Two\mbox{-}\hspace{.9pt}{\mathsf{cat}}}

\def\d{\mathsf{d}}
\def\cv{\mathsf{c}}
\def\uv{\mathsf{v}}

\newcommand{\coop}{\mathsf{coop}}
\newcommand{\Rel}{\mathsf{Rel}}
\newcommand{\RelPresh}{\mathsf{RelPresh}}

\renewcommand{\to}{\longrightarrow}

\begin{document}


\title[Extending set functors to generalised metric spaces]
      {Extending set functors to generalised metric spaces}

\author{Adriana Balan}
\address{Department of Mathematical Methods and Models, University Politehnica of Bucharest, Romania}
\email{adriana.balan@mathem.pub.ro}

\author{Alexander Kurz}
\address{Schmid College of Science and Technology, Chapman University, Orange, CA, USA}
\email{akurz@chapman.edu}

\author{Ji\v{r}\'{\i} Velebil}
\address{Department of Mathematics, Faculty of Electrical Engineering, 
Czech Technical University
         in Prague, Czech Republic}
\email{velebil@math.feld.cvut.cz}

\thanks{%
The first author acknowledges the support of University Politehnica of Bucharest, through the Excellence Research Grants Program, UPB--GEX, grant ID 252 UPB--EXCELEN\c{T}A--2016, research project no. 23/26.09.2016. 
\newline
The third author acknowledges the support by the grant  No.~P202/11/1632 of the Czech Science Foundation.} 

\keywords{generalised metric spaces, domain theory, denotational semantics, type theory, coalgebra,  enriched category theory, Kan extensions, density presentation}

\amsclass{18B35, 18D20}

\subjclass{}

\begin{abstract}

For a commutative quantale $\V$, the category $\Vcat$ can be perceived as a category of generalised metric spaces and non-expanding maps. 
We show that any type constructor $T$ (formalised as an endofunctor on sets) can be extended in a canonical way to a type constructor $T_\V$ on $\Vcat$. 
The proof yields methods of explicitly calculating the extension in concrete examples, which cover well-known notions such as the Pompeiu-Hausdorff metric as well as new ones. 

\medskip\noindent 
Conceptually, this allows us to to solve the same recursive domain equation $X\cong TX$ in different categories (such as sets and metric spaces) and we study how their solutions (that is, the final coalgebras) are related via change of base.

\medskip\noindent 
Mathematically, the heart of the matter is to show that, for any commutative quantale $\V$, the ``discrete'' functor $\Set\to \Vcat$ from sets to categories enriched over $\V$ is $\Vcat$-dense and has a density presentation that allows us to compute left-Kan extensions along $D$. 
\end{abstract}

\maketitle

\tableofcontents

\section{Introduction}\label{sec:intro}

Solving recursive domain equations 
$$X\cong TX$$  
is an important method to define mathematical models of computation in which recursive equations have unique solutions \cite{smyth-plotkin}.
The technique was introduced by Dana Scott \cite{scott:continuouslattices} in order to give a model of the untyped lambda calculus. 
Following a suggestion of Lawvere, Scott also emphasises that this solution arises as a limit of a chain constructed in a systematic fashion from a certain functor $T$ on a category of continuous lattices. 
In fact, Scott's solution is then the final $T$-coalgebra
$$X\to TX$$ 
constructed in the by now standard way called the final coalgebra sequence.
In domain theory \cite{abra-jung:dt}, this construction has been employed to give solutions to domain equations $X\cong TX$ for functors $T$ on sophisticated base categories such as metric spaces, measurable spaces and many more. Indeed,  there is a considerable variety of interesting categories of domains, each supporting different properties and type constructors. It may not always be obvious which category of domains is the most appropriate for a given modelling task, and the research into finding new good classes of domains still continues.

\medskip\noindent
While an unusual property in most areas of mathematics, in domain theory  final coalgebras often coincide with initial algebras. 
So final coalgebras as solutions of domain equations became more prominent only after Aczel in his work on non-well founded sets~\cite{aczel:book} showed that many important domain equations can be solved in the category of sets, a category in which initial algebras and final coalgebras do not coincide and in which the final coalgebras are the ones that provide the desired solutions to recursive equations.
Subsequent work by Rutten \cite{rutten:uc} and many others showed that streams, automata, probabilistic and other systems can be successfully described by functors on $\Set$, the category of sets and functions. 
Moreover, working with coalgebras over $\Set$ instead of over more complicated domains, allowed many powerful results to emerge, in particular on modal logics for coalgebraic systems and on formats of systems of equations that have unique solutions in final coalgebras.

\medskip\noindent
On the other hand, studying these techniques for richer base categories than $\Set$ has also been an important topic.
Coalgebras over categories of algebras often correspond to more powerful automata in which the states are equipped with extra memory. 
And coalgebras over partial orders and metric spaces allow us to capture simulation instead of bisimulation \cite{hugh-jaco:simulation} or metric bisimulations for quantitative systems \cite{rutten:cmcs98,worrell:cmcs00}. 
To summarise, while coalgebra successfully promoted simpler set-based models of computation, it also extended the class of interesting domains in which to solve domain equations.

\medskip\noindent
This increasing variety of categories of domains makes it necessary to find systematic ways to relate them. 
For example, in this paper, one of the questions we ask is the following. 
What does it mean to solve the same domain equation
$$X\cong TX$$
in different categories? 
On the face of it, this question does not make sense as the definition of a functor $T:\XX\to \XX$ depends on the base category $\XX$. 
So what does it mean to be the ``same'' functor on two different base categories $\XX$ and $\YY$?
We take our answer from Kelly's work on enriched category theory and consider two functors $T:\XX\to\XX$ and $\overline T:\YY\to\YY$ the same if there is a dense ``discrete'' functor $D:\XX\to\YY$ and $\overline T$ is the left-Kan extension of $DT$ along $D$
\begin{equation*}
\xymatrix{
\YY \ar[r]^{\overline T} & \YY\\
\XX \ar[r]_T \ar[u]^D& \XX\ar[u]_D
}
\end{equation*}
For example, we know already from previous work \cite{lmcs:bkv} that sets are dense in posets and that the convex powerset functor $\overline{\mathcal P}$ on $\Pos$, the category of partial orders and monotone maps, is the extension (posetification) of the powerset functor ${\mathcal P}$ on $\Set$. 
\begin{equation*}
\xymatrix{
\Pos \ar[r]^{\overline{\mathcal P}} & \Pos\\
\Set \ar[r]_{\mathcal P} \ar[u]^D& \Set\ar[u]_D
}
\end{equation*}
As we will show in Section~\ref{sec:beh}, this implies that the final coalgebra solution of the domain equation $Y\cong \overline{\mathcal P} Y$ is the same, in a suitable sense, as the the one of $X\cong \mathcal PX$.

\medskip\noindent
Another question that we study is how to define type constructors on ``richer'' categories from type constructors on ``easier'' categories. 
For example, it is well-known that a power domain for simulation is given by the ``up-set functor'' $\mathcal U:\Pos\to\Pos$, which is the left Kan extension of $\mathcal P_u$ along $D$
\begin{equation*}
\xymatrix{
\Pos \ar[r]^{{\mathcal U}} & \Pos\\
\Set \ar[ru]_{\mathcal P_u} \ar[u]^D& 
}
\end{equation*}
where $\mathcal P_u(X)$ is again the powerset of $X$, but now \emph{ordered by inclusion}. 
Section~\ref{sec:beh} shows that the solution of a ``richer'' equation such as $Y\cong \mathcal U Y$ classifies the same notion of bisimulation as the solution of the corresponding ``easier'' equation, here $X\cong \mathcal P X$, but also carries an ordered structure that classifies simulation.

\medskip\noindent
The purpose then of this paper is to extend the observations above from posets to metric spaces and, more generally, to the category $\Vcat$ of categories enriched over a commutative quantale $\V$. The essence is to show that the discrete functor 
$$D:\Set\to\Vcat$$
is dense. Importantly, the proof will exhibit a formula that allows us to compute the left Kan extensions along $D$ in many concrete examples. 

\medskip\noindent
What is modelled by $\V$ and what role does $\Vcat$ play? 

\medskip\noindent
To say that $\V$ is a commutative quantale is to say that $\V$ is a lattice of ``truth values'' or ``distances''. 
The leading example is the lattice of real numbers $[0,\infty]$ which is promoted from a lattice to a quantale when equipped with the extra structure of addition of distances, important to capture the triangle inequality.

\medskip\noindent
Objects of $\Vcat$ are then ``categories enriched over $\V$'', that is, sets $\X$ equipped with a distance $\X(x,y)\in\V$. To say that $\X$ is a ``category enriched over $[0,\infty]$ with $+$'' is to say that $\X$ satisfies the triangle inequality
$$ \X(x,y)+\X(y,z) \ge_\mathbb R \X(x,z).$$
Thus, one reason to work with enriched categories is that many structures of interest such as posets and metric spaces appear themselves as categories.
Another one, emphasised in \cite{smyth-plotkin}, is the importance of  homsets carrying ordered/metric structure.
A third reason, important for this paper, will be discussed now in more detail, namely that enriched categories make available a richer variety of colimits.

\medskip\noindent
That $D:\Set\to\Pos$ is dense implies that every poset is a colimit of sets. But how can we construct ordered sets by taking quotients of discrete sets? The crucial point must be to work with a more general notion of quotient, namely with quotients that do not merely add equations between elements (coequalisers), but with quotients that add inequations. 

\medskip\noindent
Such a richer notion of quotient is automatically provided to us in $\Pos$ or $\Vcat$ as a special case of a so-called weighted (or indexed) colimit, a notion native to the theory of enriched categories.
In this paper, we only need some special cases of weighted colimits that are reviewed in detail in Section~\ref{sec:preliminaries}. 

\medskip\noindent
In the case of posets, the basic idea is easily explained. 
Every poset is a weighted colimit of discrete posets. 
How? Let $X$ be a poset and  write $X_0$ for its set of elements and $X_1\subseteq X_0\times X_0$ for its order. 
Then $X$ is the ``ordered coequalizer'', or rather ``coinserter'', as it is known in the literature
\begin{equation}
\label{eq:poset_coins1}
\xymatrix@C=35pt{
X_1 
\ar@<0.5ex>[r]^{} 
\ar@<-0.5ex>[r]_{} 
& 
X_0 \ar[r]^-{} & X}.
\end{equation} 
It is important here that the coinserter turns the discrete posets $X_0$, $X_1$ into a genuine poset $X$. 
This also makes it clear how to compute the left Kan extension $\overline T$ of $T$ by applying $T$ to the parallel pair of \eqref{eq:poset_coins1} and then taking the coinserter in $\Pos$
\begin{equation}
\label{eq:poset_coins2}
\xymatrix@C=35pt{
T X_1 
\ar@<0.5ex>[r]^{} 
\ar@<-0.5ex>[r]_{} 
& 
T X_0 \ar[r]^-{} &\overline T X}
\end{equation} 
This formula, describing the left-Kan extension as a certain weighted colimit of discrete $\V$-categories, allows us to compute such Kan extensions in many concrete situations. For example, in \cite{lmcs:bkv} we showed that it follows easily from \eqref{eq:poset_coins2} that the extension of $\mathcal P$ from $\Set$ to $\Pos$ is the convex powerset functor. 

\medskip\noindent
A conceptual way of summarising this construction is to say that the coinserters of \eqref{eq:poset_coins1} form a ``density presentation'' of $D:\Set\to\Pos$. But density presentations work best when $D$ is fully faithful, which coincides with the quantal being integral, whereas the generalisation of the formula \eqref{eq:poset_coins2} also applies in case $D$ is not fully faithful.

\paragraph{\bf Summary of main results.} 

\begin{enumerate}

\item We prove in Theorem~\ref{thm:Lan} that \eqref{eq:poset_coins2} can be generalised to compute left Kan extensions along $D:\Set\to\Vcat$ for any commutative quantale $\V$. As a corollary, we obtain that $D:\Set\to\Vcat$ is dense. If $D$ is fully faithful, we obtain a density presentation in Theorem~\ref{thm:dense-pres}.  Those functors that arise as left Kan extensions along $D$ are characterised by preserving certain colimits in Theorems \ref{thm:char-discrete-arities} and \ref{thm:char-Vcatification}. 

\item  It is well known that the fully-faithfulness of $D$ is equivalent to the unit $H\to (\Lan{D}H)D$ of the left Kan extension of $H$ along $D$ being an isomorphism for {\em all} $H$. We show in Proposition~\ref{prop:D-ff} that $D$ is fully faithful if and only if the quantale $\V$ is integral, which is the case in all the examples we pursue. We also characterise, in the general case, those functors $H$ for which the unit is an isomorphism in Theorem~\ref{thm:unitiso}.

\item We use that $\V$-categories are relational presheaves to show in Theorem~\ref{thm:wpbcontrelpresh} how the extension of weak pullback preserving $\Set$-functors can be computed from their relation lifting. 
This is often the easiest route to concrete computations, the  Pompeiu-Hausdorff metric being Example~\ref{ex:Vcatif-powerset}.   

\item Theorems \ref{thm:beh-vcatification}  and \ref{thm:beh-discrete-arities} formalise the intuition that 
the final coalgebra over $\Vcat$ of 
a ``functor $T:\Set\to\Set$  equipped with a $\V$-metric'' is the final $T$-coalgebra over $\Set$ equipped with an appropriate metric. 
In particular, both final coalgebras determine the same equivalence relation (bisimulation), but the final coalgebra over $\Vcat$ gives refined metric information about states that are not bisimilar in the sense of measuring how far such states are from being equal. 
\end{enumerate}

\noindent
The formulas for computing left Kan extensions are \eqref{eq:shortest-path} in the general case and \eqref{eq:wpb} for weak pullback preserving functors. Examples of left Kan extensions are in Sections~\ref{sec:examples}--\ref{sec:Vmetric-functors}. Closure properties of $\Vcat$-ifications are studied in Section~\ref{sec:closure} and 2-categorical properties of the category of coalgebras over $\Vcat$ in Section~\ref{sec:2-cat}.

\medskip\noindent The present paper is a revised and extended version of \cite{bkv:V-cat}.


\section{Preliminaries}
\label{sec:preliminaries}

In this section we gather all the necessary technicalities
and notation from category theory enriched in a complete and 
cocomplete symmetric monoidal closed category that we will use later. 
For the standard notions of  enriched categories, enriched functors and enriched natural transformations we refer to Kelly's book~\cite{kelly:book} and for their importance to logical methods in computer science see for example~\cite{lawvere:gen_metric_spaces,turi-rutten,wagner,worrell:cmcs00,worrellPhD}.
Readers familiar with enriched category theory are invited to skip these preliminaries and pass directly to Section~\ref{sec:extension} after taking note of Section~\ref{sec:discrete functor} on the fully-faithfulness of the discrete functor $D:\Set \to \Vcat$.

\medskip\noindent
 We will mainly use two prominent enrichments: that in a commutative quantale $\V$ and that in the category $\Vcat$ of {\em small\/} $\V$-categories and $\V$-functors for a commutative quantale $\V$. 
We spell out in more detail how the relevant notions look like and carefully write all the enrichment-prefixes. 
In particular, the underlying category of an enriched category will be denoted by the same symbol, followed by the subscript ``$o$'' as usual. 


\subsection{$\V$-categories as generalised metric spaces}

\ 

\bigskip\noindent
Suppose $\V=(\V_o,\tensor,e,[{-},{-}])$ is a {\em commutative quantale\/}.
More in detail: $\V_o$ is a complete lattice, equipped with
a commutative and associative monotone binary operation 
$\tensor$, called the {\em tensor\/}. 
We require the element $e$ to be a {\em unit\/} of tensor. 
Furthermore, we require every monotone map ${-}\tensor r:\V_o\to\V_o$ to have a right adjoint $[r,{-}]:\V_o\to\V_o$. We call $[{-},{-}]$ the {\em internal hom\/} of $\V_o$.
We will suppose that $\V_o$ is non-trivial, that is, $\V_o$ is not a singleton (or equivalently, $e \neq \bot$). Recall that the quantale is said to be {\em integral} if $e = \top$.\footnote{%
A commutative integral quantale is sometimes called a {\em complete residuated lattice.}
}

\medskip\noindent
Commutative quantales are complete and cocomplete symmetric monoidal closed categories.
Therefore, one can define $\V$-categories, $\V$-functors, and $\V$-natural transformations. 
Before we say what these are, let us mention several examples of commutative quantales.

\begin{exas}
\label{ex:quantale}
\hfill
\begin{enumerate}
\item \label{ex:quantale_two}
The two-element chain $\Two=\{0,1\}$ with the usual order and tensor $r\tensor s=r\wedge s$. The internal hom is implication.

\item \label{ex:quantale_infinity} 
The real half line $([0,\infty], \geq_\mathbb R)$, with (extended) addition as tensor product. The internal hom $[r,s]$ is truncated minus, that is,  $[r,s]=\textsf{ if } r\ge_\mathbb R s\textsf{ then } 0 \textsf{ else } s-r$.

\item \label{ex:quantale_unit}
The unit interval $([0,1], \geq_\mathbb R)$ with tensor product $r\tensor s=\max(r,s)$. 
The internal hom is given by $[r,s]=\textsf{ if }r\ge_\mathbb R s\textsf{ then } 0 \textsf{ else } s$.

\item \label{ex:quantale_prob} 
The poset of all monotone functions $f:[0,\infty]\to [0,1]$ such that the equality $f(r)=\bigvee_{s < r} f(s)$ holds, with the pointwise order. 
It becomes a commutative quantale with the tensor product
$$
(f \tensor g) (r) = \bigvee_{s + s' \leq r} f(s)\cdot g(s')
$$
having as unit the function mapping all nonzero elements to $1$, and $0$ to itself~\cite{hofmann+reis}. 
\item \label{ex:quantale_three}
Let $\underline n=\{0,\ldots n-1\}$ be an ordinal equipped with a monotone, commutative, idempotent operation satisfying $0\otimes i=0$ for all $i\in\underline n$. 
Then there are unique $e$ and $[-,-]$ such that $(\underline n, \otimes, e, [-,-])$ is a quantale \cite[Prop 2.5]{Casley-phd}. 
In the case $n=2$ we just obtain Item~\eqref{ex:quantale_two} above. 
In case $n=3$ there are exactly two ways of turning $\underline 3$ into a commutative quantale with an idempotent tensor, determined by choosing either $1\tensor 2=2$ or $1\tensor 2=1$~\cite[Cor 2.6]{Casley-phd}. %
We briefly explain the associated $\V$-categories in Example~\ref{exle:casley}. 
The three-element quantale $\underline 3$ which is determined by $e=1$ is the smallest non-integral quantale, in the sense that it embeds into any other non-integral $\V$.%

\item
\label{ex:free-quant} Let $(\mathsf M, \cdot, e)$ be a monoid. 
Then $\mathcal P(\mathsf M)$, the powerset of $\mathsf M$, becomes a quantale ordered by inclusion with tensor $S \otimes S' = \{x \cdot x'\mid x\in S, x'\in S'\}$ and unit $\{e\}$. 
In fact, $\mathcal P(\mathsf M)$ is the free quantale over the monoid $\mathsf M$~\cite{rosenthal} and  $\mathcal P(\mathsf M)$ is commutative if $\mathsf M$ is commutative. 
If $\mathsf M$ is the free monoid over an alphabet of size at least two, then $\mathcal P(\mathsf M)$ is the non-commutative non-integral quantale of languages over the alphabet.
\end{enumerate}
\end{exas}

\begin{rem}
If we think of the elements of the quantale as truth values, then $\le$ is external implication, $e$ is true, $\bot$ is false, $\tensor$ is a conjunction and $[-,-]$ is internal implication. 
Some standard logical laws such as  $[e,r]=r$ and $r \tensor [r,s]\le s$ and $r\le[s,s\tensor r]$ and $s\le r \ \Leftrightarrow \ e\le [s,r]$ hold in all quantales.
If we consider the elements of the quantale as representing distances, the order of the quantale is opposite to the natural order of distances, with $0$ as top and $\infty$ as bottom. 
To reconcile the two points of view, we can think of distance as ``distance from truth'', so that a distance of 0 corresponds to true and a distance of $\infty$ to false.
\qed
\end{rem}

\medskip\noindent
A ({\em small\/}) {\em $\V$-category\/} $\X$ consists of a (small) set of objects, together with an object $\X(x',x)$ in $\V_o$ for each pair $x'$, $x$ of objects, subject to the following axioms
$$
e\leq\X(x,x),
\quad
\X(x',x)\tensor\X(x'',x')\leq\X(x'',x)
$$
for all objects $x''$, $x'$ and $x$ in $\X$. 

\noindent A $\V$-category $\X$ is called \emph{discrete} if 
\begin{equation}\label{eq:discrete V-category}
\X(x',x) = \begin{cases} 
\, e \ \ , \ \ x'=x
\\
\bot \ \ , \ \ \mbox{otherwise}
\end{cases}
\end{equation}

\medskip\noindent
A {\em $\V$-functor\/} $f:\X\to\Y$ is given by the object-assignment $x\mapsto fx$, such that
$$
\X(x',x)\leq\Y(fx',fx)
$$
holds for all $x'$, $x$.

\medskip\noindent
A {\em $\V$-natural transformation\/} $f\to g$ is given
whenever
$$
e\leq\Y(fx,gx)
$$ 
holds for all $x$. Thus, there is at most one $\V$-natural transformation
between $f$ and $g$.

\medskip\noindent
$\V$-categories and $\V$-functors form a category which we denote by $\Vcat_o$ (actually, a 2-category having $\V$-natural transformations as 2-cells).

\begin{exa}\label{exle:V-categories}
\hfill
\begin{enumerate}
\item \label{exle:Twocat}
The two-element chain $\Two$ is a commutative quantale. 
A small $\Two$-category\footnote{ Not to be confounded with the notion of a 2-category, that is, a $\mathsf{Cat}$-enriched category.} $\X$ is  precisely a {\em preorder\/}, where $x'\leq x$ iff $\X(x',x)=1$, while a $\Two$-functor $f:\X \to \Y$ is a monotone map. 
A $\Two$-natural transformation $f\rightarrow g$ expresses that $fx\leq gx$ holds for every $x$. Thus $\Two\mbox{-}\mathsf{cat}$ is  the category  $\Pre$ of  preorders and monotone maps. 
It plays an important role not only because we study $\Vcat$ by generalising from $\Two$-cat, but also because the definition
\begin{equation}
\label{eq:V-metric=>order}
x\le_\X y \ \Longleftrightarrow \ e\le \X(x,y).
\end{equation}
equips every $\V$-category  $\X$ with the structure of a preorder $\le_\X$. 
In case $\X$ is $\V$, we have ${\le_\X}={\le}$. In the following, we will drop the subscript in $\le_\X$.
\medskip
\item \label{ex:V as V-category itself}
$\V$ is itself a $\V$-category with $\V(r,s)=[r,s]$. This is of fundamental importance, at least because the internal hom $\V(r,s)$ usually has a richer structure than the external hom $\V_o(r,s)$. For example in case $\V=[0,\infty]$, the external hom is only two-valued whereas the internal hom is $[0,\infty]$-valued. 
\end{enumerate}
\end{exa}

\noindent A good intuition is that $\V$-categories are possibly non-symmetric metric spaces and $\V$-functors are non-expanding maps. 
This intuition goes back to Lawvere~\cite{lawvere:gen_metric_spaces}. 
We show next some examples that explain this intuition. 
A good resource is also~\cite{rutten:ultrametric_spaces}. 

\begin{exas}\label{exs:Vcats=metric_spaces}
\begin{enumerate}
\item \label{ex:V=[0,infty]}
Let $\V$ be the real half line $([0,\infty],\geq_\mathbb R, +, 0)$ as in Example~\ref{ex:quantale}\eqref{ex:quantale_infinity}.
A small $\V$-category can be identified with a set $X$ equipped with a mapping $d_X:X\times X\to [0,\infty]$ such that $(X,d_X)$ is a {\em generalised metric space\/}. 
The generalisation of the usual notion is three-fold. 
First, $d_X$ is a pseudo-metric in the sense that two distinct points may have distance $0$. 
Second, $d_X$ is a quasi-metric in the sense that distance is not necessarily symmetric. 
Third, distances are allowed to be infinite, which has the important consequence that the category of generalised metric spaces has colimits (whereas metric spaces do not even have coproducts).
  
\noindent A $\V$-functor $f:(X,d_X)\to (Y,d_Y)$ is then a exactly a {\em non-expanding mapping\/}, that is, one satisfying
the inequality $d_Y(fx',fx)\leq_\mathbb R d_X(x',x)$ for every 
$x,x'\in X$.

\noindent The existence of a $\V$-natural transformation $f\to g$ means that $\bigvee_x d_Y(fx,gx)=0$, that is, the distance 
$d_Y(fx,gx)$ is $0$, for every $x\in X$, or also that $fx\le gx$ with respect to the order discussed in Example~\ref{exle:V-categories}\eqref{exle:Twocat}.  
\item 
For the unit interval 
$\V=([0,1], \geq_\mathbb R, \max, 0)$
from Example~\ref{ex:quantale}\eqref{ex:quantale_unit}, 
a $\V$-category is a {\em generalised ultrametric space\/} 
$( X,d_X:X \times X \to [0,1])$~\cite{rutten:ultrametric_spaces,worrell:cmcs00}. 
Again, the slight generalisation of the usual notion lies in the fact that 
the distance function $d_X$ is not necessarily symmetric and 
$d_X(x',x)=0$ does not necessarily entail $x=x'$. Similarly, $\V$-functors are precisely the non-expanding maps, and the existence of a $\V$-natural transformation $f\to g:( X,d_X) \to ( Y, d_Y)$ 
means, again, that $\bigvee _x d_Y (fx,gx)=0$, that is, the distance 
$d_Y(fx,gx)$ is $0$, for every $x\in X$. 
\item 
Using the quantale $\V$ from 
Example~\ref{ex:quantale}\eqref{ex:quantale_prob} 
leads to
{\em probabilistic metric spaces\/}: for a $\V$-category $\X$, and for every
pair $x$, $x'$ of objects of $\X$, the hom-object is a function %
$\X(x',x):[0,\infty]\to [0,1]$ %
with the intuitive meaning $\X(x',x)(r)=s$ holds
iff $s$ is the probability that the distance from $x'$ to $x$
is smaller than $r$. See~\cite{flag-kopp:continuity-spaces,hofmann+reis}.

\item The categories enriched in the free quantale $\mathcal P(\mathsf M)$ of Example~\ref{ex:quantale}\eqref{ex:free-quant} can be perceived as $\mathsf M$-labeled automata~\cite{betti,rosenthal:paper}. 
The ``distance'' between two states is the language connecting them.
In particular, if $\mathsf M$ is the free monoid over an alphabet, the enriched hom $\X(x,y)$ between two states $x,y$ of an automaton $\X$ is the language accepted by the automaton $\X$, considered with initial state $x$ and final state $y$. However, this is an example outside the scope of the paper, since in this case $\V=\mathcal P(\mathsf M)$ is not commutative.
 
\end{enumerate}
\end{exas}
     

\subsection{$\V$-categories as relational presheaves}
\label{sec:relpresh}

\ 

\bigskip\noindent
A $\V$-category $\X$ cannot only be considered as generalised metric space, but also as a set $X$ equipped with a collection $(X_r)_{r\in\V_o}$ of binary relations given by 
\begin{equation}\label{eq:Xr}
X_r = \{(x',x)\in X \times X \mid r\le \X(x',x)\}.
\end{equation}
In the case of the generalised metric spaces discussed above, the collection%
\footnote{ $X_\infty, X_0$ are redundant since $\infty$ is bottom and $X_0=\bigcap\{X_r\mid r>_{\mathbb R} 0\}$.}
$(X_r)_{0<r<\infty}$  can be considered as a basis for the quasi-uniformity \cite{fletcher} associated with $\X$. 
But interest in the collection $(X_r)_{r\in\V}$ can also arise from considerations independent of the view of $\V$-categories as metric spaces as the following example from concurrency theory demonstrates.

\begin{exa}\label{exle:casley}
A category $\X$ enriched over the quantale $\underline n$ of Example~\ref{ex:quantale}\eqref{ex:quantale_three} can be understood as a set equipped with $n-1$ transitive relations 
$$X_{n-1}\subseteq \ldots \subseteq X_i\subseteq \ldots \subseteq X_1.$$ 
given by $X_i=\{(x',x)\mid i\le \X(x',x) \}$, $0<i<n$. Following \cite{Casley-phd,Casley-etal}, the $\V$-enriched categories for the three-element quantales $\V=(\underline 3, \otimes, e, [-,-])$  
of Example~\ref{ex:quantale}\eqref{ex:quantale_three} can be interpreted as well-known models of concurrency given by sets equipped with two relations (in the terminology of the $(X_r)_{r\in\V_o}$ above, the relation $X_0$ is redundant since $0$ is bottom). 

\medskip\noindent
The first case, which is determined by $e=1$, accounts for the {\em prossets} of Gaifman and Pratt~\cite{gaifman-pratt}.
Explicitly, the objects of a $\V$-category can be seen as events subject to a schedule, endowed with a preorder ${x\leq y}$ given by $X_1$ (with the interpretation that ``$x$ happens no later than $y$'') and a binary relation $x\prec y$ given by $X_2$ (which is intended to mean ``$x$ happens strictly earlier than $y$''). 
Then $X_2\subseteq X_1$ says that strict precedence implies weak precedence, while the multiplication law $1\otimes 2=2$ reflects the prosset-law that $x\le y \prec z \le w$ implies $x\prec w$.%
\footnote{See also~\cite[Definition I-1.11]{ghklms:ContLatDom-book}, where $\prec$ is called an auxiliary relation.} %

\medskip\noindent
The second case, which is determined by $e=2$, is due to Gaifman~\cite{gaifman}. 
The relation $x<^ty$ given by $X_1$ is interpreted as ``$x$ precedes $y$ in time'' and the relation $x<^c y$ given by $X_2$ is interpreted as ``$x$ causally precedes $y$''. 
$X_2\subseteq X_1$ captures that causal precedence implies temporal precedence and the multiplication law $1\otimes 2=1$ reflects that $x<^t y <^c z <^t w$ implies $x<^t w$.
\qed
\end{exa}

\medskip\noindent
The idea illustrated in the two examples above can be formalised as a relational presheaf \cite[Chapter~3.4]{rosenthal}, that is, as a (lax) monoidal functor
\begin{align*}
X_\_ : \V & \to\Rel\\
r \, & \  \mapsto \ X_r 
\end{align*}
which satisfies
\begin{gather}
\label{eq:monoidal1}
\Id \subseteq X_e\\
\label{eq:monoidal2}
X_r \cdot X_s\subseteq X_{r\otimes s} 
\end{gather}
Moreover such a functor comes from a $\V$-category iff  
$$ 
X_{\bigvee_{i\in I} r_i} = \bigcap_{i\in I} X_{r_i}
$$
that is, the presheaf is continuous. 
We present this result in some detail, because it is related to the $\V$-nerves of Definition~\ref{def:Vcoinserter} and because it prepares the grounds for Theorem~\ref{thm:wpbcontrelpresh}.

\medskip\noindent
We will follow Chapter~3.4 of~\cite{rosenthal}, specialised  
to an one-object quantaloid, that is, to a  commutative
quantale $\V=(\V_o,\tensor,[-,-],e)$.

\medskip\noindent
Let $\Sigma\V$ denote the {\em suspension\/} of $\V$. 
That is, $\Sigma\V$ is an {\em ordered\/} category on one object $*$, with $\Sigma\V(*,*)=\V_o$. 
The composition in $\Sigma\V$ is given by $\tensor$ and the unit $e$ serves as an identity.
Further, let $\Rel$ be the 2-category of {\em ordinary\/} relations.

\medskip\noindent
The {\em ordinary\/} category $\RelPresh$ of {\em relational presheaves\/} and their morphisms is defined as follows:
\begin{enumerate}
\item 
A relational presheaf $\frak{X}:(\Sigma\V)^\coop\to\Rel$ is a {\em lax\/} 2-functor. 
That is, if we denote $X=\frak{X}(*)$, the inclusions
$$
\Id_X\subseteq \frak{X}(e),
\quad
\quad
\frak{X}(r)\cdot \frak{X}(s)\subseteq \frak{X}(r\tensor s)
$$
hold in $\Rel(X,X)$, with ``$\cdot$'' denoting relational composition. 
Moreover, the inclusion 
\begin{equation}\label{eq:monot-rel-presheaf}
\frak{X}(r)\subseteq \frak{X}(s)
\end{equation}
holds whenever $r\geq s$ holds in $\V$.

\item
A relational morphism from $\frak{X}$ to $\frak{Y}$ is  a lax natural
transformation $\frak f$ from $\frak{X}$ to $\frak{Y}$ that has maps as components. 
That is, if we denote $X=\frak{X}(*)$, $Y=\frak{Y}(*)$ and write $f$ for $\frak{f}_{*}$, then for every $r$ the inclusion
\begin{equation}\label{eq:rel-morph}
\vcenter{
\xymatrix{X
\ar[d]|-{\object @{/}}_{\frak{X}(r)}
\ar[r]|-{\object @{/}}^{f_\diamond}
\ar@{}[dr]|{\subseteq}
&
Y
\ar[d]|-{\object @{/}}^{\frak{Y}(r)}
\\
X
\ar[r]|-{\object @{/}}_{f_\diamond}
&
Y
}}
\end{equation}
holds. 
Above, $f_\diamond$ is the graph relation of $f:X\to Y$.
\end{enumerate}

\begin{exa}
\label{ex:cat-as-presheaf}
Every small $\V$-category $\X$ can be turned into a relational
presheaf 
$$
\Phi(\X):(\Sigma\V)^\coop\to\Rel
$$
as follows:
\begin{enumerate}
\item 
Let $\Phi(\X)(*)=X_\extra$, where we write $X_\extra$ for the set of all objects of $\X$.\footnote{ In the introduction, in the special case of preorders, we used the more familiar notation $X_0$ instead of $X_\extra$, but in the general case of $\V$-categories we now have examples where 0 is an element of $\V$ and the notation $X_0$ has a meaning given by \eqref{eq:Xr}.}
\item
Put $\Phi(\X)(r)=\{ (x',x)\mid r\leq \X(x',x)\}$ for every $r$.
Then the relations
$$
\Id_{X_\extra}\subseteq \Phi(\X)(e),
\quad
\quad
\Phi(\X)(r)\cdot \Phi(\X)(s)
\subseteq 
\Phi(\X)(r\tensor s)
$$
hold in $\Rel(X_\extra,X_\extra)$ precisely because $\X$ is a $\V$-category.
\end{enumerate}
Observe that the presheaf $\Phi(\X)$ satisfies an additional condition: the equality
$$
\Phi(\X)(\bigvee_{i\in I} r_i)
=
\bigcap_{i\in I} \Phi(\X)(r_i)
$$
holds for any family $\{r_i\mid i\in I\}$ of elements of $\V$, since
$$
\{ (x',x)\mid \bigvee_{i\in I} r_i\leq \X(x',x)\}
=
\bigcap_{i\in I} 
\{ (x',x)\mid r_i\leq \X(x',x)\}
$$
It is easy to see that any $\V$-functor $f:\X\to\Y$ yields a morphism $\Phi(f):\Phi(\X) \to \Phi(\Y)$ of relational presheaves. 
Indeed, denote by $f_\extra:X_\extra\to Y_\extra$ the object-assignment of $f$. 
Then the inclusion
$$
(f_\extra)_\diamond\cdot \Phi(\X)(r)
\subseteq
\Phi(\Y)(r)\cdot (f_\extra)_\diamond
$$ 
means the following
\[
\mbox{if $r\leq\X(x',x)$, then $r\leq\Y(fx',fx)$,
for all $x'$, $x$ in $X_\extra$}
\]
But this holds precisely since $f$ is a $\V$-functor.
\end{exa}

\begin{defiC}[{\cite[Definition~3.4.1]{rosenthal}}]
A relational presheaf $\frak{X}$ is called {\em continuous\/}, if
$$
\frak{X}(\bigvee_{i\in I} r_i)
\supseteq
\bigcap_{i\in I} \frak{X}(r_i)
$$
holds for any family $\{r_i\mid i\in I\}$ of elements of $\V$.
The full subcategory of $\RelPresh$ spanned by continuous relational presheaves is denoted by $\RelPresh_c$.
\end{defiC}

\begin{rem}
The inclusion
$$
\frak{X}(\bigvee_{i\in I} r_i)
\subseteq
\bigcap_{i\in I} \frak{X}(r_i)
$$
holds for any family $\{r_i\mid i\in I\}$ of elements of $\V$ and {\em any\/} relational presheaf $\frak{X}$, because \eqref{eq:monot-rel-presheaf} implies that 
$
\frak{X}(\bigvee_{i\in I} r_i)\subseteq \frak{X}(r_i)
$
holds for any $i$.
\end{rem}

\begin{propC}[{\cite[Proposition~3.4.1]{rosenthal}}]\label{prop:relpresh}
The assignment $\X\mapsto\Phi(\X)$ extends to an {\em ordinary\/} functor $\Phi:(\Vcat)_o\to\RelPresh_c$ which is an equivalence of categories. 
\end{propC}

\begin{proof}
It is easy to see that the processes $\X\mapsto\Phi(\X)$ and $H\mapsto\Phi(H)$ of Example~\ref{ex:cat-as-presheaf} extend to a functor $\Phi:(\Vcat)_o\to\RelPresh_c$. 
Its pseudoinverse $\Psi:\RelPresh_c\to (\Vcat)_o$ sends a continuous relational presheaf $\frak{X}:(\Sigma\V)^\coop\to\Rel$ into the following $\V$-category $\Psi(\frak{X})$:
\begin{enumerate}
\item 
The set of objects of $\Psi(\frak{X})$ is the set $\frak{X}(*)$. 
\item
For every pair $x'$, $x$ in $\frak{X}(*)$ we put
$$
\Psi(\frak{X})(x',x)
=
\bigvee\{r\mid (x',x)\in \frak{X}(r) \}.
$$
Then $\Psi(\frak{X})$ is a $\V$-category for the following reasons:
\begin{enumerate}
\item The inequality
$$
e
\leq
\bigvee\{r\mid (x,x)\in \frak{X}(r) \}
=
\Psi(\frak{X})(x,x)
$$
holds since $\Id\subseteq \frak{X}(e)$ holds.

\item The inequality 
$$
\Psi(\frak{X})(x'',x')\tensor\Psi(\frak{X})(x',x)
\leq
\Psi(\frak{X})(x'',x)
$$
holds, since
\begin{eqnarray*}
\Psi(\frak{X})(x'',x')\tensor\Psi(\frak{X})(x',x)
&=&
\bigvee\{r\mid (x'',x')\in \frak{X}(r) \}
\tensor
\bigvee\{s\mid (x',x)\in \frak{X}(s) \}
\\
&=&
\bigvee_{(x'',x')\in \frak{X}(r) } \
\bigvee_{(x',x)\in \frak{X}(s) }
r\tensor s
\\
&\leq&
\bigvee_{(x'',x)\in \frak{X}(r\tensor s) }
r\tensor s
\\
&=&
\Psi(\frak{X})(x'',x)
\end{eqnarray*}
Above, we have used that $\frak{X}(r)\cdot \frak{X}(s)\subseteq \frak{X}(r\tensor s)$.
\end{enumerate}
\end{enumerate} 
That $\Phi$ and $\Psi$ are essentially inverse to each other is verified as follows:
\begin{enumerate}
\item
Suppose a $\V$-category $\X$ is given. Then the $\V$-category  $\Psi\Phi(\X)$ has the same set of objects as $\X$. 
Moreover,
$$
\Psi\Phi(\X)(x',x)
=
\bigvee\{r\mid r\leq\X(x',x)\}
=
\X(x',x)
$$
\item
Suppose $\frak{X}:(\Sigma\V)^\coop\to\Rel$ is a relational presheaf.
Then the presheaf $\Phi\Psi(\frak{X})$ has the value $\frak{X}(*)$ at $*$ by definition of $\Psi$ and $\Phi$.

\medskip\noindent
We need to prove that $(x',x) \in \frak{X}(r)$ holds iff
$r\leq \bigvee\{ s\mid (x',x)\in \frak{X}(s) \}$ holds.
The implication from left to right is trivial. For the 
converse implication, apply $\frak{X}$ to the inequality
$r\leq \bigvee\{ s\mid (x',x) \in \frak{X}(s) \}$ and use 
continuity of $\frak{X}$:
$$
\frak{X}(r)
\supseteq 
\frak{X}\left(\bigvee\{ s\mid (x',x) \in \frak{X}(s) \}\right)
\supseteq
\bigcap\{ \frak{X}(s)\mid (x',x) \in \frak{X}(s) \}
\ni
(x',x)
$$ 
\end{enumerate}
This finishes the proof.
\end{proof}

\begin{rem}\label{rem:enrichedRelPresh}
The functor $\Psi$ can be extended to an equivalence of {\em $\Vcat$-categories\/} (see the next section for an introductory exposition on $\Vcat$-enriched categories).
This means that every
$$
\RelPresh_c(\frak{X},\frak{Y})
$$
should have the structure of a $\V$-category. 
We do it by transferring this structure from $\Vcat$, using the equivalence of $(\Vcat)_o$ and $\RelPresh_c$.
More precisely, write $\frak{X}=\Phi(\X)$ and $\frak{Y}=\Phi(\Y)$ for some
unique $\V$-categories $\X$ and $\Y$ and put
$$
\RelPresh_c(\frak{X},\frak{Y})
=
[\X,\Y].
$$
\end{rem}


\subsection{Categories enriched in $\Vcat$}\label{sec:V-cat-cat}

\

\bigskip\noindent
Suppose that $\V=(\V_o,\tensor,e,[{-},{-}])$ is a commutative quantale and recall that we denoted by $\Vcat_o$ the (ordinary) category of all small $\V$-categories and
all $\V$-functors between them. The category $\Vcat_o$ has a monoidal closed
structure: The {\em tensor product\/}
is inherited from $\V_o$. Namely, for $\V$-categories $\X, \Y$, put $\X\tensor \Y$ to be the $\V$-category having as objects the corresponding pairs of objects and $\V$-homs given by 
$$
(\X\tensor \Y)((x',y'),(x,y))=
\X(x',x)\tensor \Y(y',y)
$$
The {\em unit\/} for the tensor product is the $\V$-category 
$\One$, with one object $0$ and corresponding $\V$-hom given by $\One(0,0)=e$. 

\medskip\noindent
The $\V$-functor ${-}\tensor \Y:\Vcat_o\to\Vcat_o$ has a right adjoint $[\Y,{-}]$. Explicitly, $[\Y,\Z]$ is the following $\V$-category:
\begin{enumerate}
\item 
Objects of $[\Y,\Z]$ are $\V$-functors from $\Y$ to $\Z$.
\item
The $\V$-``distance'' between two $\V$-functors $f,g:\Y \to \Z$ is 
\begin{equation}\label{eq:Vcat internal hom}
[\Y,\Z](f,g) = \bigwedge_y \Z(fy,gy)
\end{equation}
\end{enumerate}
\noindent We will sometimes write $\Vcat(\Y,\Z)$ or even $\Z^\Y$ instead of $[\Y,\Z]$. 

\medskip\noindent
By~\cite{kelly+lack:locally-bounded}, the symmetric monoidal closed category $\Vcat=(\Vcat_o,\tensor,\One,[{-},{-}])$ is complete and cocomplete, with generators consisting of $\V$-categories of the form $\Two_r$, $r\in \V_o$. 
Here, every $\Two_r$ has two objects $0$ and $1$, with $\V$-homs 
\begin{equation}\label{eq:Two_r}
\Two_r(0,0)=\Two_r(1,1)=e \, , \, \Two_r(0,1)=r \, , \, \Two_r(1,0)=\bot
\end{equation}

\medskip\noindent
Since $\Vcat$ is a symmetric monoidal closed category, we can define $\Vcat$-enriched categories, $\Vcat$-functors and $\Vcat$-natural transformations.

\medskip\noindent
 A ({\em small\/}) {\em $\Vcat$-category\/} $\XX$ consists of a (small) set of objects $X$, $Y$, $Z$, \dots, a small $\V$-category $\XX(X,Y)$ for every pair $X$, $Y$ of objects, and $\V$-functors
$$
u_X:\One\to \XX(X,X),
\quad
c_{X,Y,Z}:\XX(Y,Z)\tensor\XX(X,Y)\to\XX(X,Z)
$$
that represent the identity and composition and satisfy the usual axioms~\cite{kelly:book}:
$$
\xymatrix@R=30pt@C=40pt{
\XX(Z,U)\tensor \XX(Y,Z) \tensor \XX(X,Y) 
\ar[r]^-{\Id \tensor c_{X,Y,Z}} \ar[d]_{c_{Y,Z,U} \tensor \Id}
& 
\XX(Z,U)\tensor \XX(X,Z) 
\ar[d]^{c_{X,Z,U}}
\\
\XX(Y,U)\tensor \XX(X,Y) 
\ar[r]^{c_{X,Y,U}} 
&
\XX(X,U)
}
$$
$$
\xymatrix@R=30pt@C=27pt{
\One \tensor \XX(X,Y) 
\ar[r]^-{u_Y \tensor \Id} 
\ar[dr]_{\cong}
&
\XX(Y,Y)\tensor \XX(X,Y) 
\ar[d]^{c_{X,Y,Y}}
& 
\XX(X,Y)\tensor \XX(X,X) 
\ar[d]_{c_{X,X,Y}}
&
\XX(X,Y) \tensor \One 
\ar[l]_-{\Id \tensor u_X}
\ar[dl]^\cong
\\
&
\XX(X,Y)
& 
\XX(X,Y)
}
$$
Objects of $\XX(X,Y)$ will be denoted by $f:X\to Y$ and their $\V$-distance by $\XX(X,Y)(f,g)$ in $\V$. 
The action of $c_{X,Y,Z}$ at objects $(f',f)$ in $\XX(Y,Z)\tensor\XX(X,Y)$ is denoted  by $f'\cdot f$, and for their distances the inequality below (expressing that $c_{X,Y,Z}$ is a $\V$-functor) holds:
$$
\left(\XX(Y,Z)\tensor\XX(X,Y)\right)(
(f',f),(g',g)
)
\leq
\XX(X,Z)(f'\cdot f,g'\cdot g)
$$

\begin{exa}
\label{ex:Vcat-categories}
\hfill
\begin{enumerate}
\item Start with the simplest quantale, namely $\V=\Two$. 
We have seen in Example~\ref{exle:V-categories}\eqref{exle:Twocat} that preorders are $\Two$-categories, and that monotone maps are $\Two$-functors, hence that $\Twocat=\Pre$. 
Then $\Twocat$-categories are categories with ordered homsets, such that composition is monotone in both arguments.
Examples of $\Twocat$-categories are the category $\Twocat=\Pre$ of preorders itself, the category $\Set$ of sets with the discrete enrichment (see Section~\ref{sec:discrete functor} on the discrete enrichment of $\Set$, discussed for general $\V$), as well as any $\Vcat$-category, by~\eqref{eq:V-metric=>order}. In particular, $\Vcat$ as such can be seen $\Twocat$-enriched (for more details, we refer to Section~\ref{sec:beh-cat}, where this situation, known as {\em change-of-base}, is discussed).

\item The $\boldsymbol{O}$-categories of Smyth and Plotkin~\cite{smyth-plotkin} are special cases of $\Pre$-enriched categories, in the sense that the hom-sets are not only preorders, but actual partial orders such that  every ascending $\omega$-sequence has a 1.u.b. and composition of morphisms is $\omega$-continuous.

\item \label{ex:Vcat as enriched over itself}
$\Vcat$ itself is a $\Vcat$-category (see~\cite[Section~1.6]{kelly:book} and also Example~\ref{exle:V-categories}\eqref{ex:V as V-category itself}).

\item\label{ex:metric-enriched}
Let $\V=([0,\infty],\geq_{\mathbb R},+,0)$ be the quantale of Example~\ref{exs:Vcats=metric_spaces}\eqref{ex:V=[0,infty]}. As explained there, $\V$-categories are generalised metric spaces, and $\V$-functors are non-expanding maps, while $\Vcat$-categories are known as locally metric categories, or metric-enriched categories -- an example of such being the (sub)category of $\Vcat$ consisting of complete and bounded-by-1 metric spaces $\mathsf{CMS}$, with embedding-projection pairs as arrows~\cite{america-rutten,wagner}. $\mathsf{CMS}$ provides an appropriate context for studying reflexive quantitative domain equations~\cite{america-rutten,flag-kopp:continuity-spaces,turi-rutten,wagner}. The $\Vcat$-category structure of $\mathsf{CMS}$ is inherited from $\Vcat$ itself. 

\end{enumerate} 
\end{exa}

\medskip\noindent
A {\em $\Vcat$-functor\/} $F:\XX\to\YY$ is given by:
 
\begin{enumerate}
\item 
The assignment $X\mapsto FX$ on objects.
\item
For each pair of objects $X,Y$ in $\XX$, a $\V$-functor $F_{X,Y}:\XX(X,Y)\to\YY(FX,FY)$,
whose action on objects $f:X\to Y$ is denoted by $Ff:FX\to FY$.
For the distances we have the inequality
$$
\XX(X,Y)(f,g)
\leq
\YY(FX,FY)(Ff,Fg)
$$
\end{enumerate}

\noindent Of course, the diagrams of $\V$-functors below, expressing the preservation of unit and composition, should commute:
$$
\xymatrix@R=17pt@C=10pt{
\XX(X,X)
\ar[0,2]^-{F_{X,X}}
&
&
\YY(FX,FX)
\\
&
\One
\ar[-1,-1]^{u_X}
\ar[-1,1]_{u_{FX}}
&
}
\quad
\xymatrix@R=17pt@C=17pt{
\XX(Y,Z)\tensor\XX(X,Y)
\ar[0,2]^-{F_{Y,Z}\tensor F_{X,Y}}
\ar[1,0]_{c_{X,Y,Z}}
&
&
\YY(FY,FZ)\tensor\YY(FX,FY)
\ar[1,0]^{c_{
FX,FY,FZ
}}
\\
\XX(X,Z)
\ar[0,2]_-{F_{X,Z}}
&
&
\YY(FX,FZ)
}
$$

\medskip\noindent Given $F,G:\XX\to\YY$, a {\em $\Vcat$-natural transformation} $\tau:F\to G$ is given by a collection of $\V$-functors $\tau_X:\One\to\YY(FX,GX)$, such that the diagram
$$
\xymatrix@C=37pt@R=20pt{
&
\One\tensor\XX(X,Y)
\ar[0,1]^-{\tau_{Y}\tensor F_{X,Y}}
&
\YY(FY,GY)\tensor \YY(FX,FY)
\ar[1,1]^{\phantom{MM}c_{FX,FY,GY}}
&
\\
\XX(X,Y)
\ar[-1,1]^{\cong}
\ar[1,1]_{\cong}
&
&
&
\YY(FX,GY)
\\
&
\XX(X,Y)\tensor \One
\ar[0,1]_-{G_{X,Y}\tensor\tau_{X}}
&
\YY(GX,GY)\tensor\YY(FX,GX)
\ar[-1,1]_{\phantom{MM}c_{FX,GX,GY}}
&
}
$$
of $\V$-functors commutes. We will abuse the notation and denote by $\tau_X:FX\to GX$ the 
image in $\YY(FX,GX)$ of the object $0$ from $\One$ under the $\V$-functor $\tau_X:\One\to\YY(FX,GX)$.
The above diagram (when read at the object-assignments of the
ambient $\V$-functors) then translates as the equality
\begin{equation}\label{eq:Vcat-nat}
Gf \cdot \tau_{X}=\tau_{Y} \cdot Ff
\end{equation}
of objects of the $\V$-category $\YY(FX,GY)$, for every object $f:X \to Y$. On hom-objects, the above diagram says 
nothing\footnote{ This is
well-known for $\Pre$-natural transformations: one only needs
to verify ordinary naturality.} 
(recall that
$\V_o$ is a poset, hence there are no parallel pairs of morphisms
in $\V_o$).

\begin{exa}\label{ex:V-cat-functor}
\hfill
\begin{enumerate}

\item \label{ex:Two-cat-functor}
We consider again first the case $\V=\Two$. A $\Twocat$-functor, that is, a $\Pre$-functor, is also known as a locally monotone functor. 
For example, the usual discrete functor $D:\Set\to\Pre$ is a $\Pre$-functor,  but the forgetful functor $V:\Pre_o\to\Set_o$ is not, because $V$ is not locally monotone (the mapping $V_{X,Y}:\Pre(X,Y)\to\Set(VX,VY)$ is not order-preserving).

\item The hom-contracting functors on $\mathsf{CMS}$, employed in the metric domain equations of~\cite{america-rutten}, are in fact $\Vcat$-enriched functors, for $\V=[0,\infty]$ (see Example~\ref{ex:Vcat-categories}\eqref{ex:metric-enriched}). 

\end{enumerate}
\end{exa}


\subsection{On the fully-faithfulness of the discrete functor $D:\Set \to \Vcat
$}\label{sec:discrete functor}

\ 

\bigskip\noindent
This section studies in detail the observation that the discrete functor $$D:\Set \to \Vcat$$ is fully faithful if and only if $\V$ is integral, that is, $e=\top$.%
\footnote{All our examples are integral with the exception of the last two items of Example~\ref{ex:quantale}.} 
Indeed, if $e<\top$ then 
$$\Set(\emptyset,\emptyset)(\Id,\Id) = e 
\quad\quad\textrm{and}\quad\quad 
\Vcat(\mathbb 0,\mathbb 0)(\Id,\Id) = \top$$
where $\mathbb 0 = D\emptyset$ denotes the empty $\V$-category.

\medskip\noindent
This contradicts the fully faithfulness of $D$, which is defined in $\Vcat$ via the existence of an isomorphism $\Set(A,B)\cong \Vcat(DA,DB)$ and implies
$$\Set(A,B)(f,g) = \Vcat(DA,DB)(Df,Dg).$$

\noindent
To make the argument above precise, we need to explain in what sense $\Set$ is a $\Vcat$-category and in what sense $D$ is a $\Vcat$-functor. The remainder of this section formalises the concept of discreteness within $\Vcat$ following~\cite{eilenberg-kelly, kelly:book} and may be skipped by a reader who accepts the argument as sketched above. 

\medskip\noindent
Let $\Set$, for now, denote the ordinary category of sets and functions. The ordinary forgetful functor $V:\Vcat_o\to\Set$, mapping each $\V$-category to its set of objects, has a left adjoint~\cite[Section~2.5]{kelly:book}, known as the free $\Vcat$-category functor (or, as we will call it, {\em the discrete functor})
$$D:\Set \to \Vcat_o$$

\medskip\noindent
The functor $D$ maps each set $X$ to the $\V$-category $X\cdot \mathbb 1$, the coproduct in $\Vcat$ of $X$-copies of the unit $\V$-category $\mathbb 1$ (recall that $\mathbb 1$ has one object $0$, with $\V$-self-distance $\mathbb 1(0,0)=e$). That means that each $\V$-category $DX$ is discrete, as in~\eqref{eq:discrete V-category}.  
In particular, $D\emptyset$ is the empty $\V$-category $\mathbb 0$.

\medskip\noindent 
By~\cite[Section~2.5]{kelly:book}, $D$ is strong monoidal: 
\[
D(X \times Y) \cong DX \otimes DY, \quad D1 \cong \One
\]
hence it induces a (2-)functor
$$D_\ast:\Set\hspace{.9pt}\mbox{-}\hspace{.9pt}{\mathsf{cat}} \to \Vcat\hspace{.9pt}\mbox{-}\hspace{.9pt}{\mathsf{cat}}$$
between ordinary categories (i.e., $\Set$-enriched) and $\Vcat$-enriched categories~\cite[Section~6]{eilenberg-kelly}. 
The functor $D_\ast$ maps an ordinary category $\mathcal C$ to the $\Vcat$-category with the same set of objects as $\mathcal C$, and $\Vcat$-homs $(D_\ast\mathcal C)(f,g)=D(\mathcal C(f,g))$. 
In particular, $\Set$ itself gets enriched to a $\Vcat$-category $D_\ast\Set$ with sets as objects, and for any sets $X,Y$, 
$$(D_\ast\Set)(X,Y)=D(\Set(X,Y))=\Set(X,Y)\cdot \mathbb 1$$
Notice that $(D_\ast\Set)(\emptyset,\emptyset) =\mathbb 1$ (as there is only one map $\emptyset\to\emptyset$, namely the identity)

\medskip\noindent 
Then $D$ can be perceived as a $\Vcat$-functor 
$$\widehat{D}:D_\ast \Set \to \Vcat$$
with same action as $D$ on objects (i.e. it maps a set $X$ to the $\V$-category $DX=X\cdot \mathbb 1$). 
Saying that $\widehat D$ is a $\Vcat$-functor means that for each sets $X,Y$, there is a $\V$-functor 
$$\widehat{D}_{X,Y}:(D_\ast\Set)(X,Y)\to \Vcat(DX,DY)$$
induced by the monoidal (closed) structure of $D$, and this collection of $\V$-functors is compatible with identity and composition.
Explicitly, $\widehat{D}_{X,Y}$ maps a function $f:X\to Y$ to the $\V$-functor $Df:DX \to DY$, the $\V$-functor structure of $\widehat{D}_{X,Y}$ being witnessed by the inequality
$$(D_\ast\Set)(X,Y)(f,g)\leq \Vcat(DX,DY)(Df,Dg)$$
In particular, there is a $\V$-functor
$$\widehat{D}_{\emptyset,\emptyset}:(D_\ast\Set)(\emptyset,\emptyset)=\mathbb 1 \to \Vcat(D\emptyset,D\emptyset)$$
Observe now that the only $\V$-functor $D\emptyset \to D\emptyset$ is the identity $\Id$, and that 
$$\Vcat(D\emptyset,D\emptyset)(\Id,\Id)=\top$$
hence $\Vcat(D\emptyset, D\emptyset)$ is (isomorphic to) the $\V$-category $\mathbb 1_\top$ with one object $0$ and $\mathbb 1_\top(0,0)=\top$. 
Therefore we have 
$$\widehat{D}_{\emptyset,\emptyset}:\mathbb 1 \to \mathbb 1_\top$$
This situation extends to 
\begin{equation}\label{eq:D not ff}
\widehat{D}_{\emptyset, Y}: (D_\ast\Set)(\emptyset, Y)= \Set(\emptyset, Y)\cdot \mathbb 1 = \mathbb 1 \to \Vcat(D\emptyset,DY)=\Vcat(\mathbb 0,DY)=\mathbb 1_\top
\end{equation}
for arbitrary set $Y$. 

\medskip\noindent
In view of the observation above, we may formulate the following result, vital for the development of the paper:

\begin{prop}\label{prop:D-ff} 
The $\Vcat$-functor $\widehat{D}:D_\ast\Set\to \Vcat$ is  fully-faithful if and only if the quantale $\V$ is integral.
\end{prop}

\begin{proof}
It is easy to see that for non-empty $X$, $\Vcat(DX,DY)$ is a discrete $\V$-category, hence 
\[
\widehat{D}_{X, Y}: (D_\ast\Set)(X, Y)= \Set(X, Y)\cdot \mathbb 1 \to \Vcat(DX,DY)
\]
is an isomorphism of $\V$-categories. The result now follows from~\eqref{eq:D not ff}, using that $e=\top$ iff $\mathbb 1 \cong \mathbb 1_\top$.
\end{proof}

\begin{rem}\label{ex:discrete V-cat-functor}
\hfill
\begin{enumerate}
\item Notice that not only $D$, but actually the whole adjunction $D\dashv V$ lifts to an adjunction 
\[
D_\ast \dashv V_\ast : \Vcat\hspace{.9pt}\mbox{-}\hspace{.9pt}{\mathsf{cat}} \to \Set\hspace{.9pt}\mbox{-}\hspace{.9pt}{\mathsf{cat}} 
\]
where $V_\ast$ maps a $\Vcat$-category to its underlying ordinary category. 
That means that ordinary functors, like $\Set\to \Set = V_\ast D_\ast \Set$ or $\Set\to \Vcat_o = V_\ast \Vcat$, and ordinary natural transformations between such, automatically get $\Vcat$ enriched to $D_\ast \Set \to D_\ast \Set$ or $D_\ast \Set \to \Vcat$, and so on.
\item 
It will be important in the sequel that although $D$ acquires an enrichment to the $\Vcat$-functor $\widehat D$, its ordinary right adjoint $V$ does not. Hence $\widehat D$ lacks a $\Vcat$-enriched right adjoint.
\end{enumerate}
\end{rem}

\paragraph{\bf Notation.} For simplicity, we write in the sequel $\Set$ for the $\Vcat$-enriched category $D_\ast\Set$ and $D:\Set\to\Vcat$ for the $\Vcat$-functor  $\widehat{D}:D_\ast \Set \to \Vcat$. When referring to $\Set$ or $D$ as an ordinary category/functor, we will denote them by $\Set_o$ and $D_o:\Set_o\to \Vcat_o$, respectively. 
This agrees with the convention set at the beginning of the paper, because $V_\ast D_\ast = \Id$, as the reader can easily check.


\subsection{On $\Vcat$-enriched colimits}\label{sec:colimit}

\

\bigskip\noindent
The last bit of standard notation from enriched category %
theory concerns colimits. We recall it for %
$\Vcat$-categories.

\begin{defi}\label{def:colimit}
A {\em colimit\/} of a $\Vcat$-functor $F:\AA\to\XX$, where $\AA$ is small, weighted by a $\Vcat$-functor $W:\AA^\op\to\Vcat$, consists of an object $\Colim{W}{F}$ of $\XX$, together with an isomorphism in $\Vcat$
\[
\XX(\Colim{W}{F},X)\cong [\AA^\op,\Vcat](W,\XX(F-,X))
\]
which is $\Vcat$-natural in $X$.%
\footnote{
Above, $[\AA^\op, \Vcat]$ is the usual $\Vcat$-category of presheaves over $\AA$.}%
\qed
\end{defi}

\medskip\noindent
In case $\AA$ is the unit $\Vcat$-category,%
\footnote{
Having only one object, with corresponding $\V$-category-hom $\One$.} %
we may identify $F$ with an object $X$ of $\XX$ and $W$ with a $\V$-category $\A$. %
The resulting weighted colimit, usually called the {\em copower} (or the {\em tensor}) of $\kat{A}$ with $X$, shall be denoted then $\kat{A}\Sum X$ instead of $\Colim{W}{F}$.

\medskip\noindent 
To emphasise the importance of \emph{weighted} colimits, recall that in ordinary category theory, in particular in algebraic theories, quotienting by a set of equality constraints produces a coequalizer (an ordinary colimit). Contrary, in the simplest $\Vcat$-enriched setting -- that is, $\Pre$-enriched, quotienting by a set of inequality constraints is an example of a weighted colimit called {\em coinserter}, which cannot be obtained using only ordinary colimits. %

\begin{exa}\label{exle:coinserter}
In the case $\V=\Two$, for which $\Twocat=\Pre$, a coinserter is a colimit that has as weight 
$\varphi:\AA^\op \to \Pre$, where $\mathbb A$ 
is 
\[
\xymatrix{\cdot \ar@<0.5ex>[r] \ar@<-0.5ex>[r] & \cdot}
\] 
and $W$ maps $\AA$
to the parallel pair
\[
\xymatrix{\mathbb 2 &\ar@<-0.5ex>[l]_1 \ar@<0.5ex>[l]^0 \mathbb 1}
\]
in $\Pre$, with arrow $0$ mapping to $0\in\mathbb 2$ and arrow $1$ mapping to $1\in\mathbb 2$ (recall that $\mathbb 2$ is the preorder  $\{0 < 1\}$). 
\end{exa}

\noindent
This example is of importance to us because the next section is based on the $\V$-generalisation of the observation that every preorder is a coinserter of discrete preorders, which was used in \cite{lmcs:bkv} to show that $\Set$ is dense in $\Pos$ (and also in $\Pre$).

\begin{exa}\label{exle:preord-coins-discrete}
Let $X$ be a poset (or a preordered set). Denote by 
$X_0=DVX$ the discrete preorder of the elements of $X$, and by $X_1$ the  discrete preorder 
of all comparable pairs, 
$X_1=\{ (x',x)\in X \mid x'\leq x\}=DV(X^\Two)$.\footnote{ $X^\Two=\Pre(\Two, X)$ is the ordered set of monotone maps $\Two \to X$.}
Let $d_0, d_1:X_1\to X_0$ be the two projections. Then the obvious  map $q:X_0\to X$, $q(x)=x$, exhibits $X$ as the coinserter
\begin{equation}
\label{eq:poset_coins}
\xymatrix@C=35pt{
X_1 
\ar@<0.5ex>[r]^{d_0} 
\ar@<-0.5ex>[r]_{d_1} 
& 
X_0 \ar[r]^-q & X}
\end{equation} 
in $\Pre$ of the diagram 
\begin{equation}
\xymatrix@C=35pt{
X_1 
\ar@<0.5ex>[r]^{d_0} 
\ar@<-0.5ex>[r]_{d_1}
&
X_0.
} 
\end{equation} 
of discrete posets $X_1,X_0$. \qed
\end{exa}

\medskip\noindent 
To follow the technical developments of the next section, it is worth understanding which part of Definition~\ref{def:colimit} forces the colimit $X$ in \eqref{eq:poset_coins} to be ordered. 
If we instantiate Definition~\ref{def:colimit} with $\XX=\Pre$ and $W$, $F$ as in Example~\ref{exle:coinserter}, then it is the monotonicity of the component $\Two\to\Pre(X_1,X)$ of the natural transformation in 
$$[\AA^\op,\Pre](W,\Pre(F-,X))$$
that gives $c \circ d_0\le c \circ d_1$.

\bigskip\noindent
Having introduced weighted colimits, we are now in the position to talk about Kan extensions: 

\begin{defi}[Kan extension]
\label{def:Lan}
Let $J:\mathbb A \to \mathbb B$, 
$H:\mathbb A \to \mathbb X$ be $\Vcat$-functors. 
A $\Vcat$-enriched \emph{left Kan extension of $H$ along $J$}, is a $\Vcat$-functor $\Lan{J}{H}:\mathbb B\to \mathbb X$, 
such that there is a $\Vcat$-natural isomorphism 
\begin{equation}\label{eq:lan}
(\Lan{J}{H}) B \cong \mathbb B(J-, B) * H
\end{equation}
for each $B$ in $\mathbb B$.
\end{defi}

\begin{rem}
\label{rem:lan}
\hfill
\begin{enumerate}
\item According to Definition~\ref{def:colimit}, Equation~\eqref{eq:lan} above simply says that there is an isomorphism 
\[
\mathbb X ((\Lan{J}{H})B,X) \cong [\mathbb A^\op,\Vcat](\mathbb B(J-, B), \mathbb X(H-, X))
\]
 in $\Vcat$, natural in $B$ and $X$.
\item 
\label{rem:weak_def_lan}
For any $\Vcat$-functor $H':\mathbb B\to \mathbb X$, there is a bijection between $\Vcat$-natural transformations $\Lan{J}{H}\to H'$ and $\Vcat$-natural transformations $H \to H' J$, 
in analogy to the case of ordinary left Kan extensions. 
In particular, there is a $\Vcat$-natural transformation 
$\alpha:H\to (\Lan{J}{H} ) J$, called the {\em unit\/} of the left Kan extension, which is {\em universal\/} in the sense that for a $\Vcat$-functor $H':\mathbb B \to \mathbb X$, any $\Vcat$-natural transformation $H\to H' J$ factorises through 
$\alpha$. If the codomain $\XX$ is $\Vcat$ itself, the above bijection can be taken as an alternative definition of the $\Vcat$-enriched left Kan extension (see the discussion after Theorem~4.43 in~\cite{kelly:book}).

\item \label{unit of Lan iso} If $J:\mathbb A \to \mathbb B$ is fully faithful, then the unit $\alpha:H \to ( \Lan{J}{H} )J$ of the left Kan extension is an isomorphism~\cite[Proposition~4.23]{kelly:book}. 

\item The $\Vcat$-enriched left Kan extension $\Lan{J}{H}$ exists whenever $\mathbb A$  is small and $\mathbb X$ is cocomplete~\cite[Section~4.1]{kelly:book}.
But it might exist even when $\mathbb A$ is not small, as we will see later in a special case (Theorem~\ref{thm:Lan}).
\qed
\end{enumerate}
\end{rem}


\section{Extending functors from $\Set$ to $\Vcat$}
\label{sec:extension} 

As explained in the introduction, we are interested in left Kan extensions
$$
\xymatrix@R=20pt{
\Vcat
\ar[0,2]^-{\Lan{D}{(DT)}}
&
&
\Vcat
\\
\Set
\ar[rr]_-{T}
\ar[-1,0]^-{D}
&
&
\Set\ar[u]_D
}
\quad\quad
\quad\quad
\xymatrix@R=20pt{
\Vcat
\ar[0,2]^-{\Lan{D}{H}}
&
&
\Vcat
\\
\Set
\ar[-1,2]_-{H}
\ar[-1,0]^-{D}
&
&
}
$$
where $\V=(\V_o,\tensor,e,[{-},{-}])$ is again a commutative quantale and $D:\Set \to \Vcat$ is the discrete functor mapping a set $X$ to the discrete $\V$-category $DX$, that is, the $\V$-category that has $X$ as set of objects, all self-distances $e$, and all other distances bottom, as in Section~\ref{sec:discrete functor}. 

\medskip\noindent
We call $\Lan{D}{(DT)}$ in the left-hand diagram the \emph{$\Vcat$-ification} of $T$.

\medskip\noindent
It is now important to note that we are interested in $\Lan{D}{(DT)}$ and, more generally, in $\Lan{D} H$, in the $\Vcat$-enriched sense and not in the ordinary sense. 
Given that the ordinary functor $D_o$ has the forgetful functor $V:\Vcat_o\to\Set_o$ as a right adjoint, the {\em ordinary} left Kan extension of $H_o$ along $D_o$, is $H_oV$.
We observe that $\Lan{D_o} H_o =H_oV$ is not interesting from a metric point of view as $H_oV(\X)$ does not depend on the metric of $\X$. 
But, crucially, $D$ does not have an \emph{enriched} right adjoint and, as we will see, the $\Vcat$-enriched left Kan extension $\Lan{D}H$ is rather more interesting than $\Lan{D_o}H_o=H_oV$.


\subsection{Preliminaries on extensions and liftings}

\ 

\bigskip\noindent
Here we will have a closer look at extensions along $D:\Set\to\Vcat$ and liftings along $V:\Vcat_o\to\Set_o$.  
Notice that the ordinary adjunction 
$D_o\dashv V:\Vcat_o\to\Set_o$ exhibits $\Set_o$ as a full coreflective subcategory of $\Vcat_o$. 

\begin{defi}\label{def:ext-lift}
Let $T:\Set\to \Set$, $\ol T:\Vcat \to \Vcat$ be $\Vcat$-functors. 
\begin{enumerate}
\item \label{def:ext}
We say that a $\Vcat$-natural isomorphism
$$
\xymatrix@R=20pt{
\Vcat
\ar[0,2]^-{\ol{T}}
&
&
\Vcat
\\
\Set
\ar[0,2]_-{T}
\ar[-1,0]^{D}
\ar@{}[-1,2]|{\nwarrow\alpha}
&
&
\Set
\ar[-1,0]_{D}
}
$$
of $\Vcat$-functors exhibits $\ol{T}$ as an {\em extension\/} of $T$. 

\noindent 
If a natural (iso)morphism $\alpha:DT\to \ol T D$ happens to be the unit of a left Kan extension, that is, if $\ol{T}\cong \Lan{D}{(D  T)}$ holds, then we say that $\alpha$ exhibits $\ol{T}$ as the {\em $\Vcat$-ification\/} of $T$, and we will denote it by $T_\V$.
\item \label{def:lift}
We say that a natural isomorphism
$$
\xymatrix@R=20pt{
\Vcat_o
\ar[0,2]^-{\ol{T}_o}
&
&
\Vcat_o
\\
\Set_o
\ar[0,2]_-{T_o}
\ar@{<-}[-1,0]^{V}
\ar@{}[-1,2]|{\nearrow\beta}
&
&
\Set_o
\ar@{<-}[-1,0]_{V}
}
$$
of ordinary functors exhibits $\ol{T}$ as a {\em lifting\/} of $T$. 
\end{enumerate}
\end{defi}

\begin{exas}\label{ex:extensions}
\hfill
\begin{enumerate}
\item
The identity $\Vcat$-functor $\Id:\Vcat\to\Vcat$ is an extension and a lifting of the identity functor on $\Set$. 
We will see later that $\Id:\Vcat\to\Vcat$ is universal among all extensions of $\Id:\Set\to\Set$, in that it is the $\Vcat$-ification  (for an arbitrary commutative quantale $\V$) of the identity functor on $\Set$.

\item \label{ex:unicity-of-extens}
If $D:\Set\to \Vcat$ has a left adjoint $C\dashv D$,%
\footnote{
In case $\V=\Two$, the left adjoint $C:\Pre\to\Set$ takes connected components. For the general situation see Remark~\ref{rem:change-of-base-cd}.
} %
then not only the identity functor $\Id$ on $\Vcat$ but also the composite $DC$ is an extension of the identity functor (but not a lifting!). 

\item 
The convex powerset functor $\overline{\mathcal P}:\Pos\to\Pos$ is an extension, in fact the  posetification \cite{lmcs:bkv}, of $\mathcal P:\Set\to\Set$, but is not a lifting. On the other hand, the extension of the powerset functor to $\Pre\to\Pre$ is also a lifting, as are all $\Vcat$-ifications (see ).

\item Every functor $T:\Set \to \Set$ admits the lifting $\ol T_\top:\Vcat\to \Vcat$ mapping a $\V$-category $\X$ to the $\V$-category with set of objects $TX_\extra$, where $X_\extra$ is the set of objects of $\X$,%
\footnote{%
See Proposition~\ref{prop:Vcoinserter}.} %
and $\ol T_\top\X(A',A)=\top$ for all objects $A', A $. In fact, $\ol T_\top$ is final among all liftings of $T$, in in the sense that for any other lifting $\ol T$ with $\beta:T_o V \to V\ol T_o$, there is a unique $\Vcat$-natural transformation $\gamma:\ol T \to \ol T_\top$ such that $\beta_\top = V\gamma_o \cdot \beta$, where $\gamma _o$ is the ordinary natural transformation associated to $\gamma$, and $\beta_\top:T_o V \to V{\ol T_\top}_o$ is the isomorphism corresponding to $\ol T_\top$.

\item
The $\Vcat$-ification $T_\V$ exists for every {\em accessible\/}
functor $T:\Set\to\Set$ for general reasons. 
More in detail, if $T$ is $\lambda$-accessible
for a regular cardinal, then $T=\Lan{J_\lambda}{(T J_\lambda)}$,
where $J_\lambda:\Set_\lambda\to\Set$ is the inclusion of the full
subcategory $\Set_\lambda$ spanned by $\lambda$-small sets. Consequently,  
$$
T_\V=\Lan{D J_\lambda}{(D  T J_\lambda)}
$$
exhibits $T_\V$ as $\Lan{D}{(D  T)}$ by~\cite[Theorem~4.47]{kelly:book}.
In particular, the $\Vcat$-ification $(T_\Sigma)_\V$ exists for every polynomial 
functor 
$$
T_\Sigma X=\coprod_n \Set(n,X)\Sum\Sigma n 
$$
where $\Sigma:|\Set_\lambda|\to\Set$ is a $\lambda$-ary signature.
We will give an explicit formula for the $\Vcat$-ification 
$(T_\Sigma)_\V$ later. 
\end{enumerate}   
\end{exas}


\subsection{The extension theorem}

\

\bigskip\noindent
Referring to the last item of the examples above, one of the aims of this paper is to show that $\Vcat$-ifications 
exist even if $T$ is not accessible. 
This is a consequence of a more general result, which we are going to prove as Theorem~\ref{thm:Lan}, namely that the left Kan-extension
$$
\xymatrix@R=20pt{
\Vcat
\ar[0,2]^-{\Lan{D}{H}}
&
&
\Vcat
\\
\Set
\ar[-1,2]_-{H}
\ar[-1,0]^-{D}
&
&
}
$$
exists for any $H:\Set\to\Vcat$.

\medskip\noindent
We will give an explicit construction of these left Kan-extensions, the idea of which is as follows. 
\begin{itemize}
\item
Every $\V$-category $\X$ can be represented as a certain small colimit of discrete $\V$-categories. 

\item
Since $H$ is defined on discrete $\V$-categories,  $H$ can be extended to $\X$ by  applying $H$ to the corresponding discrete categories and then computing the colimit.

\item Since the colimit is small, the extension of $H$ always exists, even though $\Set$ is not a small subcategory of $\Vcat$.
\end{itemize}

\medskip\noindent
Recall from Example~\ref{exle:coinserter} the definition of a coinserter, and from Example~\ref{exle:preord-coins-discrete} how any $\Two$-category (preorder) appears as the coinserter  \eqref{eq:poset_coins} of discrete $\Two$-categories.
Before adapting the notion of coinserter to enrichment over $\V$, we first explain why the obvious modification of the $\Twocat$-situation leads to a notion that is too strong for our purposes.

\begin{rem}
Given that we want to generalise the construction \eqref{eq:poset_coins} of a preorder as a coinserter
\begin{equation}\label{eq:canonical-coinserter}
\xymatrix@C=35pt{
DV(X^\Two)
\ar@<0.5ex>[r]^{d_0} 
\ar@<-0.5ex>[r]_{d_1} 
& 
DVX \ar[r]^-q & X}
\end{equation} 
it is tempting to ask what it could mean for 
\begin{equation}\label{eq:XtotheV}
\xymatrix@C=45pt{
DV(\X^\V)
\ar[r]^{(d^r)_{r\in\V_o} }
& 
DV\X \ar[r]^-q & \X}
\end{equation} 
 to be a colimit in $\Vcat$, where we denoted $\X^\V = \Vcat(\V,\X)$. In detail, let $\AA$ be the free $\Vcat$-category built on the ordinary category with two objects $\{1,0\}$ and with arrows
\[
\xymatrix{1 \ar[r]^{\delta^r} & 0}
\]
for all $r\in\V_o$. 
In analogy with Example~\ref{exle:coinserter}, we consider the weight $W:\AA^\op \to \Vcat$ mapping $0$ to $\One$ and $1$ to $\V$.%
\footnote{
Recall from Section~\ref{sec:V-cat-cat} that $\One$ is the unit $\V$-category, with only one object $0$ and with $\One(0,0)=e$.} %
On arrows, $W\delta^r$ sends the unique object of $\One$ to $r\in \V_o$
\[
\xymatrix{\One \ar[r]^{0\mapsto r} & \V}
\]
Furthermore, let $F:\AA \to \Vcat$ be the $\Vcat$-functor which maps the arrow $\delta^r:1\to 0$ to the ``evaluation at $r$'' $\V$-functor 
\[
\xymatrix{DV(\X^\V) \ar[r] & DV\X}
\]
that is, $F\delta^r(f)=f(r)$. 
Now spelling out Definition~\ref{def:colimit} of what it would mean for $\X$ to be the colimit $W * F$, we see that for each $\V$-category $\Y$, 
$$\Vcat(\X,\Y)\cong [\AA^\op,\Vcat](W,\Vcat(F-,\Y))$$
an object of the $\V$-category on the right hand side amounts to a map on objects $g:\X\to \Y$ satisfying 
\begin{equation}\label{eq:rsle}
[s,r]\le \Y(g(f(s)),g(f(r)))
\end{equation} 
for all $\V$-functors $f:\V \to \X$. For $\X$ to be a colimit, one needs the implication 
\[
r\le \X(x',x) \ \Rightarrow r\le \Y(g(x'),g(x))
\]
to hold for all $r\in\V$, which follows from \eqref{eq:rsle}, if one can find $f:\V\to\X$ such that $f(e)=x'$ and $f(r)=x$. But note that, in the case of  $\V=[0,\infty]$ and $\X=\{x',x\}$ with $0<_\mathbb R\X(x',x)<_\mathbb R\infty$ no such $f$ exists. Indeed, it would be interesting to define $\X$ to be ``path-connected'', or $\V$-connected, if for all $r\le \X(x',x)$ there is a $\V$-functor $f:\V\to\X$ such that $f(e)=x', f(r)=x$. \qed
 
\end{rem}

\noindent Intuitively, what goes wrong with trying to set up \eqref{eq:XtotheV} as a colimit is that $\V$-functors $f:\V\to \X$ can be perceived as ``paths'' and that a general $\X$ need not be ``path-connected''. While this may be interesting to pursue in the future, the solution of our problem is to replace \eqref{eq:XtotheV} by
\begin{equation}\label{eq:Vnerve}
\xymatrix@C=45pt{
DX_r
\ar@<0.5ex>[r]^{\partial^r_0}
\ar@<-0.5ex>[r]_{\partial^r_1} 
& 
DV\X \ar[r]^-q & \X}
\end{equation} 
where the sets $X_r$ were introduced in \eqref{eq:Xr} as $X_r = \{(x',x) \mid r\le \X(x',x)\}$ and can equally be described as $V\X^{\Two_r}$ (see~\eqref{eq:Two_r} and Example~\ref{ex:non-Vcatif}) in total accordance now with~\eqref{eq:canonical-coinserter}. 

\medskip
\noindent
Below we give the full details and prove that every $\X$ is a colimit as in \eqref{eq:Vnerve}, from which the existence of left-Kan extensions along $D$ then follows.

\medskip
\noindent
We use the letter ``N''  in various fonts and notations 
to indicate the analogy with the nerve of an ordinary category.

\begin{defi}\label{def:Vcoinserter}
A $\V$-coinserter is a colimit that has the weight
$$N:\mathbb{N}^\op \to \Vcat,$$ where $\NN$ 
is the {\em free} $\Vcat$-category built upon the following ordinary category $\mathsf N$: the objects are all $r$ in $\V_o$, together with an extra symbol $\extra$, 
with arrows 
\[
\xymatrix{r \ar@<0.45ex>[r]^{\delta_1^r} \ar@<-0.45ex>[r]_{\delta_0^r} & \extra}
\] 
for $r$ ranging over the elements of $\V_o$. 
The weight $N$ is the $\Vcat$-functor sending $\extra$ to $\One$, and $r$ to $\Two_r$. 
The action of $N$ on arrows is defined as follows:
$N\delta^r_0:\One\to\Two_r$ sends $0$ to $0$, while $N\delta^r_1:\One\to\Two_r$ sends $0$ to $1$:
\[
\xymatrix{\mathbb 2_r &\ar@<-0.45ex>[l]_{N\delta^r_1} \ar@<0.45ex>[l]^{N\delta^r_0} \mathbb 1}
\]
\end{defi}

\begin{rem}
\label{rem:V-coins-exist}
As $\Vcat$ is complete and cocomplete as a $\Vcat$-category itself~\cite[Sections 3.2,~3.10]{kelly:book}, $\V$-coinserters do exist in $\Vcat$ 
and can be computed using copowers and conical colimits, the latter of which, on the level of objects, are computed as in $\Set_o$ because of the topologicity of the forgetful functor $\Vcat_o\to \Set_o$~\cite{hst:book}. But this general recipe is not so easy to use in practice, and we will provide in Corollary~\ref{cor:Lan by zigzag paths} and Lemma~\ref{lem:wpb} alternatives for computing the $N$-weighted colimits of interest to us.
\qed
\end{rem}

\noindent 
We are now going to show that every $\V$-category $\X$ is a $\V$-coinserter of discrete $\V$-categories which we would like to think of as components of ``its $\V$-nerve''.

\begin{prop}
\label{prop:Vcoinserter}
Let $\X$ be a $\V$-category and let $F\!_\X:\NN\to\Set$ be given by
\begin{align*}
\mathsf N & \to\Set_o
\\
\extra 
& 
\ \mapsto \ X_\extra 
\mbox{ , the set of objects of $\X$} \\ 
r & \ \mapsto \ X_r =\{(x',x)\in X_\extra \times X_\extra  \mid r\leq \X(x',x)\}
\end{align*}
with $F\!_\X\delta^r_0$ and $F\!_\X\delta^r_1$ the evident projections.
Then the colimit $$\Colim{N}{(D F_\X)}$$ of the diagram
\begin{equation*}\label{eq:DFX}
(N:\mathbb{N}^\op \to \Vcat,DF_\X:\NN\to\Vcat)
\end{equation*}
 in $\Vcat$ is isomorphic to $\X$.
\end{prop}

\begin{proof}
The colimit $\Colim{N}{(DF\!_\X)}$ exists in $\Vcat$, since
the $\Vcat$-category $\NN$ is small.

\medskip\noindent 
To ease the notation, we put 
\begin{gather*}
\X_\extra=D F_\X(\extra )
\quad\quad 
\quad\quad 
\X_r=DF_\X (r)\\
\ \ 
\partial^r_0=DF_\X\delta^r_0
\quad\quad 
\quad\quad \ 
\partial^r_1=D  F_\X\delta^r_1
\end{gather*}

\noindent 
Let us analyse the defining isomorphism
$$
\Vcat(\Colim{N}{(DF_\X)},\Y)
\cong
[\NN^\op,\Vcat](N,\Vcat(DF_\X{-},\Y))
$$
of $\V$-categories, natural in $\Y$.
The $\V$-category 
$$[\NN^\op,\Vcat](N,\Vcat(DF_\X{-},\Y))$$
of $N$-weighted ``cocones'' for $DF_\X$ is described explicitly as follows.
\begin{enumerate}
\item 
The objects are $\Vcat$-natural transformations
$$\tau:N\to \Vcat(DF_\X{-},\Y).$$ 
Each such $\tau$ consists of the following $\V$-functors:

\begin{itemize}
\item $\tau_\extra: N\extra\to\Vcat(\X_\extra,\Y)$. 

Since $N\extra=\One$, $\tau_\extra$
picks up a $\V$-functor $$f_\extra:\X_\extra\to \Y.$$
No other restrictions are imposed since $\One(0,0)=e$.
\item
$\tau_r:Nr\to\Vcat(\X_r,\Y)$. 

This $\V$-functor picks up
two $\V$-functors 
$$f^r_0:\X_r\to \Y$$ and 
$$f^r_1:\X_r\to \Y.$$
Since $\X_r$ is discrete, both $f^r_0$ and $f^r_1$ are
defined by their object-assignments only. There is, however,
the constraint below, because $Nr=\Two_r$:
$$
r\leq\bigwedge_{r \leq \X(x',x)} \Y(f^r_0 (x',x),f^r_1 (x',x))
$$ 
\end{itemize}
In addition to the above, there are various commutativity conditions since $\tau$
is natural. Explicitly, for $\delta^r_0:r\to \extra$, we have the commutative square
$$
\xymatrix@R=17pt{
N\extra =\One
\ar[0,2]^-{\tau_\extra}
\ar[1,0]_{N\delta^r_0}
&
&
\Vcat(\X_\extra,\Y)
\ar[1,0]^{\Vcat(\partial^r_0,\Y)}
\\
Nr=\Two_r
\ar[0,2]_-{\tau_r}
&
&
\Vcat(\X_r,\Y)
}
$$
that, on the level of objects, is the requirement 
$$f_\extra \cdot \partial^r_0 = f^r_0 $$
Similarly, for $\delta^r_1:r \to \extra$, we obtain 
$$f_\extra \cdot \partial^r_1 = f^r_1$$
We conclude that to give $\tau:N\to \Vcat(DF_\X{-},\Y)$ reduces to a $\V$-functor 
$$f_\extra:\X_\extra\to\Y$$
(and, recall, this $\V$-functor is given just by the
object-assignment $x\mapsto f_\extra x$, since $\X_\extra$ is discrete) such that
$r\leq \Y(f_\extra x',f_\extra x)$
holds for every object $(x',x)$ in 
$\X_r$ and every $r$,
which means precisely that 
$$\X(x',x)\leq \Y(f_\extra x',f_\extra x)$$ holds.

\medskip

\item
Given $\Vcat$-natural transformations $\tau$ and $\tau'$, then
the $\V$-distance between them is given by
$$
[\NN^\op,\Vcat](N,\Vcat(DF_\X{-},\Y))(\tau,\tau')
=
\bigwedge_{x} \Y(f_\extra x,f'_\extra x)
$$
where $f_\extra$ corresponds to $\tau$ and $f'_\extra$ corresponds to $\tau'$.
\end{enumerate}

\medskip\noindent
 From the above description of the $\V$-category of $N$-weighted ``cocones'' for $DF_\X$, it follows that the $\V$-functor 
\begin{equation}\label{eq:couniversal cocone}
q_\X:\X_\extra\to \X
\end{equation}
that sends each object $x$ to itself is the couniversal such ``cocone''. 
More precisely, the inequality $r\leq \X(q_\X x',q_\X x)$ holds for every $(x',x)$ in $\X_r$ and every $r$.

\medskip\noindent Furthermore, given any $\V$-functor $f_\extra:\X_\extra\to \Y$ with the above
properties, there is a unique $\V$-functor 
$f_\extra^\sharp:\X\to \Y$ such that $f_\extra^\sharp q_\X=f_\extra$ holds.

\medskip\noindent The ``2-dimensional aspect'' of the colimit, that is,
$$
\bigwedge_x \Y(f^\sharp_\extra x,f'^\sharp_\extra x)
=
\bigwedge_x \Y(f_\extra x,f'_\extra x).
$$
is satisfied because $f^\sharp_\extra$ and $f_\extra$ coincide on objects. 
Hence we have proved that $\X$ is isomorphic to $\Colim{N}{(DF_\X)}$.
\end{proof}

\begin{rem}
\hfill
\begin{enumerate}

\item Assume that the quantale $\V$ is integral. Then  following the same steps as in the proof above, one can see that the colimit $N\ast F_{DX}$ exists in $\Set$ and equals $X$. 
Intuitively, what is happening is that if $e=\top$, the diagram is non-empty everywhere and $D$ is fully faithful on non-empty domains, so it reflects colimits in general and the colimit $DX=N\ast DF_{DX}$ in particular.
\item The ordinary category $\mathsf N$ carries information about the elements of $\V_o$, but not about its order. This may seem peculiar at first sight, but the order is actually captured within the observation that the identity-on-objects $\Id_{r,s}:\Two_r =N(r) \to \Two_s = N(s)$ is a $\V$-functor iff $r\leq s$, and it is the only $\V$-functor such that $\Id_{r,s} \circ N\delta^r_i = N\delta^s_i$, $i=0,1$. 
Hence nothing is gained if additionally $\mathbb N$ is enhanced with arrows witnessing the order relation in $\V_o$, compatible with the existing arrows $r\rightrightarrows \extra$ and with (the transitivity of) the order relation in $\V_o$.
\end{enumerate}
\end{rem}

\medskip\noindent The lemma above allows us to compute left Kan extensions along $D:\Set\to\Vcat$:

\begin{thm}
\label{thm:Lan}
Every functor $H:\Set\to\Vcat$ 
has a $\Vcat$-enriched left Kan extension $H^\sharp:\Vcat\to\Vcat$
along $D:\Set\to\Vcat$ given by $H^\sharp\X = \Colim{N}{(HF\!_\X)}$.
\end{thm}

\begin{proof}
Based on Proposition~\ref{prop:Vcoinserter}, we will prove the following.
If we define $H^\sharp \X$ as the colimit 
$\Colim{N}{(H F\!_\X)}$, then 
the assignment $\X\mapsto H^\sharp\X$ can be extended to a $\Vcat$-functor
that is a left Kan extension of $H$ along $D$. 


\medskip\noindent
Suppose $H:\Set\to\Vcat$ is given. We will use the notation employed in the proof of Proposition~\ref{prop:Vcoinserter}.
\medskip
\begin{enumerate}
\item \label{item 2.a}
We first define a $\Vcat$-functor $H^\sharp:\Vcat\to\Vcat$.

\medskip\noindent
For every small $\V$-category $\X$, let $F_\X$ be as in Proposition~\ref{prop:Vcoinserter}. Recall that for every $r\in \V_o$, we put $F_\X(r)=X_r$, the {\em set} of pairs $(x',x)$
such that $r\leq\X(x',x)$, and $F_\X(\extra)=X_\extra$, the {\em set} of objects of $\X$.
Analogously, for a $\V$-functor $f:\X\to\Y$, we now denote by 
$f_r:X_r\to Y_r$ and $f_\extra:X_\extra\to Y_\extra$ the {\em maps} corresponding to $(x',x)\mapsto (fx',fx)$ and the object assignment of $f$, respectively.
Let also put $d^r_0=F_\X\delta^r_0$ and $d^r_1=F_\X \delta^r_1$. 
 
\medskip\noindent
We define 
\begin{equation}\label{eq:H(Vnerve)}
H^\sharp\X=\Colim{N}{(H F_\X)},
\end{equation}
where $N$ is as in Definition~\ref{def:Vcoinserter}. 

\medskip\noindent
Unravelling the definition of the weighted colimit, the $1$-dimensional aspect says that to give a $\V$-functor $f^\sharp : H^\sharp \X \to \Y$ is the same as to give a $\V$-functor $f:HX_\extra \to \Y$ such that 
\begin{equation}
\label{eq:couniversal-property}
r \leq \bigwedge_{C\in HX_r} \Y (f Hd^r_0 (C), f Hd^r_1(C))
\end{equation}
holds for all $r$.%
\footnote{
By slight abuse of language, we will use here and subsequently notation like $C\in HX_r$ to mean that $C$ runs through all objects in \emph{the $\V$-category} $HX_r$. 
} 
In particular, there is a ``quotient'' $\V$-functor $c_\X :HX_\extra \to H^\sharp \X$ such that
\begin{equation}
\label{eq:Hsharp-quotient}
r \leq \bigwedge_{C \in HX_r} H^\sharp \X (c_\X Hd^r_0 (C), c_\X Hd^r_1(C))\end{equation}
holds for all $r$, with the property that any $\V$-functor $HX_\extra \to \Y$ satisfying~\eqref{eq:couniversal-property} uniquely factorizes through $c_\X$. 

\medskip\noindent
The 2-dimensional aspect of the colimit says that given any $f,g:HX_\extra \to \Y$ satisfying~\eqref{eq:couniversal-property}, the relation 
\begin{equation}
\label{eq:2-dim-aspect}
\bigwedge_{B \in HX_\extra} \Y( f (B),g (B)) 
=
\bigwedge_{A \in H^\sharp \X} \Y( f^\sharp (A),g^\sharp (A)) 
\end{equation}
holds. 

\medskip\noindent
For a $\V$-functor $f:\X\to\Y$ we recall that the diagram
$$
\xymatrix{
X_r
\ar@<.5ex> [0,1]^-{d_1^r}
\ar@<-.5ex> [0,1]_-{d_0^r}
\ar[1,0]_{f_r}
&
X_\extra
\ar[1,0]^-{f_\extra}
\\
Y_r
\ar@<.5ex> [0,1]^-{d_1^r}
\ar@<-.5ex> [0,1]_-{d_0^r}
&
Y_\extra
}
$$
commutes serially. Hence $f$ induces a $\Vcat$-natural transformation
$F_f:F_\X\to F_\Y$. Therefore we can define 
$H^\sharp f:H^\sharp\X\to H^\sharp\Y$
as the unique mediating $\V$-functor
$$
\Colim{N}{(H F_f)}
:
\Colim{N}{(H F_\X)}
\to
\Colim{N}{(H F_\Y)}
$$
In particular, we have the commutative diagram below:
\begin{equation*}
\vcenter{\xymatrix{
HX_\extra 
\ar[r]^{c_\X} 
\ar[d]_{Hf_\extra}
& H^\sharp \X 
\ar[d]^{H^\sharp f}
\\
HY_\extra
\ar[r]^{c_\Y}
&
H^\sharp \Y
}}
\end{equation*}
Also, from the 2-dimensional aspect of the colimit (see Eq.~\eqref{eq:2-dim-aspect}), we have that for any $f,g:\X \to \Y$, the equality below holds: 
\begin{equation}
\label{eq:Hsharp-2-dim-aspect}
\bigwedge_{B \in HX_\extra} H^\sharp \Y( c_\Y Hf_\extra(B), c_\Y Hg_\extra (B)) 
=
\bigwedge_{A \in H^\sharp \X} H^\sharp \Y( H^\sharp f (A), H^\sharp g (A)) 
\end{equation}
It remains to prove that the inequality
$$
\Vcat(\X,\Y)(f,g)
\leq
\Vcat(H^\sharp \X,H^\sharp \Y)(H^\sharp f,H^\sharp g)
$$
is satisfied. 
To that end, suppose that $r\leq \Vcat(\X,\Y)(f,g)$
holds. This is equivalent to the fact that there is a 
mapping
$t:X_\extra\to Y_r$ such that the triangles
\begin{equation}
\label{eq:internal_nat1}
\vcenter{
\xymatrix{
X_\extra
\ar[0,1]^-{t}
\ar[1,1]_{f_\extra}
&
Y_r
\ar[1,0]^{d_0^r}
&
&
X_\extra
\ar[0,1]^-{t}
\ar[1,1]_{g_\extra}
&
Y_r
\ar[1,0]^{d_1^r}
\\
&
Y_\extra
&
&
&
Y_\extra
}
}
\end{equation}
commute. In fact, $t(x)=(f(x),g(x))$. 
To prove that 
$
r
\leq 
\Vcat(H^\sharp \X,H^\sharp \Y)(H^\sharp f,H^\sharp g)
$
holds, we need to prove the inequality
$$
r \leq \bigwedge_{A\in H^\sharp \X} H^\sharp \Y ( H^\sharp f(A), H^\sharp g(A))
$$
This follows from:
\allowdisplaybreaks
\begin{align*}
r 
\leq
& 
\bigwedge_{C \in HY_r} H^\sharp \Y (c_\Y Hd^r_0(C), c_\Y Hd^r_1 (C))
&
\mbox{by~\eqref{eq:Hsharp-quotient}}
\\
\leq
&
\bigwedge_{B \in HX_\extra} H^\sharp \Y (c_\Y Hd^r_0 Ht(B), c_\Y Hd^r_1 Ht (B)) 
\\
= 
&
\bigwedge_{B \in HX_\extra} H^\sharp \Y (c_\Y Hf_\extra (B), c_\Y Hg_\extra (B)) 
&
\mbox{by~\eqref{eq:internal_nat1}}
\\
=
&
\bigwedge_{A \in H^\sharp \X} H^\sharp \Y (H^\sharp f (A), H^\sharp g (A)) 
&
\mbox{by~\eqref{eq:Hsharp-2-dim-aspect}}
\end{align*}
Preservation of composition and identity follows easily from the colimit property, hence we obtain that the correspondence $\X\mapsto H^\sharp\X$ can be extended to a $\Vcat$-functor $H^\sharp:\Vcat\to\Vcat$.

\medskip

\item
We show now that $H^\sharp\cong\Lan{D}{H}$ holds. As explained in~\cite[Section~4.3]{kelly:book} (see also Remark~\ref{rem:lan}\eqref{rem:weak_def_lan}), it is enough to check that for any $\Vcat$-functor $K:\Vcat\to\Vcat$, there is a one-to-one correspondence between $\Vcat$-natural transformations $H^\sharp\to K$ and $H\to KD$. 

\medskip\noindent 
Due to the definition of $H^\sharp$, there is a $\Vcat$-natural transformation $\alpha:H\to H^\sharp D$ having as components the $\V$-functors $\alpha_X = c_{DX}$. 
We will prove that $\alpha$ is the unit of a left Kan extension.

\medskip\noindent
Let $\tau:H^\sharp\to K$ be a $\Vcat$-natural transformation.  
Then the composite
$$
\xymatrix{
H
\ar[0,1]^-{\alpha}
&
H^\sharp D
\ar[0,1]^-{\tau D}
&
K D
}
$$
is a $\Vcat$-natural transformation $\tau^\flat:H\to K D$.

\medskip\noindent 
Conversely, let $\sigma: H\to K D$ be a $\Vcat$-natural transformation.
To give a $\Vcat$-natural transformation $\sigma^\sharp:H^\sharp \to K$ means to give a family of $\V$-functors 
$$\sigma^\sharp:H^\sharp \X \to \K\X$$ 
satisfying the naturality condition~\eqref{eq:Vcat-nat}. 
In turn, each $\V$-functor $\sigma^\sharp:H^\sharp \X \to \K\X$ is uniquely determined by a $\V$-functor $HX_o\to K\X$ such that~\eqref{eq:couniversal-property} holds. 

\noindent 
For this, we choose the composite $\V$-functor 
$$HX_o\overset{\sigma_{X_o}}{\to} KDX_o\overset{Kq_{\X}}{\to} K\X$$
where $q_\X:DX_o \to \X$ is the canonical cocone~\eqref{eq:couniversal cocone} associated to each $\V$-category $\X$.

\noindent We verify~\eqref{eq:couniversal-property} for each $r\in \V_o$: 
\begin{align*} 
&&
r
&&
\leq 
&&&
\ \  \bigwedge_{(x',x)\in X_r} \X(x',x)
\\
&&
&&
= 
&&&
\bigwedge_{(x',x)\in \mathsf{ob} DX_r} \X(q_{\X}  d^r_0(x',x), q_{\X}  d^r_1(x',x))
\\ 
&&
&&
=
&&&
[DX_r, \X](q_{\X} d^r_0, q_{\X} d^r_1)
\\
&&
&&
\leq 
&&&
[KDX_r, K\X](K(q_\X d^r_0), K(q_\X d^r_1))
\\
&&
&&
=
&&&
\bigwedge_{B\in \mathsf{ob} KDX_r} K \X (Kq_{\X}    Kd^r_0(B), Kq_{\X}  Kd^r_1(B))
\\ 
&&
&&
\leq 
&&&
\bigwedge_{C\in \mathsf{ob}HX_r} K\X(Kq_{\X}  Kd^r_0  \sigma_{X_r}(C),Kq_{\X}  Kd^r_1 \sigma_{X_r}(C))
\\
&&
&&
= 
&&&
\bigwedge_{C\in \mathsf{ob}HX_r} K\X(Kq_{\X}  \sigma_{X_o}  Hd^r_0(C),Kq_{\X}  \sigma_{X_o}  Hd^r_1(C))
\end{align*}
In the above, we have used that $K$ is a $\Vcat$-functor and the naturality of $\sigma$. 
 
\medskip\noindent 
Hence there is a unique $\V$-functor $\sigma^\sharp_\X :H^\sharp \X \to K\X$ such that $\sigma^\sharp_\X  c_\X = Kq_\X  \sigma_{X_o}$ holds. 
Recall that $c_\X:HX_\extra \to H^\sharp \X$ is the colimiting cocone. 

\medskip\noindent 
We leave to the reader the verifications that the $\V$-functors $\sigma_\X$ are the components of a $\Vcat$-natural transformation $\sigma^\sharp :H^\sharp \to K$, and that the correspondences $\tau\mapsto\tau^\flat$ and
$\sigma\mapsto\sigma^\sharp$ are inverses to each other. \qedhere
\end{enumerate}
\end{proof}

\begin{cor}\label{cor:dense}
$D:\Set\to\Vcat$ is dense.
\end{cor}

\begin{proof}
It follows from Proposition~\ref{prop:Vcoinserter} and Theorem~\ref{thm:Lan} that the identity $\Vcat\to\Vcat$ is the left Kan extension of $D:\Set\to\Vcat$ along $D$. But this is one of the equivalent definitions of density in~\cite[Thm.5.1]{kelly:book}
\end{proof}

\medskip\noindent 
To be able to compute these left Kan extensions in concrete examples, we show that from the proof above one can extract a more explicit construction in terms of ``shortest paths''. \footnote{
Left Kan extensions can be computed pointwise as colimits, hence using copowers and conical colimits. Due to the topologicity of $\Vcat_0$ over $\Set_0$, the latter are obtained (on object-level) as in $\Set_0$ and endowed with the corresponding $\V$-metric structure. This is why the subsequent formula~\eqref{eq:shortest-path} is reminiscent of the usual construction of quotient metric spaces~\cite{smyth}.
}

\begin{cor}
\label{cor:Lan by zigzag paths}
The left Kan extension $H^\sharp = \Lan{D}{H}$ of a $\Vcat$-functor $H:\Set \to \Vcat$ is obtained as follows: 
\begin{enumerate}
\item For a $\V$-category $\X$, $H^\sharp \X$ is the $\V$-category having the same objects as $HX_\extra$ (that is, the underlying set of objects of the $\V$-category obtained by applying $H$ to $X_\extra$, the set of objects of $\X$). 
For any objects $A'$ and $A$ of $H^\sharp\X$, their corresponding $\V$-hom  $H^\sharp \X (A', A)$ is given by
\begin{equation}\label{eq:shortest-path}
\begin{gathered}
H^\sharp \X (A', A)
=  
\bigvee 
\{
{HX_\extra(A'_0,A_0)} {\tensor}  r_1 {\tensor} HX_\extra (A'_1,A_1) {\tensor} {\ldots} {\tensor} r_n {\tensor} HX_\extra(A_n', A_n) 
\}
\end{gathered}
\end{equation}
where the suprema is computed over all paths
$$
(A'_0,A_0), (A'_1, A_1), \ldots , (A'_n, A_n) 
$$ 
where $A'=A'_0$, $A=A_n$, and all (possibly empty) tuples of elements 
$$
(r_1, \ldots, r_n)
$$ 
such that there are objects $C_i$ in $HX_{r_i}$ with $Hd^{r_i}_0(C_i) = A_{i-1}$, $Hd^{r_i}_1(C_i) = A'_i$, for all $i=1,n$:%
$$
\xymatrix@C=-12pt@R=30pt{
&
C_1 
\ar[dl]|{Hd^{r_1}_0}
\ar[dr]|{Hd^{r_1}_1}
&
&
C_2 
\ar[dl]|{Hd^{r_2}_0}
\ar[dr]|{Hd^{r_2}_1}
&
&
\quad 
\quad 
\ar@{}[d]|{.\ . \ .}
&
&
\quad C_n \, \, 
\ar[dr]|{Hd^{r_n}_1}
&
&
\\
*+[l]{A' =A_0' , A_0} 
&
&
{\quad A'_1 , A_1 \quad}
&
&
{\quad A'_2  , A_2 \quad}
&
&
{\quad A'_{n-1}  , A_{n-1}\quad} \ar@{<-}[]!<5ex,1ex>;[ur]!<0ex,-1ex>|{Hd^{r_n}_0}
&
&
*+[r]{A_n'  , A_n=A}
}
$$
\item The couniversal cocone $c_\X : HX_\extra \to H^\sharp \X$ is the identity on objects. 
\item For a $\V$-functor $f:\X\to\Y$, $H^\sharp f$ acts as $Hf_\extra$ on objects.
\end{enumerate} 
\ 
\end{cor}

\begin{proof}
First, notice that the construct $H^\sharp\X$ described above satisfies 
\begin{equation}\label{eq:quotient V-functor}
HX_\extra (A',A) \leq H^\sharp \X (A',A)
\end{equation}
for all objects $A', A$. 
In particular, $e\leq H^\sharp \X (A,A)$ holds.
Next, the inequality $H^\sharp \X(A', A) \tensor H^\sharp \X (A'', A') \leq H^\sharp \X(A'', A)$ can be established by path concatenation, using that $HX_\extra$ is a $\V$-category. Hence $H^\sharp \X$ is a $\V$-category. 

\medskip\noindent
We will verify that $H^\sharp \X$ as above is indeed $\Colim{N}{(H F\!_\X)}$. Let $c_\X:HX_\extra \to H^\sharp \X$ be the identity on objects. It is a $\V$-functor by~\eqref{eq:quotient V-functor}. 

\medskip\noindent
We show now that~\eqref{eq:Hsharp-quotient} holds. Let $r$ be an arbitrary element of the quantale and fix an object $C$ in $HX_r$. Then by employing the path
$$
\xymatrix@C=-2pt@R=30pt{ 
&
&
C
\ar[dl]|{Hd^{r}_0}
\ar[dr]|{Hd^{r}_1}
&
\\
Hd^r_0(C)
\ar@{}[r]_{,}
&
Hd^r_0(C)
&
&
Hd^r_1(C)
\ar@{}[r]_{,}
&
Hd^r_1(C)
}
$$
one can see that 
\begin{align*}
r =
& 
\
e \tensor r\tensor e 
\\
\leq 
& 
\
HX_\extra (Hd^r_0(C), Hd^r_0(C)) \tensor r\tensor HX_\extra (Hd^r_1(C), Hd^r_1(C)) \\
\leq 
& 
\
H^\sharp \X (Hd^r_0(C), Hd^r_1(C))
\end{align*}

\medskip\noindent
We check the 1-dimensional aspect of the colimit: Let $f:HX_\extra \to \Y$ be a $\V$-functor such that~\eqref{eq:couniversal-property} holds. 
Then the object-assignment $A \mapsto f(A)$ (uniquely!) extends to a $\V$-functor $f^\sharp : H^\sharp \X \to \Y$ such that $f^\sharp c_\X = f$. 
This is because for each path $(A'=A'_0,A_0), (A'_1, A_1), \ldots , (A'_n, A_n=A)$ and each tuple of elements $(r_1, \ldots, r_n)$ we have that
\begin{align*}
&HX_\extra(A_0',A_0) \tensor  r_1 \tensor HX_\extra (A'_1,A_1) \tensor \ldots  \tensor r_n  \tensor HX_\extra(A'_n, A_n)
\\
&
=
\\
&
HX_\extra(A_0',Hd^{r_1}_0(C_1)) \tensor  r_1 \tensor HX_\extra (Hd^{r_1}_1(C_1),Hd^{r_2}_0(C_2)) \tensor \ldots \tensor r_n  \tensor HX_\extra(Hd^{r_n}_0(C_n), A_n)
\\
&
\leq
\\
&
\Y(f(A_0'),fHd^{r_1}_0(C_1)) \tensor \Y(fHd^{r_1}_0(C_1), fHd^{r_1}_1(C_1)) \tensor \Y( f Hd^{r_1}_1(C_1) , f Hd^{r_2}_0(C_2)) \tensor \ldots 
\\
&
\tensor \Y(fHd^{r_n}_0(C_n), fHd^{r_n}_1(C_n))  \tensor \Y ( f Hd^{r_n}_0(C_n), f(A_n) )
\\
&
\leq 
\\
&
\Y(f(A'), f(A))
\end{align*}
where we used that $f$ is a $\V$-functor and~\eqref{eq:couniversal-property}. 
This proves that 
$$H^\sharp \X(A',A) \leq \Y(f^\sharp (A'), f^\sharp (A))$$ holds for all objects $A',A$.

\medskip\noindent
Finally, the 2-dimensional aspect of the colimit~\eqref{eq:2-dim-aspect} is trivial, because $c_\X$ is the identity on objects. 

\medskip\noindent 
Hence we have proved that $H^\sharp \X = \Colim{N}{(H F\!_\X)}$. 
The action of 
$H^\sharp$ on 
$\V$-functors is as follows: 
For $f:\X \to \Y$, 
$H^\sharp f:H^\sharp \X \to H^\sharp \Y$ is the unique $\V$-functor satisfying 
\[
(H^\sharp f) c_\X = c_\Y (H f_\extra)
\]
Given that both $c_\X$ and $c_\Y$ are identity on objects, we see that $H^\sharp f$ and $H f_\extra$ coincide on objects.  
\end{proof}

\begin{cor}
\label{cor:Vcat-ifications_exist}
Every $T:\Set\to\Set$ has a $\Vcat$-ification $T_\V:\Vcat \to \Vcat$.
\end{cor}
\begin{proof}
Apply Theorem~\ref{thm:Lan} to the composite $H=D T:\Set\to\Vcat$.
\end{proof}


\subsection{When the unit of the left-Kan extension is an isomorphism}\label{sec:unit-kan-iso}

\

\bigskip\noindent
Recall from Proposition~\ref{prop:D-ff} that the discrete functor $D:\Set \to \Vcat$ is fully faithful (hence left Kan extensions along $D$ are genuine extensions) if and only if the quantale is integral. In particular, for $e<\top$, there is no guarantee that the unit $H\to (\mathsf{Lan}_D H)D$ of the left-Kan extension of $H$ along $D$ is an isomorphism.

\medskip\noindent
The reason why $D$ is not fully-faithful for non-integral $\V$ is that, for a discrete $\V$-category $\X=DX$, the relation $X_\top$ as defined in Proposition~\ref{prop:Vcoinserter} is empty in case $e<\top$, which in turn  forces the self-distances in  $(\mathsf{Lan}_{D}H)D\emptyset$ to be $\top$ whereas self-distances in $H\emptyset$ may be $<\top$.

\medskip\noindent
In this section we shall see necessary and sufficient conditions for $\Lan{D}H$ to coincide with $H$ on discrete $\V$-categories. 

\medskip\noindent
Recall that $\One$ is the one-element $\Vcat$ with self-distance $e$ and $\One_\top$ is the one-element $\Vcat$ with self-distance $\top$. We use the same notation to denote the respectice constant functors  $\One:\Set\to\Vcat$  and $\One_\top:\Vcat\to\Vcat$.
The main technical observation is contained in the following

\begin{exa}\label{exle:One}
The left Kan extension of the functor $\One:\Set\to\Vcat$ is the functor $\One_\top:\Vcat\to\Vcat$.
\end{exa}

\begin{defi}
A {\em constant} of a $\Vcat$-functor $K:\Vcat\to \Vcat$ is a $\Vcat$-natural transformation $\One_\top\to K$. A {\em constant} of a $\Vcat$-functor $H:\Set\to \Vcat$ is a $\Vcat$-natural transformation $\One\to H$. 
\end{defi}

\begin{rem} The constants of a functor $\Vcat\to\Vcat$ form a $\V$-category in which all self-distances are $\top$. --- 
Let $\tau:\One_\top\to K$ be a constant. Then each $\tau_\X: \One_\top \to K \X$ is a $\V$-functor, picking an object $\tau_\X(0)$ in $K \X$ with 
\[
\One_\top(0,0)=\top \leq K\X(\tau_\X(0),\tau_\X(0))
\]
hence we obtain the self-distance
\[
[\One_\top, K \X](\tau_\X, \tau_\X) = \top
\]
and consequently 
\[
[\Vcat,\Vcat](\One_\top,K)(\tau,\tau)=\bigwedge_{\X}[\One_\top, K \X](\tau_\X, \tau_\X) = \top
\]
\end{rem}

\begin{prop}\label{prop:constantsvcat} 
The $\V$-category of constants of a functor $\Vcat\to\Vcat$ is isomorphic to its value on the empty $\Vcat$.
\end{prop}

\begin{proof}
There is an isomorphism of $\V$-categories
$$
[\Vcat,\Vcat](\One_\top, K) \cong [\Set,\Vcat](\One, KD)\cong [\Set, \Vcat](\Set(\emptyset,-),KD)\cong KD\emptyset
$$
where the the first isomorphism is due to Example~\ref{exle:One} and the last isomorphism is the Yoneda lemma. 
\end{proof}

\noindent
The remark and the proposition combine to the result that, for any $\Vcat$-functor $K:\Vcat\to\Vcat$, the self-distances in $KD\emptyset$ must be $\top$. In particular, we have 

\begin{cor}
If $e<\top$, then there is no constant functor $\Vcat\to\Vcat$ with value $\One$, or with value $DX$ for any $X\in\Set$.
\end{cor}

\noindent More generally, what is at stake here is that $e$ is not a retract of $\top$, see the discussion towards the end of~\cite[Ch.3.9]{kelly:book}.

\medskip\noindent
Essentially the same argument as for Proposition~\ref{prop:constantsvcat} also proves

\begin{prop}\label{prop:constantsset} The $\V$-category of constants of a functor $\Set\to\Vcat$ is isomorphic to its value on the empty set.
\end{prop}

\noindent
To summarise, for any functor $K:\Vcat\to\Vcat$, the constants of $K$ are isomorphic to the constants of $KD$, and isomorphic to $KD\emptyset$. 
In particular, self-distances on elements of $KD\emptyset$ are $\top$.

\medskip\noindent 
Therefore,  in order to have $H \cong (\Lan{D}H)D$, self-distances in $H\emptyset$ must be $\top$. 
We are now going to show that this necessary condition is also sufficient.

\medskip\noindent An analysis of why, in Example~\ref{exle:One}, the left Kan extension of $\One$ along $D$ is $\One_\top$ but not $\One$, suggests that the reason for the failure of the unit being an isomorphism resides in the fact that  $\One$ is not a cocone over the $\V$-nerve $(N, \One)$ of $\One$. This leads to 

\begin{defi}
The functor $H:\Set \to \Vcat$ {\em preserves $\V$-nerves of sets} if the identity $\V$-functor on $HX$ induces a cocone of $HX$ over the diagram $(N, HF_{DX})$ for each $X$.  
\end{defi}

\begin{rem}\label{rem:pres-discr-nerves}
\hfill
\begin{enumerate}

\item $H$ preserves $\V$-nerves of sets iff  $\Id_{HX}$ satisfies~\eqref{eq:couniversal-property}, that is, iff
\[
r\leq \bigwedge_{C \in H(DX)_r} HX(Hd^r_0(C), Hd^r_1(C))
\]
for all $r\in\V$.
\item \label{rem:pres-discr-nerv-impl-unit-iso}
If $H$ preserves $\V$-nerves of sets, then by the universal property of the colimit there is a $\V$-functor $\beta_X:(\Lan{D}H) DX \to HX$ such that the identity of $HX$ factorises as 
\[
\xymatrix@C=35pt{HX \ar[r]^{\alpha_{X}} \ar[dr]_{\Id} & H^\sharp DX \ar[d]^-{\beta_X}\\
& HX}
\]
Since $\alpha_X$ is the identity on objects, $\beta_X$ will again be so. Moreover, because both $\alpha_X$ and $\beta_X$ are $\V$-functors, the inequalities 
\[
HX(A',A) \leq H^\sharp DX(A',A) \leq HX(A',A)
\]
hold for each $A',A$. Hence $(\Lan{D}H)  DX(A',A) = HX(A',A)$, that is, the unit of the left Kan extension $\alpha_X$ is an isomorphism (actually the identity). 
\end{enumerate}
\end{rem}

\begin{thm}\label{thm:unitiso}
Let $H:\Set \to \Vcat$ a functor.  
The following are equivalent.
\begin{enumerate}
\item
The unit of the left-Kan extension $H\to(\Lan{D}H)D$ is an isomorphism.
\item Self-distances on elements of $H\emptyset$ are $\top$.
\item All self-distances of all constants of $H$ are $\top$.
\item $H$ preserves $\V$-nerves of sets.
\end{enumerate}
These conditions always hold if $\V$ is integral, or, for general $\V$, if $H\emptyset=\emptyset$.
\end{thm}

\begin{proof}
For $(1)\Rightarrow (3)$ and  $(2)\Leftrightarrow (3)$ see  the comment after Proposition~\ref{prop:constantsset}. $(4)\Rightarrow (1)$  has been established in Remark~\ref{rem:pres-discr-nerves}. For $(2)\Rightarrow (4)$, let $X$ be an arbitrary set and assume that $(2)$ holds.
We will show that the identity of $HX$ satisfies~\eqref{eq:couniversal-property}
\[
r\leq \bigwedge_{C \in H(DX)_r} HX(Hd^r_0(C), Hd^r_1(C))
\]
for each $r$, hence $H$ preserves $\V$-nerves of sets by Remark~\ref{rem:pres-discr-nerves}.
We start by examining the the associated binary relations ($r$-level sets) of the (discrete) $\V$-category $DX$: 
\[
(DX)_r =  \begin{cases} X \times X & r=\bot \\ \Delta_X=\{(x,x)\mid x\in X\} & \bot <r\leq e \\ \emptyset & \mbox{otherwise}\end{cases}
\]
Condition~\eqref{eq:couniversal-property} is automatically satisfied for $r=\bot$. 

\noindent If $\bot<r\leq e$, then $(DX)_r$ is the diagonal of $X$ and $d_0^r=d^r_1$ is a bijection. Hence $Hd_0^r=Hd_1^r$, being an isomorphism of $\V$-categories, implies that
\[
r \leq e\leq \bigwedge_{C\in H\Delta_X} H\Delta_X(C,C)=\bigwedge_{C\in H(DX)_r} HX(Hd^r_0(C),Hd^r_1(C))
\]
proving that~\eqref{eq:couniversal-property} holds for $\bot<r\leq e$.

\noindent We will now consider the case $r\nleq e$. Then $(DX)_r=\emptyset$ and $d^r_0=d^r_1$ is the unique map $!:\emptyset \to X$. By hypothesis, $H\emptyset$ has all self-$\V$-distances equal to $\top$. Then
\[
r \leq \top = \bigwedge_{C\in H\emptyset}H\emptyset(C,C)\leq \bigwedge_{C \in H\emptyset} HX(H! (C), H! (C)) = \bigwedge_{C\in H(DX)_r} HX(Hd^r_0(C),Hd^r_1(C))
\]
which establishes \eqref{eq:couniversal-property} and, therefore, that $H$ preserves $\V$-nerves of sets.
\end{proof}


\subsection{Characterisation theorems and density} 

\ 

\bigskip\noindent
There is a well-known connection between left Kan extensions and colimits. For example, the finitary (that is, filtered colimit preserving) functors $\Set\to\Set$ are precisely those which are left Kan extensions of their restrictions along the inclusion of finite sets
\begin{equation}\label{eq:finitary}
\vcenter{
\xymatrix@R=20pt@C=32pt{
\Set \ar[r]^{} & \Set\\
\Fin \ar[ru]_{} \ar[u]& 
}}
\end{equation}
In this section the functors that arise as $\Lan{D}{H}$ 
\begin{equation}\label{eq:discretearities}
\vcenter{
\xymatrix{
\Vcat \ar[r]^{} & \Vcat\\
\Set \ar[ru]_(.58){H} \ar[u]^(.45)D& 
}}
\end{equation}
are characterised as those functors that preserve $\V$-nerves~\eqref{eq:Vnerve}. The proof proceeds by showing that the colimits exhibited in~\eqref{eq:Vnerve} and subsequently in Proposition~\ref{prop:Vcoinserter} constitute a density presentation of $D$~\cite[Section~5.4]{kelly:book}.

\medskip\noindent
The functors $\Set\to\Set$ that arise as left Kan extensions in diagram \eqref{eq:finitary} are precisely those that have a presentation by finitary operations and equations. Note that the elements $n\in\Fin$ appear as the arities $X^n\to X$ of the operations of the presentation. Similarly, we may speak of the functors that arise as left Kan extensions in diagram \eqref{eq:discretearities} as functors that have a \emph{presentation in discrete arities}, with the presentation given by the usual formula $(\Lan{D} H)\X=\int^S \X^{DS} \otimes HS$,%
\footnote{Note that although $\Set$ is not small, the coend does exist, by~\cite[Section~4.2]{kelly:book}.
} %
but which is more easily computed by applying $H$ to the $\V$-nerve as in Equation~\eqref{eq:H(Vnerve)}.

\medskip\noindent
A complication, studied in Proposition~\ref{prop:D-ff}, is that $D:\Set\to\Vcat$ is not fully faithful if $\V$ is not integral, that is, if $e<\top$. 

\medskip\noindent
Nevertheless we know from Corollary~\ref{cor:dense}, even without restricting to integral $\V$, that the discrete functor $\Set\to\Vcat$ is dense. Here we show that this result also follows for general reasons that do not depend on the particular presentation of $\V$-categories as colimits of $\V$-nerves.

\begin{prop}\label{prop:D-dense}
The  $\Vcat$-functor $D:\Set\to\Vcat$ is dense.
\end{prop}

\begin{proof} 
Let $\mathbb I$ denote the unit $\Vcat$-category (which has only one object $*$ and $\Vcat$-hom $\mathbb I (*,*) = \mathbb 1$). 
Consider the $\Vcat$-functors $F:\mathbb I \to \Set, \quad F(*)=\underline{1}$, where $\underline{1}$ denotes a one-element set, and $G:\mathbb I \to \Vcat, \quad G(*)=\mathbb 1$. Then the usual coend formula gives $\mathsf{Lan}_F G = D$, using that $\mathbb I$ is small and $\Vcat$ is cocomplete. Now by~\cite[(5.17)]{kelly:book}, $G$ is dense, and by~\cite[Proposition~5.10]{kelly:book}, $D$ is also dense, given that $F$ is fully faithful. 
\end{proof}

\medskip\noindent
The result on density presentations in \cite[Chapter 5]{kelly:book} requires Kan extensions along fully faithful functors. We will henceforth demand in the remaining of this Section that the quantale $\V$ is integral, that is, $e=\top$, see Proposition~\ref{prop:D-ff}.

\medskip\noindent
Recall from Proposition~\ref{prop:Vcoinserter} that every $\V$-category $\X$ is the colimit $N\ast DF_\X$ of the diagram $(N,DF_\X)$ of discrete $\V$-categories.

\begin{thm}\label{thm:dense-pres}
The $\Vcat$-functor $D :\Set\to\Vcat$ is dense and the diagrams $(N,DF_\X)$ form a density presentation of $D:\Set\to\Vcat$. %
\end{thm}

\begin{proof}
We know that $D$ is dense from Proposition~\ref{prop:D-dense}. Using Proposition~\ref{prop:Vcoinserter}, we see that all $\V$-coinserters $N\ast DF_\X$ exist in $\Vcat$ and the category
$\Vcat$ is the closure of $\Set$ under these colimits. %
By~\cite[Theorem~5.19]{kelly:book}, in order to establish the density presentation of $D$, it remains to show that the colimits 
$N\ast DF_\X$ are preserved by the functors 
$\Vcat(DS,-):\Vcat\to \Vcat$ for all sets $S$. 
Abbreviate as earlier $\Vcat(DS,\X)=\X^{DS}$. Going back to the notation of the proof of Proposition~\ref{prop:Vcoinserter}, we notice that $(DX_r)^{DS}=(\X^{DS})_r$ and $(DX_\extra)^{DS}=(\X^{DS})_\extra$, so that the diagram $(DF_\X)^{DS}$ 
\[
\xymatrix{
(DX_r)^{DS} \ar@<0.5ex>[rr]^{(Dd_1^r)^{DS}} \ar@<-0.5ex>[rr]_{(Dd_0^r)^{DS}} 
&& 
(DX_\extra)^{DS}}
\] 
is equally the diagram $DF_{\X^{DS}}$
\[
\xymatrix{
(\X^{DS})r \ar@<0.5ex>[rr]^{} \ar@<-0.5ex>[rr]_{} 
&& 
(\X^{DS})_\extra}.
\] 
This gives us, using Proposition~\ref{prop:Vcoinserter}, that
\begin{align*}
\Vcat(DS,N\ast DF_\X) & \cong \X^{DS} \\
& \cong N\ast DF_{\X^{DS}}\\
& \cong N\ast (DF_\X)^{DS}\\
& \cong N\ast \Vcat(DS,DF_\X)
\end{align*}
showing that $\Vcat(DS,-)$ preserves $N\ast DF_\X$.
\end{proof}

\noindent
We next characterise the $\Vcat$-functors with presentations in discrete arities, that is, the functors that are left Kan extensions of their restrictions along $D:\Set\to\Vcat$, as those functors that preserve colimits of $\V$-nerves.

\begin{thm}[Characterisation of $\Vcat$-functors with presentations in discrete arities]\label{thm:char-discrete-arities}
For $G:\Vcat\to\Vcat$ the following are equivalent:
\begin{enumerate}
\item
There exists a functor $H:\Set\to\Vcat$ such that
$G\cong \Lan{D}{H}$.
\item
$G$ preserves the colimits of
all diagrams $(N,DF_\X)$.
\end{enumerate}
\end{thm}

\begin{proof}
Since the $(N,DF_\X)$ form a density presentation (Theorem~\ref{thm:dense-pres}), this is the equivalence of items~(i) and~(iii) of \cite[Theorem~5.29]{kelly:book}.
\end{proof}

\noindent $\Vcat$-ifications are similarly characterised but additionally preserve discreteness.

\begin{thm}[Characterisation of $\Vcat$-ifications]\label{thm:char-Vcatification}
For $G:\Vcat\to\Vcat$ the following are equivalent.
\begin{enumerate}
\item
There exists a functor $T:\Set\to\Set$ such that
$G\cong \Lan{D}{(DT)}$, that is, $G$ is \hspace{.2em}$T_\V$, the $\Vcat$-ification of $T$.
\item
$G$ preserves the colimits of
all diagrams $(N,DF_\X)$, and $G$ preserves discrete $\V$-categories.
\end{enumerate}
\end{thm}

\begin{proof}
We use Theorem~\ref{thm:char-discrete-arities} and observe that $G$ mapping discrete $\V$-categories to discrete $\V$-categories is equivalent to $GD$ factoring as $DT$ for some $T$.
\end{proof}

\begin{exa}
\label{ex:non-Vcatif}
We will now give an example of a $\Vcat$-functor $G:\Vcat \to \Vcat$ which is not a $\Vcat$-ification, even though it preserves discreteness of $\V$-categories. $G$ is presented in non-discrete arities as follows: For a $\V$-category $\X$, put $G\X = \X^{\Two_r}$, where the $\V$-category $\Two_r$ was introduced in~\eqref{eq:Two_r}, for $r$ an element of the quantale. 
Notice that $G$ is the identity on discrete $\V$-categories, for every $\bot<r\leq e$. 
Hence, if $G$ was $\Lan{D}(DT)$ for some $T:\Set \to \Set$, we could choose $T$ to be the identity and then $G$ would also have to be the identity, by Theorem~\ref{thm:dense-pres} and~\cite[Theorem~5.1]{kelly:book}.
\end{exa}


\subsection{Extending polynomial and finitary functors}\label{sec:examples}

\

\bigskip\noindent
Let now $\V$ be an arbitrary commutative quantale. We show that functors defined by finite powers and colimits extend from $\Set$ to $\Vcat$ via their  universal properties. In particular, the $\Vcat$-ification of a finitary functor is presented by the same operations and equations as its underlying $\Set$-functor.

\begin{exas}[\bf The $\Vcat$-ification of polynomial functors]
\label{exs:polynomial_functors}
\hfill
\begin{enumerate}
\item \label{ex:Vcatif-ct}
Let $T:\Set\to\Set$, be a constant functor at a set $S$. Then $T_\V$ acts as follows: For any $\V$-category $\X$, $T_\V\X$ is the constant $\Vcat$-functor to the $\V$-category with $S$ as set of objects, with $\V$-distances
\[
T_\V\X(x',x) = 
\begin{cases} 
\, \top \ \ , \ \ x'=x
\\
\bot \ \ , \ \ \mbox{otherwise}
\end{cases}
\]
That is, $T_\V \X = S \cdot \One_\top$ is the coproduct in $\Vcat$ of $S$ copies of the terminal $\V$-category $\One_\top$.  
Observe that in case $\V$ is integral, we obtain $T_\V \X=D S$ for any $\V$-category $\X$.

\item
Let $T:\Set\to\Set$ be the functor $TX= X^n$, for $n$ a natural number. 
Then $T_\V$ maps a $\V$-category $\X$ to its $n$-th power $\X^n$, where an easy computation shows that
\begin{equation}\label{eq:V-dist-tuples}
\X^n((x'_0,\dots,x'_{n-1}),(x_0,\dots,x_{n-1}))
=
\X(x'_0,x_0)\wedge\dots\wedge\X(x'_{n-1},x_{n-1}).
\end{equation}

\item \label{ex:Vcatif-power}
If $n$ is an {\em arbitrary\/} cardinal number, the $\Vcat$-ification $T_\V$ of 
$T:\Set\to \Set$, $TX=X^n$ also exists and $T_\V\X((x'_i),(x_i))=\bigwedge_i \X (x'_i,x_i)$. That is, $T_\V\X=\X^n$.
\item
The $\Vcat$-ification of a finitary polynomial functor
$
X\mapsto \coprod_n X^n\Sum \Sigma n
$
is the ``strongly polynomial'' $\Vcat$-functor
$
\X\mapsto \coprod_n \X^n \tensor D\Sigma n
$, where $n$ ranges through finite sets.
In particular, the $\Vcat$-ification of the list functor $LX = X^* = \coprod_n X^n$ maps a $\V$-category $\X$ to the $\V$-category having as objects tuples of objects of $\X$, with non-trivial $\V$-distances only between tuples of same order, given by~\eqref{eq:V-dist-tuples}. 
\end{enumerate}
\end{exas}

\smallskip\noindent %
Our next aim is to show that quotients of polynomial functors can be $\Vcat$-ified  by taking the ``same'' quotients in $\Vcat$.

\begin{rem}
\label{rem:VV}
The $\Vcat$-functor $D:\Set\to\Vcat$ preserves conical 
colimits. 
This follows from the $D _o$ being an \emph{ordinary} left adjoint. 
However, the $\Vcat$-functor $D:\Set\to\Vcat$ is not
a left \emph{$\Vcat$-adjoint}, as its ordinary right adjoint functor $V$ cannot be extended to a $\Vcat$-functor.
\end{rem}

\begin{prop}
\label{prop:Vcatification=functor}
The assignment $({-})_\V:[\Set,\Set]\to [\Vcat,\Vcat]$, $T\mapsto T_\V$ of the $\Vcat$-ification
preserves all colimits preserved by $D:\Set\to\Vcat$. %
In particular, $T\mapsto T_\V$ preserves conical colimits.
\end{prop}
\begin{proof}
Any natural transformation $\tau:T\to S$ induces a $\Vcat$-natural transformation
$$
(\tau_\V)_\X
=
\Colim{N}{(D\tau F\!_\X)}
:
\Colim{N}{(D T F\!_\X)}
\to
\Colim{N}{(D S F\!_\X)}
$$ 
Since any colimit is cocontinuous in its weight and since 
$$
\Colim{N}{(D T F\!_\X)}
\cong
\Colim{(D T F\!_\X)}{N}
$$
holds~\cite[Section 3.1]{kelly:book}, the assignment $T\mapsto T_\V$ preserves all colimits that are preserved by $D:\Set\to\Vcat$. 
The last statement follows from Remark~\ref{rem:VV}.
\end{proof}

\begin{cor}
\label{cor:presentation}
Suppose that the coequalizer
$$
\xymatrix{
T_\Gamma
\ar@<.5ex>[0,1]^-{\lambda}
\ar@<-.5ex>[0,1]_-{\rho}
&
T_\Sigma
\ar[0,1]^-{\gamma}
&
T
}
$$
is the equational presentation of an accessible functor $T:\Set\to\Set$. 
Then the $\Vcat$-ification $T_\V$ can be obtained as the coequalizer
$$
\xymatrix{
(T_\Gamma)_\V
\ar@<.5ex>[0,1]^-{\lambda_\V}
\ar@<-.5ex>[0,1]_-{\rho_\V}
&
(T_\Sigma)_\V
\ar[0,1]^-{\gamma_\V}
&
T_\V
}
$$  
in $[\Vcat,\Vcat]$.
\end{cor}

\begin{rem}[\bf The $\Vcat$-ification of finitary functors]
Corollary~\ref{cor:presentation} allows us to say that the $\Vcat$-ification $T_\V$ of a finitary functor $T$ presented by operations and equations is given by imposing the ``same'' operations and equations in $\Vcat$.%
\footnote{
This explains~\cite[Remark~6.5(1)]{ammu:finitary-functors}.
} %
This gives us an alternative way of computing the $\Vcat$-ification of a finitary (or accessible) functor $T$: 
First, one presents $T$ as the quotient of a polynomial functor. 
Then one extends the polynomial functor to $\Vcat$ as in Example~\ref{exs:polynomial_functors}. 
Finally, one computes the relevant coequalizers to take a quotient in $\Vcat$ (see~\cite[Proposition~2.11]{wolff:Vcat&Vgraph} on coequalizers in $\Vcat$).
\end{rem}

\begin{exa}
The finite powerset functor is a quotient of the list functor by the familiar equations stating that sets are invariant under the order and multiplicity of elements. It is shown in \cite{VelebilK11} that adding to these equations the requirement that lists are monotone, one obtains the finite convex powerset functor, which is indeed the posetification of the finite powerset functor \cite{lmcs:bkv}.
\end{exa}

\subsection{Extending functors via relation lifting}\label{sec:wp}

\ 

\bigskip\noindent
Extending set-functors to various (ordered) categories using the so-called relation lifting is an established topic of research going back to~\cite{barr} and \cite{trnkova}. Using the representation of quantale-enriched categories as relational presheaves of Section~\ref{sec:relpresh}, we will see here that this idea can also be applied to the extension of $\Set$-functors to $\Vcat$ and, moreover, that $\Vcat$-ifications of weak-pullback preserving functors are more easily computed in this way. 
In particular, \eqref{eq:wpb} simplifies \eqref{eq:shortest-path} significantly. As an application, we will compute the extension of the multiset functor and show that the well-known Pompeiu-Hausdorff metric arises as a $\Vcat$-ification of the powerset functor.

\medskip\noindent
As in ~\cite{wagner}, 
``what we do in effect is to reason about metric spaces (or other structures) in a universe where they {\em are} pre-orders, viz.\ in sheaves over $[0, \infty]$ (or other $\V$).'' 
This point of view is summarised by Theorem~\ref{thm:wpbcontrelpresh} stating mild conditions under which  the $\Vcat$-ification is obtained by applying the relation lifting to these preorders.

\begin{rem}[Relation lifting, see for example~\cite{barr,trnkova,freyd-scedrov:allegories,HermidaJ98,kv:rellift}] 
Let $\Rel$ be the 2-category of sets with relations ordered by inclusion as arrows and $(-)_\diamond:\Set\to\Rel$ the functor mapping a function to its graph. Given a functor $T:\Set\to\Set$ the relation lifting $\mathsf{Rel}_T:\Rel\to\Rel$ extends $T$ to a colax functor
$$
\xymatrix{
\Rel
\ar[0,2]^-{\mathsf{Rel}_T}
&
&
\Rel
\\
\Set
\ar[0,2]_-{T}
\ar[-1,0]^{(-)_\diamond}
&
&
\Set
\ar[-1,0]_{(-)_\diamond}
}
$$
and is defined as follows. Let $p:Z\to X$ and $q:Z\to Y$ two arrows (called a span from $X$ to $Y$). A \emph{span $(p,q)$ represents the relation}
\begin{equation}\label{eq:relofspan}
q_\diamond \cdot (p_\diamond)^{-1} = \{(x,y) \mid \mbox{ there is $z\in Z$ s.t. $x=p(z)$ and $q(z)=y$}\}
\end{equation}
 where $\cdot$ is composition and $^{-1}$ is relational inverse. 
 Composition of two relations is represented by the pullback of their spans. 
 Let $R$ be a relation represented by the span $(p,q)$. 
 Then \emph{$\mathsf{Rel}_T(R)$ is the relation represented by $(Tp,Tq)$}. 
 This definition is independent of the choice of span. 
 In the following we only need to know that $\mathsf{Rel}_T$ is monotone, preserves identities, and that it preserves function composition iff $T$ preserves weak pullbacks. \qed
\end{rem} 

\medskip\noindent The following lemma considerably simplifies the formula in Corollary~\ref{cor:Lan by zigzag paths} for computing the left-Kan extension $\Lan{D}{H}$ in case $H=DT$ with $T$ preserving weak pullbacks.
 
\begin{lem}\label{lem:wpb}
Let $T:\Set\to\Set$ be a weak pullback preserving functor. Then the $\Vcat$-ification of $T$ is given by
\begin{equation}\label{eq:wpb}
T_\V\X(A',A)
=
\bigvee_s \{ s\mid (A',A)\in \mathsf{Rel}_T(X_s) \}
\end{equation}
\end{lem}

\begin{proof}
To show ``$\ge$'' assume $(A',A)\in \mathsf{Rel}_T(X_s)$.  
Then there is some $C\in T(X_s)$ such that $Td_0^s(C)=A'$ and $Td_1^s(C)=A$. Since this is a (very special) zig-zag of the form of Corollary~\ref{cor:Lan by zigzag paths} it follows that $T_\V\X(A',A) \ge s$. 

\medskip\noindent
Conversely, for ``$\le$'', consider a zig-zag 
\begin{equation*}
(A',A_0, C_1\in DTX_{r_1}, A'_1, A_2, C_1\in DTX_{r_2},\ldots , A_n')
\end{equation*}
as in Corollary~\ref{cor:Lan by zigzag paths}, contributing with
$$r=r_1\otimes\ldots r_n$$ to the distance $T_\V\X(A',A)$. 
We want to see that 
$$(A',A)\in \mathsf{Rel}_T(X_r).$$ 
Note that $A'=A_0$ and  $A'_i=A_{i+1}$, $0<i<n$, and $A'_n=A$, so that we are looking at the zig-zag
\begin{equation}\label{eq:zigzag}
\vcenter{
\xymatrix@C=0pt@R=45pt{
&
DT X_{r_1}
\ar[dl]_{DTd^{r_1}_0}
\ar[dr]^{}
&
&
DTX_{r_2}
\ar[dl]_{}
\ar[dr]^{}
&
&
\quad \cdots \quad 
\ar[dl]_{\cdots}
\ar[dr]^{\cdots}
&
&
DTX_{r_n}
\ar[dl]_{}
\ar[dr]^{DTd^{r_n}_1}
&
&
\\
DTX
&
&
DTX
&
&
DTX
&
&
DTX
&
&
DTX
}}
\end{equation}
Consequently, $(A',A)$ is in the relation represented by the zig-zag~\eqref{eq:zigzag}, that is, in 
$$\mathsf{Rel}_T (X_{r_1})\cdot \mathsf{Rel}_T (X_{r_2})\cdot\ldots \cdot \mathsf{Rel}_T (X_{r_n})$$
We calculate
\begin{align*}
\mathsf{Rel}_T (X_{r_1})\cdot \mathsf{Rel}_T (X_{r_2})\cdot\ldots \cdot \mathsf{Rel}_T (X_{r_n}) 
& = 
\mathsf{Rel}_T(X_{r_1}\cdot X_{r_2}\cdot\ldots \cdot X_{r_n}) \\
& \subseteq 
\mathsf{Rel}_T(X_{r_1\otimes r_2\otimes\ldots \otimes r_n})\\
& =\mathsf{Rel}_T(X_r)
\end{align*}
where the first step is due to $T$ preserving weak pullbacks and the second step is due to $X_\_$ being lax monoidal, see \eqref{eq:monoidal2}.
\end{proof}

\begin{exa}[\bf The $\Vcat$-ification of the multiset functor]
The (finitary) multiset functor $M:\Set\to \Set$ can be presented as $MX = \coprod_{n}X^n/S_n$, where $S_n$ denotes the full $n$-permutation group~\cite{hasegawa}. 
Denote by $\widehat{[x_1, \ldots, x_n]}$ the equivalence class (multiset) of the tuple $(x_1, . . . , x_n)$  under the action of $S_n$. 

\noindent By Lemma~\ref{lem:wpb} and Equation~\eqref{eq:relofspan}, the $\Vcat$-ification $M_\V$ admits the following description: it maps a $\V$-category $\X$ to the $\V$-category having as objects multisets of objects of $\X$ -- formally, $(M_\V \X)_\extra = M X_\extra$, 
and the only possible non-trivial $\V$-distances are between multisets of same cardinal, namely
\begin{alignat*}{4}
M_\V \X(\widehat{[x'_1,\ldots,x'_{n}]},\widehat{[x_1, \ldots, x_{n}]}) 
&
= 
&&
\bigvee_s \{ s \mid \mbox{ there is } C \in MX_s \mbox{ s.t. } Md_0^s(C) = \widehat{[x'_1,\ldots,x'_{n}]} 
\\
&&&
\mbox{ and } Md_1^s(C) = \widehat{[x_1,\ldots,x_{n}]}\, \}
\\
&
=
&&
\bigvee_s 
\{
s \mid 
\mbox{ there is }\sigma\in S_n \mbox{ s.t. } s\leq \X(x'_i,x_{\sigma(i)}) 
\\
&&&
\mbox{ for all } i=1,n \,  \} 
\\
&
=
&&
\bigvee_{\sigma\in S_n} \bigvee_s \{s \mid s\leq \X(x'_i,x_{\sigma(i)}) \mbox{ for all } i=1,n \,  \}
\\
&
=
&& 
\bigvee_{\sigma\in S_n} \bigwedge_i \X(x'_i, x_{\sigma(i)}) 
\end{alignat*}

\noindent The reader might notice that for $\V=([0,\infty],\geq_{\mathbb R} +,0)$, we regain by the above formula the matching metric on multisets~\cite{dd:encycl-dist} employed in image recognition techniques. 
Also, it generalises~\cite{dahlqvist-kurz:calco17}, where the case $\V=\Two$ (that is, $\Vcat=\Pre$) was considered, in the sense that given a (pre)ordered set $(X,\leq)$, multisets on $X$ are ordered by
\[
\widehat{[x'_1,\ldots,x'_{n}]} \leq \widehat{[x_1, \ldots, x_{n}]} \Longleftrightarrow \mbox{ there is } \sigma\in S_n \mbox{ s.t. } x'_i \leq x_{\sigma(i)} \mbox{ for all } i=1,n
\]
In fact, \cite{dahlqvist-kurz:calco17} corroborated with Lemma~\ref{lem:wpb} and \eqref{eq:relofspan} gives us the recipe to compute the $\Vcat$-ification of any analytic $\Set$-functor,%
\footnote{
Analytic functors $\Set \to \Set$ preserve weak pullbacks~\cite[Lemma~1.8]{hasegawa}.
} %
not just of the multiset functor. 
\end{exa}

\begin{exa}[The $\Vcat$-ification of the powerset]\label{ex:Vcatif-powerset}
Let $\mathcal P:\Set\to\Set$ be the powerset functor. 
Using Lemma~\ref{lem:wpb} and Equation~\eqref{eq:relofspan}, the $\Vcat$-ification $\mathcal P_\V$ is described as follows. 
Let $\X$ be any small $\V$-category.
Then the objects of $\mathcal P_\V\X$ are subsets of the set of objects of $\X$, while the $\V$-``distances'' in $\mathcal P_\V \X$ are computed as follows:
\begin{alignat*}{4}
\mathcal P_\V\X(A',A)
&
=
&&
\bigvee_s %
\{ %
s\mid \mbox{ there is $C$ in $\mathcal PX_s$ s.t. $\mathcal Pd_0^s(C)=A'$ and $\mathcal Pd_1^s(C)=A$} %
\} 
\\
&
=
&&
\bigvee_s \{ s \mid 
\forall \, x'\in A' \, \exists \, x\in A . \, s \leq \X(x',x)
\mbox{ and } 
\forall \, x\in A \, \exists\, x'\in A' . \, s \leq \X(x',x)
\}
\end{alignat*}
Remembering that the join in $\V$ denotes the infimum over all distances, we recognise the familiar Pompeiu-Hausdorff metric as seen in~\cite[Definition~2.2]{debakker-devink}, now generalised to categories enriched over an arbitrary commutative quantales $\V$. We will discuss the relationship to the more customary formulation in the following.
\qed
\end{exa}

\medskip\noindent Recall from~\cite{fawcett-wood:ccd,raney} that a quantale (more generally, a poset) $\V$ is called {\em completely distributive} if the function from the poset of downsets of $\V$ to $\V$, $\frak D \mapsto \bigvee \frak{D}$, has a left adjoint. This left adjoint maps each $r\in \V$ to the downset $\{s \mid s\ll r\}$ of all elements totally-below it, where the ``totally-below" relation is  
$$
s\ll r \quad \Longleftrightarrow \quad (\,  \forall \, \frak{D}\subseteq \V \mbox{ downset such that} \, \, r \leq \bigvee \frak{D}  \Rightarrow \, s \in \frak{D}\, )
$$
Examples of completely distributive quantales are $\V=[0,1]$ and $\V=[0,\infty]$. 
For such a quantale, we provide below another characterisation of the $\Vcat$-ification of the powerset functor:

\begin{prop}[The $\Vcat$-ification of the powerset, continued]\label{prop}
If the quantale $\V$ is completely distributive, then the $\Vcat$-ification of the powerset functor can be described as follows: 
for a $\V$-category $\X$, $\mathcal P_\V\X$ is the $\V$-category with objects $\mathcal PX_\extra$, and $\V$-homs
\begin{equation}\label{Vcat-powerset}
\mathcal P_\V \X(A', A) = \big(\bigwedge_{x'\in A'} \bigvee_{x\in A} \X(x',x) \big) \bigwedge \big(\bigwedge_{x\in A} \bigvee_{x'\in A'} \X(x',x)\big)
\end{equation}
\end{prop}

\noindent We will first need a lemma: 
\begin{lem}\label{lemma}
Let $r\in \V$ and $S\subseteq \V$. Then $r\ll \bigvee S$ if and only if there is some $s\in S$ with $r\ll s$. 
\end{lem}

\begin{proof}[Proof of Lemma~\ref{lemma}]
The ``if'' implication is immediate. To see the other implication, assume that for all $s\in S$, $r\not\ll s$. That is, for each $s\in S$ there is a downset $\frak{D}_s$ with $s\leq \bigvee \frak{D}_s$ and $r \not\in \frak{D}_s$. But then $\bigcup_{s\in S} \frak{D}_s$ is again a downset and 
$$
\bigvee S \leq \bigvee \bigcup_{s\in S} \frak{D}_s \quad \mbox{ but } r \not \in \bigcup_{s\in S} \frak{D}_s 
$$
which contradicts $r \ll \bigvee S$. 
\end{proof}

\medskip\noindent We are now ready to prove Proposition~\ref{prop}.
\begin{proof}[Proof of Proposition~\ref{prop}]
Remember that in~Example~\ref{ex:Vcatif-powerset}, we have obtained that 
\begin{alignat*}{3}
\mathcal P_\V \X(A',A) 
&=&
\bigvee_s \{ s\mid 
\mbox{%
$\forall\, x'\in A' \ \exists\, x\in A.\, s\leq\X(x',x)$ 
and 
$\forall\, x\in A\ \exists\, x'\in A'.\, s\leq\X(x',x)$%
}
\}
\end{alignat*}
Denote for simplicity 
$$
S=\{ s\mid 
\mbox{
$\forall\, x'\in A' \ \exists\, x\in A.\, s\leq\X(x',x)$ 
and 
$\forall\, x\in A\ \exists\, x'\in A'.\, s\leq\X(x',x)$
}
\}
$$
and notice that $S$ is a downset; %
also put
$$
r=\big(\bigwedge_{x'\in A'} \bigvee_{x\in A} \X(x',x) \big) \bigwedge \big(\bigwedge_{x\in A} \bigvee_{x'\in A'} \X(x',x)\big)
$$
Thus we want to see that $r=\displaystyle{\bigvee} S$. The inequality ``$\geq$'' follows easily with no assumption on the quantale $\V$. For the other inequality, use complete distributivity to see that 
$$
r
\leq
\bigvee S
$$
is equivalent with 
$$
\{s \mid s\ll r\} \subseteq S
$$
Let now $s\ll r$. 
Then for each $x'\in A'$, it follows that $s \ll \bigvee_{x\in A} \X(x',x)$, thus by the above lemma we find some $x\in A$ with $s \ll \X(x',x)$, which implies $s \leq \X(x',x)$. Similarly, for all $x\in A$ there is $x'\in A'$ with $s \leq \X(x',x)$. 
Consequently, $s\in S$ and the proof is finished. 
\end{proof}

\begin{rem}
\begin{enumerate}
\item 
Let us switch notation to the dual order (that is, the natural ``less-or-equal'' order in case of the reals). 
So we write $\inf$ for $\bigvee$ and $\sup$ for $\bigwedge$, in order to emphasise the interpretation of $\Vcat$'s as metric spaces. 
Then formula~\eqref{Vcat-powerset} becomes 
\begin{equation}\label{eq:Pompeiu-Hausdorff}
\sup \{ \ \sup_{x'\in A'}\ \underset{x\in A}{\vphantom{p}\inf\vphantom{p}} \ \X(x',x) 
\ , \ 
\sup_{x\in A}\ \underset{x'\in A'}{\vphantom{p}\inf\vphantom{p}} \ \X(x',x)
\}
\end{equation}
As mentioned earlier, this metric given by Equation~\eqref{eq:Pompeiu-Hausdorff} above, is known as the Pompeiu-Hausdorff metric  (\cite[\textsection~28]{Hausdorff},~\cite[\textsection~21]{Pompeiu},~\cite[\textsection~21.VII]{kuratowski}). 
We should mention also the connection with the work of~\cite{akvhlediani+clementino+tholen:hausdorff}. 
Finally, observe that in case $\V=\Two$ (so $\Vcat=\Pre$), the above specialises to the locally monotone functor $\mathcal P_\Two:\Pre\to\Pre$ which sends a preorder $(X,\leq)$ to the Egli-Milner preorder 
$$
A'\sqsubseteq A
\quad
\mbox{iff}
\quad
\forall\, x'\in A' \ \exists\, x\in A.\, x'\leq x
\mbox{ and }
\forall\, x'\in A \ \exists\, x\in A'.\, x'\leq x
$$
on the powerset $\mathcal PX$.

\item Notice that the $\Vcat$-ification of the powerset functor obtained in Proposition~\ref{prop} for completely distributive quantale is the (symmetrised version of the) free cocompletion monad with respect to (the saturated class of) conical colimits in $\Vcat$~\cite{stubbe:hausdorff}. 
\end{enumerate}
\end{rem}

\medskip\noindent 
Example~\ref{ex:Vcatif-powerset} raises the question whether the relations $(T_\V\X)_r$ can be computed simply by applying the relation lifting $\mathsf{Rel}_T$ to  the $X_r$. 
The next example shows that we need to be careful with this.

\begin{exa}
Let $\X$ be the metric space of real numbers considered as a generalised metric space over $\V=[0, \infty]$. Let $A'$ be the subset of irrational reals and $A$ the subset of rational reals. 
Then $\mathcal P_\V\X(A',A)=0$ but $(A',A)\notin \mathsf{Rel}_\mathcal P(X_0)$. 
Indeed, $\mathsf{Rel}_\mathcal P(X_0)$ is the diagonal, since $X_0$ is the diagonal.
\end{exa}

\noindent
In terms of Section~\ref{sec:relpresh}, the example above shows that the relational presheaf
\begin{align*}
(\Sigma\V)^\coop & \to \ \Rel\\
r \ \ \ & \ \mapsto \ \mathsf{Rel}_\mathcal P(X_r)
\end{align*}
is not continuous as we have $(A',A)\in \mathsf{Rel}_\mathcal P(X_\varepsilon), \varepsilon >_{\mathbb R} 0$, but  $(A',A)\notin \mathsf{Rel}_\mathcal P(X_0)$. 
We can repair this defect by closing the presheaf as follows. 
Given a relational presheaf $F:(\Sigma\V)^\coop\to\Rel$, and assuming that $\V$ is completely distributive, define its closure as 
\begin{align*}
\mathcal CF: (\Sigma\V)^\coop & \to \ \Rel
\\
* \ \ \ & \ \mapsto F(*)
\\
r \ \ \ & \ \mapsto \ \bigcap\{F(s) \mid s \ll r\}
\end{align*}
It is immediate from the respective definitions that $\mathcal CF$ is a continuous relational presheaf, assuming the tensor laxly preserves the totally-below relation, in the sense that $r\ll s$ and $r'\ll s'$ imply $s \tensor s'\ll r \tensor r'$. 
Also, that $\mathcal C(\mathcal C F) = \mathcal CF$. 
That is, $\mathcal C$ acts like a closure operator on relational presheaves. 
This construct extends to a functor $\mathcal C:\RelPresh \to \RelPresh_c$.

\medskip\noindent
The idea that the $\Vcat$-ification of a weak pullback preserving $\Set$-functor $T$ can be computed on a $\V$-category $\X$ by applying the relation lifting of $T$ to the ``uniformity'' $(X_r)_{r\in\V}$ of $\X$ can now be formalised as in  next result. 
It is a consequence of Proposition~\ref{prop:relpresh} and of Lemma~\ref{lem:wpb}.

\begin{thm}\label{thm:wpbcontrelpresh}
Suppose $\V$ is a completely distributive quantale, such that the tensor laxly preserves the totally-below relation. 
Let $T:\Set\to\Set$ be a functor which preserves weak pullbacks. 
Then post-composition with $\mathsf{Rel}_T:\Rel\to\Rel$ yields a functor
$$
\mathsf{Rel}_T \circ (-):
\RelPresh_c\to\RelPresh
$$
such that the diagram
$$
\xymatrix{
\RelPresh_c \ar[0,2]^-{\mathsf{Rel}_T \circ (-)}
&&
\RelPresh \ar[0,2]^-{\mathcal C}
&&
\RelPresh_c
\\
\Set
\ar[0,2]^-{T} \ar[-1,0]^{D}
&&
\Set \ar[-1,0]^{D}  \ar[0,2]^-{\Id}
&&
\Set \ar[-1,0]_{D}
}
$$
commutes. 
Moreover, the composite in the upper row is the $\Vcat$-ification of $T$.
\end{thm}

\begin{proof}
Let $\X$ be a $\V$-category and let $\Phi(\X)$ be its continuous relational presheaf. We compute 
\begin{align*}
(A',A)\in (T_\V\X)_r 
& \ \Longleftrightarrow \ 
r \leq \bigvee_s \{ s\mid (A',A)\in \mathsf{Rel}_T(X_s) \} 
\\
& \ \Longleftrightarrow \ 
\forall s \, .\, s \ll r \, \Longrightarrow \, (A',A)\in \mathsf{Rel}_T(X_s)\\
& \ \Longleftrightarrow \ 
(A',A)\in\bigcap \{ \mathsf{Rel}_T(X_s) \mid s \ll r \} \\
& \ \Longleftrightarrow \ 
(A',A)\in \mathcal C (\mathsf{Rel}_T\circ \Phi(\X))(r)
\end{align*}
where the respective steps are due to the following. (i) Lemma~\ref{lem:wpb}. (ii) For all downsets $\frak{D}$ in a completely distributive $\V$ one has  $r \leq \bigvee \frak{D}  \Longleftrightarrow \forall s(s\ll r \Longrightarrow s\in \frak{D})$. (iii) Definition of $\bigcap$. (iv) Definition of $\mathcal C$.
\end{proof}


\subsection{Extending $\Set$-functors equipped with a $\V$-metric} \label{sec:Vmetric-functors} 

\

\bigskip\noindent
In Section~\ref{sec:beh-cat}, we will see that the behaviour of coalgebras for a $\Vcat$-ification is metrically trivial. 
But often $\Set$-functors carry a natural metric and thus can be considered to be functors $\Set\to\Vcat$, which in turn produce $\Vcat$-endofunctors with non-trivial metric coalgebraic behaviour.

\begin{exa}\label{ex:functors-with-V-metric}

\begin{enumerate}
\item 
Consider $\V=\Two$ and let $H:\Set\to\Pre$ be the powerset functor with the usual inclusion relation on subsets $HX=(\mathcal P(X),\subseteq)$. Then by~\cite{calco2011:balan-kurz}, the extension $H^\sharp$ to $\Vcat=\Pre$ maps a (pre)ordered set $(X,\leq)$ to $\mathcal P(X)$, ordered with the lower half of the Egli-Milner order on subsets, namely
\[
A' \leq A \iff \forall \, x'\in A' \,.\, \exists \, x\in A \,.\, x'\le x 
\]
for all $A',A\in \mathcal P(X)$. 

\item \label{ex:machine}
Consider the functor $T:\Set\to \Set$, $TX=X^{A}\times B$. 
Recall that $T$-coalgebras are known as deterministic automata with input set $A$ and outputs in $B$: 
the transition map of a coalgebra $X\longrightarrow X^A\times B$ sends a state $x\in X$ to a pair: 
the associated next-state function $next(x):A\longrightarrow X$ that works when receiving an input from $A$, and an output $out(x)$ in $B$. 

\medskip\noindent
Assume that the output set $B$ carries the additional structure of a $\V$-category -- that, is, there is a $\V$-category $\B$ with underlying set of objects $B$. 
This is the case for example whenever there is a natural order on $B$, or a metric, which measures how far apart two outputs might be. 

\medskip\noindent
Then the ordinary functor $T_o$ can be written as the composite $V H_o$, where $H:\Set \to \Vcat$ is the $\Vcat$-functor $HX= D X^A \tensor \B$. 
Now it is immediate to see that the latter extends to the functor $H^\sharp$ on $\Vcat$ over the ``generalised metric space'' $\B$, mapping a $\V$-category $\X$ to the tensor product of $\V$-categories $H^\sharp \X =  \X^A \tensor \B$, where $\X^A= \Vcat(DA,\X)$.

\medskip\noindent
There are two particular cases of this functor that are worth mentioning from a coalgebraic point of view:
\begin{enumerate}
\item \label{ex:stream}
We start with the easiest case, when $A$ is a singleton. 
The coalgebras for the functor $TX=X\times B$ are usually called stream automata; 
under the assumption that $B$ carries the structure of a $\V$-category $\B$, we obtain as above $H^\sharp \X = \X \otimes \B$ the stream functor on $\V$-metric spaces.

\item \label{ex:det-autom} 
If we take an arbitrary set of inputs $A$, but particularise $B$ to $\{0,1\}$, then the transition function of a coalgebra  
provides binary outputs, deciding if a state is accepting (response $1$) or not -- that is, $T$-coalgebras are deterministic automata. 
Enhance now the set of outputs with a $\V$-category structure $\Two_{r,s}$ generalising~\eqref{eq:Two_r}
\[
\Two_{r,s}(0,0)=\Two_{r,s}(1,1)=e\, , \, \Two_{r,s}(0,1)=r\, , \, \Two_{r,s}(1,0)=s
\]   
where the elements $r,s$ of the quantale satisfy $r\otimes s \leq e$, in order to produce the $\Vcat$-endofunctor $H^\sharp \X = \X^A \otimes \Two_{r,s}$. 

\end{enumerate}
\end{enumerate}
\end{exa}

\medskip\noindent
The above examples are typical. 
It happens quite often for endofunctors $T$ on $\Set$ to carry an interesting $\V$-metric. Then $TX$ is a $\V$-category rather than a mere set.
The following generalises the notion of an order on a functor \cite{hugh-jaco:simulation} from $\Two$ to an arbitrary quantale $\V$.

\begin{defi}
Let $T:\Set \to \Set$ be a functor. We say that $T$ {\em carries a $\V\hspace{-1.5pt}$-metric\/} if there is a $\Vcat$-functor $H:\Set \to \Vcat$ such that $T_o$ coincides with the composite 
$$
\xymatrix{\Set_o \ar[r]^-{H_o}& \Vcat_o \ar[r]^-{V} & \Set_o}.
$$
\end{defi}

\medskip

\noindent
Let $T$ and $H$ be as in the above definition. 
How are $T$ and $H^\sharp$, the left Kan extension of $H$ along $D$ as provided by Theorem~\ref{thm:Lan}, related?

\medskip\noindent Recall from Theorem~\ref{thm:Lan} and Corollary~\ref{cor:Lan by zigzag paths} that the unit $\alpha:H \to H^\sharp D$ of the left Kan extension is the identity on objects, in particular that $\alpha$ is a $\Vcat$-natural isomorphism in case $D$ is fully faithful, that is, if the quantale is integral. 
Hence $T_o = V H_o \cong V H_o^\sharp D_o$ always holds; using now that the counit of the ordinary adjunction $D_o \dashv V$ is again the identity on objects, we obtain an ordinary natural transformation 
$$
\xymatrix@R=20pt{
\Vcat_o
\ar[0,2]^-{H^\sharp_o}
&
&
\Vcat_o
\\
\Set_o
\ar[0,2]_-{T_o}
\ar@{<-}[-1,0]^{V}
\ar@{}[-1,2]|{\nearrow\beta}
&
&
\Set_o
\ar@{<-}[-1,0]_{V}
}
$$

\begin{prop}\label{prop:hsharp-lifting}
The natural transformation $\beta$ is component-wise bijective. 
\end{prop}

\noindent
Consequently, $H^\sharp$ is a lifting of $T$ to $\Vcat$ as in Definition~\ref{def:ext-lift}\eqref{def:lift}. 

\medskip\noindent We will exhibit below another possible way of  lifting $T$ to $\Vcat$.

\begin{exa}[\bf The Kantorovich lifting]\label{exle:kantorovich}

Let $T:\Set \to \Set$ be a functor and let $\hs: T\V_o \to \V_o$ be an arbitrary map (a $\V$-valued predicate lifting), where by slight abuse we identify the quantale with its underlying set of elements.

\medskip\noindent Using the $\V$-valued predicate lifting $\hs$, we can endow $T$ with a $\V$-metric $H:\Set \to \Vcat$ as follows: 
for each set $X$, put $HX$ to be the $\V$-category with set of objects $TX$, and $\V$-distances 
\begin{equation}\label{eq:discrete-kant}
(HX) (A',A) 
= 
\bigwedge_{h: X \to \V_o}  
[(\hs \circ Th)(A'), (\hs \circ Th)(A)]
\end{equation}
Remark that the $\V$-category $HX$ is precisely the initial lift of $TX$ with respect to the family of maps $\hs \circ Th:TX \to \V_o$ indexed by all $h:X \to \V_o$. In particular, the $\V$-distances $(HX) (A',A)$ are the greatest such that all maps $\hs \circ Th$ are actually $\V$-functors $HX \to \V$. 

\noindent
For a map $f:X \to Y$, we let $Hf$ act as $Tf$ on objects. It is easy to see that $Hf:HX \to HY$ is a $\V$-functor. 

\medskip\noindent 
The above defines a $\Vcat$-functor $H:\Set \to \Vcat$ 
(the $\Vcat$-enrichment being a consequence of $\Set$ being free as a $\Vcat$-category) with $V H_o=T_o$, for every map \linebreak $\hs:T\V_o\to \V_o$.  

\medskip\noindent
That is, $H$ is a $\V$-metric for $T$ and we may consider the lifting of $T$ given by the left Kan extension $H^\sharp$, as discussed above. 
\noindent 
In order to obtain (possibly) another lifting of $T$ to $\Vcat$, we make an additional assumption: $\hs$ is $\V$-monotone,%
\footnote{
This generalises the notion of monotone predicate lifting from the two-element quantale to arbitrary $\V$, see~\cite[Section~7]{lmcs:bkv}.} %
in the following sense: for every set $X$ and maps $h,k: X \to \V_o$, the inequality 
\begin{equation}\label{eq:pred-V-monot}
\bigwedge_{x\in X} [h(x),k(x)] \leq \bigwedge_{A \in TX} [(\hs \circ Th)(A), (\hs \circ Tk)(A)]
\end{equation}
holds.%
\footnote{
In categorical terms, it says that the correspondence $\Vcat(DX,\V){\rightarrow}\Vcat(DTX,\V)$, $h \mapsto \hs \circ Th$ is a $\V$-functor.} %
In particular, taking $X=\emptyset$ and $h=k$ the unique map from the empty set, we see that $\V$-monotonicity of $\hs$ entails that self-distances on objects of $H\emptyset$ are $\top$, hence $H^\sharp D=H$ by Theorem~\ref{thm:unitiso}.

\medskip\noindent
Consider the $\Vcat$-functor $\bar H:\Vcat \to \Vcat$ which  maps a $\V$-category $\X$ to the $\V$-category $\bar H \X$ with set of objects $T X_\extra$, and $\V$-homs 
\begin{equation}\label{eq:kant}
\bar H \X(A',A) 
= 
\bigwedge_{h: \X \to \V}  
[(\hs \circ Th_\extra)(A'), (\hs \circ Th_\extra)(A)]
\end{equation}
for every $A', A$ in $TX_\extra$, where this time $h$ ranges over $\V$-functors and the notation $h_\extra$ refers to the object assignment of a $\V$-functor $h$, as in item~\eqref{item 2.a} of the proof of Theorem~\ref{thm:Lan}. 

\noindent 
As above for $H$, the $\V$-category $\bar H \X$ is precisely the initial lift of $TX_\extra$ with respect to the family of maps $\hs \circ Th_\extra:TX_\extra \to \V_o$ (indexed by all $\V$-functors $h:\X \to \V$), in particular~\eqref{eq:kant} provides the greatest $\V$-distances such that for all $\V$-functors $h:\X \to \V$, the composite maps $\hs \circ Th_\extra$ are also $\V$-functors. 

\noindent 
For $f:\X \to \Y$ a $\V$-functor, let $\bar Hf$ act as $Tf_\extra$ on objects. Observe that it is the $\V$-monotonicity of $\hs$ which ensures that $\bar H$ is indeed a $\Vcat$-functor. 
It is clear that this $\Vcat$-functor $\bar H$ is a lifting for $T$. The particular case $\V=[0,\infty]$ provides a generalised Kantorovich lifting as in~\cite{bbkk:calco2015-ext}. 

\medskip\noindent
Relations~\eqref{eq:discrete-kant} and~\eqref{eq:kant} ensure that for each $\V$-category $\X$, the identity on objects is a $\V$-functor $\bar c_\X:HX_\extra\to \bar H\X$. We will now prove that $\bar c_\X$ satisfies~\eqref{eq:couniversal-property}, thus it induces a $\V$-functor $\gamma_\X: H^\sharp \X \to \bar H\X$ such that $\gamma_\X \circ c_\X = \bar c_\X$.  

\noindent 
Let thus $r$ an arbitrary element of the quantale. Then we have
\begin{equation}\label{eq:proof1}
\begin{aligned}
&
\bigwedge_{C \in HX_r} \bar H \X (\bar c_\X Hd^r_0(C), \bar c_\X Hd^r_1(C)) 
&
= 
&
\bigwedge_{(A',A)\in \mathsf{Rel}_T(X_r)} \bar H \X (\bar c_\X (A'), \bar c_\X (A))
\\
&
&
=
&
\bigwedge_{(A',A)\in \mathsf{Rel}_T(X_r)} \bigwedge_{h:\X \to \V} [\hs\circ Th_\extra(A'),\hs\circ Th_\extra(A)]
\\
&
&
\geq
&
\quad \bigwedge_{w\in T\V_r} [\hs\circ Td^r_0(w),\hs\circ Td^r_1(w)]
\end{aligned}
\end{equation}
because 
\[
(Th_\extra \times Th_\extra)(\mathsf{Rel}_T(X_r)) \subseteq \mathsf{Rel}_T((h_\extra \times h_\extra)(X_r)) \subseteq \mathsf{Rel}_T(\V_r)
\]
In the above, we have used that $h:\X \to \V$ is a $\V$-functor and that relation lifting preserves inclusions, and denoted $\V_r = \{(s',s) \mid r\leq [s',s]\}$ and $d^r_0, d^r_1:\V_r \to \V_\extra$ as for any $\V$-category.  

\noindent
Using now~\eqref{eq:pred-V-monot} for the set $\V_r$ and for the maps $d^r_0, d^r_1$, we see that 
\begin{equation}\label{eq:proof2}
r \leq \bigwedge_{(s',s)\in \V_r} [d^r_0(s',s),d^r_1(s',s)] \leq \bigwedge_{w \in T\V_r} [\hs\circ Td^r_0(w),\hs\circ Td^r_1(w)]
\end{equation}
Putting together~\eqref{eq:proof1} and~\eqref{eq:proof2} shows that $\bar c_\X$ satisfies~\eqref{eq:couniversal-property}. 

\medskip\noindent
In conclusion, there is a $\Vcat$-natural transformation $H^\sharp \to \bar H$ whose components act as identity on objects, that is, the $\V$-distances in $H^\sharp \X$ are always smaller than in $\bar H\X$. 

\medskip\noindent Corroborating the above with the particular example of the Hausdorff metric $\bar H$ on the (finite) powerset functor $T=\mathcal P$ exhibited in~\cite{bbkk:calco2015-ext}, it raises the question under which conditions $\bar H \cong H^\sharp$ holds, very much in the spirit of the Kantorovich duality~\cite{villani}. Currently this aspect is still under research and will not be treated in this paper. 
\qed
\end{exa}


\subsection{Closure properties of $\Vcat$-ifications}
\label{sec:closure}

\

\bigskip\noindent
We have seen in Proposition~\ref{prop:Vcatification=functor} that $\Vcat$-ifications are closed under conical colimits. 
In this section, we give sufficient conditions for the $\Vcat$-ifications to be closed under composition and finite products.

\medskip\noindent
We will first discuss when $\Vcat$-ifications are closed under composition:
Generalising a result from \cite[Remark 4.8(4)]{lmcs:bkv} for the case $\V=\Two$, given weak pullback preserving functors $T,T':\Set\to \Set$, we will see that the $\Vcat$-ification of $TT'$ is the composite of the $\Vcat$-ifications $T_\V$ and $T'_\V$ if the quantale $\V$ satisfies the additional assumption of being completely distributive.  

\medskip

\noindent Observe that by the universal property of left Kan extensions, there is a $\Vcat$-natural transformation $(TT')_\V \to T_\V T'_\V$, for any functors $T,T':\Set \to \Set$. 
The component of this natural transformation at a $\V$-category $\X$, namely the $\V$-functor $(TT')_\V\X \to T_\V T'_\V\X$, is easily seen to act as identity on objects,\footnote{
Notice that both $\V$-categories $(TT')_\V \X$ and $T_\V T'_\V \X$ share the same set of objects, namely $TT' X_\extra$.} and on $\V$-homs it provides the inequality
\[
(TT')_\V \X (A',A) \leq T_\V T'_\V \X(A',A)
\]
for all objects $A',A \in TT'X_\extra$.
\medskip

\begin{prop}
Assume that the quantale $\V$ is completely distributive. 
Let $T,T':\Set\to \Set$ be functors which preserve weak pullbacks. 
Then the inequality
\[
T_\V T'_\V \X(A',A) \leq (TT')_\V \X(A',A)
\]
holds for all $A',A \in TT'X_\extra$.
\end{prop}

\begin{proof}
Because $T$ preserves weak pullbacks, $T_\V T'_\V \X (A',A) = \bigvee \{r \mid (A',A) \in \Rel_T((T'_\V \X)_r)\}$ by Lemma~\ref{lem:wpb}. 
Let $r\in \V_o$ such that $(A',A)\in \Rel_T((T'_\V \X)_r)$. Given that $\V$ is completely distributive, to show that
\[
r\leq (TT')_\V \X(A',A) = \bigvee\{s \mid (A',A)\in \Rel_{TT'}(X_s) \}
\]
is the same as
\[
\forall \, s\in \V_o\,. \, s \ll r \, \Longrightarrow (A',A) \in \Rel_{TT'}(X_s)
\] 
because $\{s \mid (A',A)\in \Rel_{TT'}(X_s) \}$ is a downset and $\V$ is  completely distributive. 
Also notice that 
\begin{equation}
\label{eq:r-cut}
(T'_\V \X)_r \subseteq \Rel_T'(X_s) \ \ \  
\end{equation}
holds for all $s \ll r$. Indeed, for $(B',B) \in (T'_\V \X)_r$, we have that  
\[
r \leq T'_\V \X (B',B) = \bigvee \{ s \mid (B',B)\in \Rel_T'(X_s) \} 
\]
using that $T'$ preserves weak pullbacks and Lemma~\ref{lem:wpb}. 
But the relation above is equivalent to 
\[
\forall \, s\in \V_o\,. \, s \ll r \, \Longrightarrow (B',B) \in \Rel_T'(X_s)
\]
Hence $(T'_\V \X)_r \subseteq \Rel_T'(X_s)$ holds for all $s \ll r$. Therefore we have 
\[
(A',A)\in \Rel_T((T'_\V \X)_r) \subseteq \Rel_T(\Rel_{T'}(X_s)) = \Rel_{TT'}(X_s) 
\]
for all $s\ll r$, where the last equality follows from~\cite[Section~4.4]{ckw}. 
\end{proof}

\begin{cor}
Let $\V$ be a completely distributive commutative quantale and $T,T':\Set \to \Set$ functors which preserve weak pullbacks. Then
\[
(TT')_\V = T_\V T'_\V
\]
holds. 
\end{cor}


\noindent Next we show that the $\Vcat$-ification process is also closed under finite products, if some additional conditions are imposed to $\V$.

\begin{prop}
\label{prop:presfinprod}
Let $\V$ be an integral commutative quantale such that finite meets distribute over arbitrary joins in $\V$ (the underlying lattice of $\V$ is a frame). 
Then the $\Vcat$-ification assignment $({-})_\V:[\Set,\Set]\to [\Vcat,\Vcat]$, $T\mapsto T_\V$, preserves finite products of weak pullback preserving functors. 
\end{prop}

\begin{proof}
The empty product, that is, the constant functor $\Set \to \Set$ at the one-element set has as $\Vcat$-ification the constant functor at $\One$. 
Because $e=\top$ holds in $\V$ by hypothesis, the latter is the terminal object in $[\Vcat, \Vcat]$~\cite{bird:PhD}.

\medskip\noindent 
Let us look now at binary products of weak-pullback preserving functors $T,T':\Set \to \Set$. We will denote by $T\times {T'}$ the product in $[\Set,\Set]$, sending $X$ to $TX \times {T'}X$ and shall extend this notation to $[\Vcat,\Vcat]$.

\medskip\noindent 
Now observe that for a $\V$-category $\X$, the $\V$-categories $T_\V\X \times T'_\V\X$ and $(T\times {T'})_\V \X$ share the same objects. 
The universal property of products induces a $\V$-functor $(T \times T')_\V \X \to T_\V \X \times T'_\V \X$, which on the corresponding $\V$-homs says that the inequality
\[
(T\times {T'})_\V(\X)((A',B'),(A,B)) \le (T_\V \times T'_\V)(\X)((A',B'),(A,B))
\]
holds for each objects $A',A$ of $T_\V \X$ and $B', B$ of $T'_\V \X$. 
The converse inequality is shown below, using  Lemma~\ref{lem:wpb} and the frame law:
\begin{align*}
&
(T_\V \times T'_\V)(\X)((A',B'),(A,B)) = 
\\ 
&T_\V(\X)(A',A) \wedge T'_\V(\X)(B',B) = 
\\
&\bigvee\{ t \mid (A',A) \in \Rel_T(X_t))\} \wedge \bigvee\{ s \mid (B',B) \in \Rel_{T'}(X_s))\} =
\\
&\bigvee\{ t \wedge s \mid (A',A) \in \Rel_T(X_t), (B',B) \in \Rel_{T'}(X_s)\} \leq 
\\
&\bigvee\{ t \wedge s \mid (A',A) \in \Rel_T(X_{t\wedge s}), (B',B) \in \Rel_{T'}(X_{t\wedge s})\} \leq
\\
&\bigvee\{ r \mid (A',A) \in \Rel_T(X_r), (B',B) \in \Rel_{T'}(X_r)\} =
\\
&\bigvee\{ r \mid ((A',B'),(A,B)) \in \Rel_{T\times T'}(X_r)\} =
\\
&(T\times T')_\V(\X)((A',B'),(A,B))
\end{align*}
That is, $(T \times T')_\V = T_\V \times T'_\V$ holds for all functors $T,T'$ which preserve weak pullbacks. 
\end{proof}


\section{Solving domain equations across different base categories} \label{sec:beh}

\

\medskip\noindent
In the previous section, we studied those functors on $\Vcat$ that arise as left Kan extensions of functors on $\Set$, with Theorems~\ref{thm:char-Vcatification} and \ref{thm:char-discrete-arities} accounting for the situations depicted on the left and right below.
\begin{equation*}
\xymatrix{
\Vcat\ar[r]^{\overline T} & \Vcat \\
\Set\ar[r]^T\ar[u]^D & \Set \ar[u]_D
}
\quad\quad 
\quad\quad
\quad\quad
\xymatrix{
\Vcat\ar[r]^{\overline T} & \Vcat \\
\Set\ar[u]^D \ar[ur]_H & 
}
\end{equation*}
In this section, we study the relationship between solving recursive equations in $\Set$ and in $\Vcat$.  
In both cases, accounted for by Theorems~\ref{thm:beh-vcatification} and \ref{thm:beh-discrete-arities}, if we have a coalgebra $X\to TX$ or $X\to VHX$, then two states in $X$ are behaviourally equivalent (bisimilar) iff they are behaviourally equivalent in $DX\to \overline TDX$. 
Moreover, while in the first case the behavioural distance of two states is either $0$ or $\infty$ (the states are behaviourally equivalent or not), in the second case the enrichment typically gives rise to non-trivial behavioural distances.

\medskip\noindent
We first review in Section~\ref{sec:2-cat} what in relation to~\cite{street:monad} could be called the ``formal theory of endofunctors''. 
Relevant references for this section include \cite{dubuc,street:monad} in case of algebras for $\V$-enriched monads and  \cite{HermidaJ98,lenisa-power-watanabe} in case of  (co)algebras for ($\V$-enriched) endofunctors.

\medskip\noindent Section~\ref{sec:beh-cat} carries out an analysis of categories of coalgebras for both liftings and extensions, in case of ``change-of-base'' for the quantale $\V$, 
extending results of \cite{calco2011:balan-kurz,lmcs:bkv}. As corollaries we get information about  final coalgebras of $\Vcat$-endofunctors $\overline T$ as in the diagrams above.

\subsection{A 2-categorical approach to coalgebras} \label{sec:2-cat} 

\ 

\bigskip\noindent
In order to relate coalgebras for different type functors on different categories, it will be convenient to work in a (2-)category where the objects are types (endofunctors) and the arrows are certain natural transformations that allow to connect types on different base categories. 
For the basic notions of $2$-categories that we need in this section we refer to~\cite{kelly-street:review-2-cats}.

\begin{defi}\label{def:endo}
Let $\ccal$ be a 2-category. 
The category $\Endo(\ccal)$ has objects $(X,T)$ given by 1-cells $T:X\to X$ in $\ccal$. 
An arrow $(F,\phi):(X,T)\to (X',T')$ is given by a 1-cell $F:X\to X'$  and a 2-cell $\phi:FT\to T'F$. 
A 2-cell $\sigma:(F,\phi)\to (F',\phi')$ is a 2-cell $\sigma:F\to F'$ such that $T'\sigma\circ\phi=\phi'\circ\sigma T$.
$$
\xymatrix{
X 
\ar[r]^T
\ar[d]_F
&
X
\ar[d]^F
\ar@{}[dl]|{\swarrow \phi}
&&
X 
\ar`l[d]`[d]_{F'}[d]
\ar[r]^T
\ar[d]^F_{\sigma \swarrow}
&
X
\ar[d]^F
\ar@{}[dl]|{\swarrow \phi}
\ar@{}[drr]|{=}
&&
X 
\ar[r]^T
\ar[d]_{F'}
&
X
\ar`r[d]`[d]^{F}[d]
\ar[d]_{F'}^{\swarrow \sigma}
\ar@{}[dl]|{\swarrow \phi}
\\
X' 
\ar[r]^{T'}
&
X'
&&
X' 
\ar[r]^{T'}
&
X'
&&
X' 
\ar[r]^{T'}
&
X'
}
$$
Notice that there is a canonical inclusion 2-functor $\mathsf{Incl}: \ccal\to\Endo(\ccal)$ mapping each object $X$ to the pair $(X,\Id)$.
\end{defi}

\medskip\noindent 
The above definition is set-up precisely for the next proposition, which allows us to say, in analogy to the formal theory of monads~\cite{street:monad}, that a general 2-category $\ccal$ admits coalgebras (for endo-1-cells) if the inclusion $\mathsf{Incl}: \ccal\to\Endo(\ccal)$ has a right 2-adjoint.%
\footnote{
Notice that changing the orientation of the 2-cells in Definition~\ref{def:endo} would lead to a different 2-category of endo-1-cells, such that the right 2-adjoint to the inclusion (if exists) produces instead the ``category of algebras'' for the endofunctor $T$.} %
This result has been independently mentioned by several authors at different moments of time (see for example~\cite{lenisa-power-watanabe} for the 2-categorical story, but also the early~\cite{cockett-spencer} which deals with algebras instead of coalgebras):

\begin{prop}
Let $\ccal$ be a 2-category with inserters.%
\footnote{ 
Inserters are dual to coinserters, see Example~\ref{exle:coinserter}.} %
Then the inclusion 2-functor $\mathsf{Incl}$ has a right 2-adjoint $\Coalg:\Endo(\ccal)\to\ccal$, mapping an object $(X,T)$ to the inserter of $\Id:X \to X$ and $T$. 
\end{prop}

\begin{proof}
This follows from the correspondences

\renewcommand{\arraystretch}{1.2}
\begin{tabular}{l}
(1-cell in $\C$)   $\ol F : Z \to \Coalg(X,T)$
\\
\hline
(2-cell in $\C$)  $\phi: F \to TF$  (where $F : Z \to X$ is a 1-cell in $\C$)
\\
\hline
(1-cell in $\Endo(\C)$)  $(F,\phi) : (Z,\Id)\to (X,T)$
\end{tabular}

\smallskip\noindent 
In the above, the second bijection is the definition of 1-cells in $\Endo(\C)$, while the first bijection is due to the fact that $\Coalg(X,T)$ is an inserter of $\Id$ and $T$ (in that order). 
More precisely,
$$ 
\phi:F\to TF
$$
is an ``inserter cone'', and therefore $\phi$ gives rise to $\ol F$ (that is, to the factorisation through the ``inserter cone'').
\end{proof}

\begin{rem} 
In case $\ccal $ is the 2-category of $\Vcat$-categories, $\Vcat$-functors, and $\Vcat$-natural transformations, recall from~\cite[Sect.~6.1]{bird:PhD} how to compute the inserter between two $\Vcat$-functors $F,G:\XX \to \YY$:
\begin{itemize}
\item $\mathsf{Ins}(F,G)$ is the $\Vcat$-category having as objects pairs $(X,\tau)$, where $X$ is an object of $\XX$ and $\tau:\One \to \YY(FX,GX)$ is a $\V$-functor (which only picks an object, still denoted $\tau:FX\to GX$). 
The $\V$-category-hom $\mathsf{Ins}(F,G)((X, \tau),(X', \sigma))$ is the equaliser in $\Vcat$ of the parallel pair
\[
\xymatrix@C=30pt@R=15pt{
&
\One\tensor\XX(X,X')
\ar[0,1]^-{\sigma\tensor F_{X,X'}}
&
\YY(FX',GX')\tensor \YY(FX,FX')
\ar[1,1]^{\phantom{MM}c_{FX,FX',GX'}}
&
\\
\XX(X,X')
\ar[-1,1]^{\cong}
\ar[1,1]_{\cong}
&
&
&
\YY(FX,GX')
\\
&
\XX(X,X')\tensor \One
\ar[0,1]_-{G_{X,X'}\tensor\tau}
&
\YY(FX,GX')\tensor\YY(FX,GX)
\ar[-1,1]_{\phantom{MM}c_{FX,GX,GX'}}
&
}
\]
Explicitly, the objects of $\mathsf{Ins}(F,G)((X, \tau),(X', \sigma))$ are objects $h:X \to X'$ of the $\V$-category $\XX(X,X')$ such that $\sigma \, \cdot \, Fh = Gh \, \cdot \, \tau$, with $\V$-distances \[
\mathsf{Ins}(F,G)((X, \tau),(X', \sigma))(h,k) = \XX(X,X')(h,k)
\]
for $h,k:X\to X'$ as above.
Composition and identity in $\mathsf{Ins}(F,G)((X, \tau),(X', \sigma))$ are induced from the composition in $\XX(X,X')$.

\item 
There is an obvious $\Vcat$-functor $J: \mathsf{Ins}(F,G) \to \XX$ mapping $(X,\tau)$ to $X$, and a $\Vcat$-natural transformation $\iota:FJ \Rightarrow GJ$ with components $\iota_{(X,\tau)}:\One \to \YY(FJ(X,\tau),GJ(X,\tau))$ mapping the unique object $0$ to $\tau$.  
\end{itemize}
\end{rem}

\medskip\noindent In particular, the above yields the description of the $\Vcat$-category $\Coalg(T)$ as $\mathsf{Ins}(\Id, T)$, for any $\Vcat$-functor $T:\Vcat\to \Vcat$. More in detail:
\begin{itemize}
\item Objects are $T$-coalgebras, that is, pairs $(\X,c)$ where $\X$ is a $\V$-category and $c:\X \to T\X$ is a $\V$-functor. 
\item The $\V$-category $\mathsf{Coalg}(T)((\X,c),(\Y,d))$ is the equaliser of 
\[
\xymatrix@C=65pt{%
\Vcat(\X,\Y) \ar@<+.5ex>[r]^-{\Vcat(\X,d)} \ar@<-.5ex>[r]_{\Vcat(c, T\Y) \circ T_{\X,\Y}} & \Vcat(\X,T\Y)
}
\]
that is, it has as objects $\V$-functors $f:\X \to \Y$ such that $d \circ f = T(f) \circ c$, with $\V$-distances
between two such $\V$-functors $f,g: \X \to \Y$ being given by $[\X,\Y](f,g) = \bigwedge_x \Y(f(x),g(x))$. 
\item There is a forgetful $\Vcat$-functor $\mathsf{Coalg}(T) \to \Vcat$ sending a $T$-coalgebra $(\X,c)$ to the underlying $\V$-category $\X$. 
\end{itemize}

\medskip\noindent 
One of the reasons to introduce the formal category $\Endo(\ccal)$ is that it allows the following characterisation of when adjunctions in $\ccal$ can be lifted.

\begin{prop}\label{prop:doctrinaladjunction}
$(L,\phi)\dashv (R,\psi)$ in the 2-category $\Endo(\ccal)$ if and only if $L\dashv R$ holds in $\ccal$, $\psi$ is iso, and $\phi$ is the mate of $\psi^{-1}$ under the adjunction $L\dashv R$.
\end{prop}

\begin{proof} 
By standard doctrinal adjunction~\cite{kelly:doctrinal}.
\end{proof}

\noindent Since $\Coalg$ is a 2-functor, it preserves adjunction of 1-cells, thus we obtain a well-known corollary allowing us to lift an adjunction $L\dashv R$ between base objects to an adjunction between ``categories of coalgebras'':

\begin{cor}\label{cor:lift-adj}
Let $\ccal$ be a 2-category with inserters, and let $L\dashv R:X' \to X$ an adjunction in $\ccal$. 
For any $T:X\to X$, $T':X'\to X'$, such that there is an iso-2-cell $\psi:RT' \to TR$, it follows that:
\begin{enumerate}
\item $R$ lifts to $\tilde R=\Coalg(R,\psi):\Coalg(X',T')\to\Coalg(X,T)$
\item There is an adjunction $\tilde L \dashv\tilde R:\Coalg(X',T')\to\Coalg(X,T)$, where $\tilde L$ is the image under $\Coalg$ of the 1-cell in $\Endo(\ccal)$ given by $L$ and the mate of the inverse of $\psi: RT'\to TR$.
\end{enumerate}
\end{cor}

\noindent 
We end this section with observations on the existence and the computation of limits and colimits in $\Vcat$-categories of coalgebras. Again, these results seem to be folklore, but being unable to find them explicitly in the literature, we provide them below.

\begin{prop}\label{prop:vcatcocomplete}
The $\Vcat$-category $\Coalg(T)$, for a $\Vcat$-functor $T:\Vcat\to \Vcat$, is cocomplete, and has all (weighted) limits that $T$ preserves.
\end{prop}

\begin{proof}
Let $W:\mathcal K^{\mathsf{op}}\to \Vcat$ and $\ol{F}:\mathcal K \to \Coalg(T)$ be $\Vcat$-functors, denoting the weight and the diagram of shape $\mathcal K$, a small $\Vcat$-category. 
By the inserter's universal property, to give $\ol{F}:\mathcal K \to \Coalg(T)$ is the same as to give a $\Vcat$-functor $F:\mathcal K \to \Vcat$, together with a $\Vcat$-natural transformation $\xi:F\to TF$. Then the colimit $W*F$ becomes a $T$-coalgebra with structure  $\alpha:W*F \to T(W*F)$ being the $\V$-functor which corresponds to the identity on $W*F$ under the composite
\begin{align*}
& 
\Vcat(W*F,W*F) 
&
\cong 
&&
[\mathcal K^{\mathsf{op}},\Vcat](W,\Vcat(F -, W*F)) 
\\
&& 
\overset{T}{\to} 
&&
[\mathcal K^{\mathsf{op}},\Vcat](W,\Vcat(TF -, T(W*F)))
\\
&& 
\overset{\xi}{\to} 
&&
[\mathcal K^{\mathsf{op}},\Vcat](W,\Vcat(F -, T(W*F)))
\\
&&
\cong 
&&
\Vcat(W*F,T(W*F)) 
\end{align*}

\noindent Now, let $(\X,c:\X \to T\X)$ be an arbitrary $T$-coalgebra and consider the diagram below:
\[
\xymatrix@C=30pt@R=35pt{
\Coalg(T)((W*F,\alpha), (\X,c))
\ar[d] \ar@{.>}[r]
&
[\mathcal K^{\mathsf{op}},\Vcat](W, \Coalg(T)(\ol{F}-, (\X, c)))
\ar[d]
\\
\Vcat(W*F, \X) 
\ar[r]^-{\cong}
\ar@<-.65ex>[d]
\ar@<+.65ex>[d]
&
[\mathcal K^{\mathsf{op}},\Vcat](W, \Vcat(F-,\X))
\ar@<-.65ex>[d]
\ar@<+.65ex>[d]
\\
\Vcat(W*F, T\X)
\ar[r]^-{\cong}
&
[\mathcal K^{\mathsf{op}},\Vcat](W, \Vcat(F-,T\X))
}
\]
where the bottom square commutes serially and both columns are equalizers -- the left one by construction of $\V$-category-homs in the inserter $\Coalg(T)$, and the right one because the representable $[\mathcal K^{\mathsf{op}},\Vcat](W,-)$ preserves limits, in particular equalizers. 
Consequently, there is a unique arrow (isomorphism)
\[
\xymatrix{
\Coalg(T)((W*F,\alpha), (\X,c))\ar[r]^-{\cong} & [\mathcal K^{\mathsf{op}},\Vcat](W, \Coalg(T)(\ol{F}-, (\X, c)))
}
\]
exhibiting $(W*F, \alpha)$ as the $W$-weighted colimit of $\ol{F}$.

\medskip\noindent A similar computation that we leave to the reader shows that $\Coalg(T)$ has all weighted limits that $T$ preserves. 
\end{proof}


\subsection{Relating behaviours across different base categories} \label{sec:beh-cat} 

\ 

\bigskip\noindent
We consider ``change of base'' across different $\V$. 
We therefore now write $D^\V:\Set\to\Vcat$ and $V^\V:\Vcat_o\to\Set_o$ to distinguish the discrete and forgetful functors for different $\V$. 
We drop the superscript $(-)^\V$ in case $\V=\Two$.

\medskip\noindent
In the previous section, we have shown that every $\Vcat$-functor $H:\Set\to\Vcat$ has a left Kan extension along $D^\V$, denoted $H^\sharp$. 
Now, each such functor induces an ordinary set-endofunctor simply by forgetting the $\Vcat$-structure
$$
\xymatrix{\Set_o \ar[r]^-{H_o} & \Vcat_o \ar[r]^-{V^\V} & \Set_o }
$$
Moreover, $H^\sharp$ is a lifting of $V^\V H_o$ according to Proposition~\ref{prop:hsharp-lifting}. 

\medskip\noindent
In the special case when $H$ is $D^\V T$, the above composite gives back $T$ (formally, it is $T_o$), and $H^\sharp$ is $T_\V$, the $\Vcat$-ification of $T$. 

\medskip\noindent
We will see how the corresponding behaviours are related.
In particular, we will show at the end of the section that $H^\sharp$ and $V^\V H_o$, hence $T_\V$ and $T$, induce the same behavioural equivalence (bisimilarity), while $H^\sharp$ may additionally induce a behavioural pseudometric.

\begin{rem}\label{rem:change-of-base}

For each commutative quantale $\V$, the inclusion (quantale morphism) \linebreak $\d:\Two \to \V$ given by $0\mapsto \bot, \  1 \mapsto e$, has a right adjoint (as it preserves suprema), denoted $\uv:\V \to \Two$ , which maps an element $r$ of $\V$ to $1$ if $e\le r$, and to $0$ otherwise.\footnote{
Notice that $\uv$ is only a \emph{lax} morphism of quantales, being right adjoint.} This induces as usual the \emph{change-of-base} adjunction (even a 2-adjunction, see~\cite{eilenberg-kelly})
$$
\xymatrix@C=45pt{ 
\Two 
\ar@/^1.5ex/[r]^{\d} 
\ar@{}[r]|{\perp}
&
\V
\ar@/^1.5ex/[l]^{\uv}
&
\mapsto
&
\Pre
\ar@/^1.5ex/[r]^{\d_*} 
\ar@{}[r]|{\perp}
&
\Vcat
\ar@/^1.5ex/[l]^{\uv_*}
}
$$
Explicitly, the functor $\d_*$ maps a preordered set $X$ to the $\V$-category  $\d_* X$ with same set of objects, and $\V$-homs given by $\d_* X(x',x)=e$ if $x'\leq x$, and $\bot$ otherwise. 
Its right adjoint transforms a $\V$-category $\X$ into the preorder $\uv_*\X$ with same objects again, and order $x'\leq x$ iff $e \leq \X(x',x)$ holds. 
Hence $\d_* X$ is the free $\V$-category on the preorder $X$, while $\uv_* \X$ is the underlying ordinary category (which happens to be a preorder, due to the simple nature of quantales) of the $\V$-category $\X$.   \qed

\medskip\noindent
Note that $\d_*$ is both a $\Vcat$-functor and a $\Pre$-functor, while its right adjoint $\uv_*$ (in fact, the whole adjunction $\d_*\dashv \uv_*$) is only $\Pre$-enriched.
\end{rem}

\begin{rem}\label{rem:change-of-base-cd}
If $\V$ is nontrivial and integral, the embedding $\d:\Two \to \V$ has also a left adjoint $\cv:\V \to \Two$, given by $\cv(r)=0$ iff $r=\bot$, otherwise $\cv(r)=1$. 
Notice that $\cv$ is only a colax morphism of quantales, in the sense that $\cv(e)\leq 1$ (in fact, here we have equality) and $\cv(r\tensor s) \leq \cv(r) \land \cv(s)$, for all $r,s$ in $\V$. 
 
\medskip\noindent
We will in the sequel assume that $\cv$ is actually a morphism of quantales. 
The reader can check that this boils down to the requirement that $r\tensor s =\bot$ in $\V$ implies $r=\bot$ or $s=\bot$. That is, the quantale has no {\em zero-divisors}. 
All our examples satisfy this assumption.

\medskip\noindent If the quantale has no zero-divisors, $\d_*$ also has a left adjoint $\cv_*$ mapping a $\V$-category $\X$ to the preorder $\cv_* \X$ with same objects and $x'\leq x$ in ${\cv_*\X}$ iff $\X(x',x)\neq \bot$. Moreover, the adjunction $\cv_* \dashv \d_*$ is $\Vcat$-enriched:
$$
\xymatrix@C=45pt{ 
\Two 
\ar@/_1.5ex/[r]_{\d} 
\ar@{}[r]|{\perp}
&
\V
\ar@/_1.5ex/[l]_{\cv}
&
\mapsto
&
\Pre
\ar@/_1.5ex/[r]_{\d_*} 
\ar@{}[r]|{\perp}
&
\Vcat
\ar@/_1.5ex/[l]_{\cv_*}
}
$$
\end{rem}

\medskip\noindent
 From the above remark we obtain the following propositions.

\begin{prop}
Let $\V$ be an arbitrary commutative quantale and let $\widehat T:\Pre \to \Pre$ be a locally monotone functor (that is, $\Pre$-enriched) and $\ol{T}:\Vcat\to \Vcat$ be a lifting of $\widehat T$ to $\Vcat$ (meaning that $\ol T$ is a $\Vcat$-functor such that $\uv_* \ol{T} \cong \widehat T \uv_*$ holds). 
Then the locally monotone adjunction $\d_*\dashv \uv_*$ lifts to a locally monotone adjunction $\widetilde \d_*\dashv \widetilde \uv_*$ between the associated $\Pre$-categories of coalgebras.
$$
\xymatrix@R=33pt{
&
\Coalg(\widehat T) 
\ar[d]
\ar@/^8pt/[rr]^{\widetilde \d_*}
\ar@{}[rr]|\bot
\ar@{<-}@/_8pt/[rr]_{\widetilde \uv_*}
&&
\Coalg(\ol T )
\ar[d]
&
\\
&
\Pre
\ar@/^8pt/[rr]^{\d_*}
\ar@{}[rr]|\bot
\ar@{<-}@/_8pt/[rr]_{\uv_*}
\POS!L(.7),\ar@(ul, dl)_{\widehat T}
&&
\Vcat 
\POS!R(.7),\ar@(ur, dr)^{\ol T}
&
}
$$
\end{prop}

\begin{prop}

Assume that $\V$ is a non-trivial integral commutative quantale without zero divisors. 
Let again $\widehat T:\Pre \to \Pre$ be a locally monotone functor, but this time consider $\ol{T}:\Vcat\to \Vcat$ be an extension of $\widehat T$ to $\Vcat$ (meaning that $\ol T$ is a $\Vcat$-functor, such that $\ol{T} \d_* \cong \widehat T \d_*$ holds). 
Then the $\Vcat$-adjunction $\cv_*\dashv \d_*$ lifts to a $\Vcat$-adjunction $\widetilde \cv_*\dashv \widetilde \d_*$ between the associated $\Vcat$-categories of coalgebras.
$$
\xymatrix@R=33pt{
&
\Coalg(\widehat T) 
\ar[d]
\ar@/_8pt/[rr]_{\widetilde \d_*}
\ar@{}[rr]|\bot
\ar@{<-}@/^8pt/[rr]^{\widetilde \cv_*}
&&
\Coalg(\ol{T})
\ar[d]
&
\\
&
\Pre
\ar@/_8pt/[rr]_{\d_*}
\ar@{}[rr]|\bot
\ar@{<-}@/^8pt/[rr]^{\cv_*}
\POS!L(.7),\ar@(ul, dl)_{\widehat T}
&&
\Vcat
\POS!R(.7),\ar@(ur, dr)^{\ol T}
&
}
$$ 
\end{prop}

\noindent We return now to the discrete functor $D^\V:\Set \to \Vcat$. 
It is easy to see that it decomposes as 
$\xymatrix@C=17pt@1{\Set \ar[r]^-D & \Pre \ar[r]^-{\d_*} & \Vcat}$. 
Additionally, recall the following (see 
also Example~\ref{ex:extensions}\eqref{ex:unicity-of-extens}):
\begin{enumerate}
\item
There are locally monotone functors $D:\Set\to\Pre$, $C:\Pre\to\Set$, where $D$ maps a set to its discrete preorder and $C$ maps a preorder to its set of connected components.

\item
There is a chain $C_o \dashv D_o \dashv V:\Pre \to \Set$
of ordinary adjunctions where $V$ is the underlying-set forgetful functor.

\item
The locally monotone adjunction $C \dashv D$ is $\Vcat$-enriched. 

\end{enumerate}

\medskip\noindent
The next two propositions from \cite{calco2011:balan-kurz} are similar to the two above, but connect $\Set$ with $\Pre$ instead of $\Pre$ with $\Vcat$.

\begin{prop}
Let $T:\Set\to \Set$ and $\widehat T:\Pre \to \Pre$ an extension of $T$ (a locally monotone functor such that $DT \cong \widehat T D$). Then the locally monotone adjunction $C\dashv D$ lifts to a locally monotone adjunction $\widetilde C \dashv \widetilde D$ between the associated categories of coalgebras:
$$
\xymatrix@R=33pt{
&
\Coalg(T) 
\ar[d]
\ar@/_8pt/[rr]_{\widetilde D}
\ar@{}[rr]|\bot
\ar@{<-}@/^8pt/[rr]^{\widetilde C}
&&
\Coalg(\widehat T)
\ar[d]
&
\\
&
\Set
\ar@/_8pt/[rr]_{D}
\ar@{}[rr]|\bot
\ar@{<-}@/^8pt/[rr]^{C}
\POS!L(.7),\ar@(ul, dl)_{T}
&&
\Pre
\POS!R(.7),\ar@(ur, dr)^{\widehat T}
&
}
$$
\end{prop}

\noindent Consequently, $\widetilde D$ will preserve limits, in particular, the final coalgebra (if it exists).

\begin{prop}
Let $T:\Set\to \Set$ and $\widehat T:\Pre \to \Pre$ a lifting of $T$ (an ordinary functor such that $TV \cong V\widehat T$).
Then the ordinary adjunction $D_o\dashv V$ lifts to an ordinary adjunction $\widetilde D_o \dashv \widetilde V$ between the associated categories of coalgebras.
$$
\xymatrix@R=33pt{
\Coalg(T) 
\ar[d]
\ar@/^8pt/[rr]^{\widetilde D_o}
\ar@{}[rr]|\bot
\ar@{<-}@/_8pt/[rr]_{\widetilde V}
&&
\Coalg(\widehat T)
\ar[d]
\\
\Set
\ar@/^8pt/[rr]^{D_o}
\ar@{}[rr]|\bot
\ar@{<-}@/_8pt/[rr]_{V}
\POS!L(.7),\ar@(ul, dl)_{T}
&&
\Pre
\POS!R(.7),\ar@(ur, dr)^{\widehat T}
\\
}
$$ 
\end{prop}

\noindent
Consequently, $\widetilde V$ will preserve ordinary limits; in particular, the underlying set of a final $\widehat T$-coalgebra will be a final $T$-coalgebra.


\medskip\noindent 
Let now $T:\Set\to \Set$ and denote by $T_\Two$ is $\Two\mbox{-}\mathsf{cat}$-ification, that is, its $\Pre$-ification~\cite{lmcs:bkv}. We plan to see how $T_\Two$ and $T_\V$, the $\Vcat$-ification of $T$, are related. We start by the following observation:


\begin{prop}\label{prop:preordvcatdense}
The embedding $\d_\ast:\Pre\to\Vcat$ is dense.
\end{prop}

\begin{proof}
We have shown in Theorem~\ref{thm:dense-pres} that 
$D^\V = \d_*  D$ is $\Vcat$-dense. Using that
$\d_\ast$ is fully faithful, it follows from~\cite[Theorem~5.13]{kelly:book} that both $D$ and $\d_*$ are $\Vcat$-dense and that $\d_*=\mathsf{Lan}_D (D^\V)$ holds.
\end{proof}

\begin{rem}\label{rem:two-step-vcatification}
Let $T:\Set\to \Set$ be a set-functor and $T_\Two$ its $\Pre$-ification as above. Then the $\Vcat$-ification $T_\V$ of $T$ can be computed in two stages, as follows: 
\begin{align*}
T_\V \ 
& 
= 
&& 
\mathsf{Lan}_{D^\V} (D^\V T) 
\\
&
= 
&& 
\mathsf{Lan}_{(\d_* D)} (\d_* D T) = \mathsf{Lan}_{\d_*} (\mathsf{Lan}_D (\d_* DT)) 
& 
\mbox{\quad (by \cite[Theorem~4.47]{kelly:book},}
\\
&
&&
&
\mbox{because $\d_\ast$ is fully faithful)} 
\\
&
\cong 
&& 
\mathsf{Lan}_{\d_*} (\mathsf{Lan}_D (\d_* T_\Two D)) 
&
\mbox{ (because $DT \cong T_\Two D$)}
\\
& 
\cong 
&& 
\mathsf{Lan}_{\d_*} (\d_* T_\Two) 
\end{align*}
where the last isomorphism holds because $\mathsf{Lan}_D (\d_* T_\Two D) \cong \d_* T_\Two$. To verify the latter isomorphism, notice first that the $\Pre$-enriched left Kan extension $\mathsf{Lan}_D (\d_* T_\Two D)$ is also the $\Vcat$-enriched left Kan extension of $\d_* T_\Two D$ along $D$, by a change-of-base argument as in~\cite[Theorem~1.7.1]{verity}. Next, apply~\cite[Theorem~5.29]{kelly:book} to the composite $\Pre$-functor $\d_* T_\Two$, using the density presentation of $D$ exhibited in~\cite{lmcs:bkv} and the fact that $\d_*$ is $\Pre$-left adjoint.   
\end{rem}

\noindent The remark above says that 
the $\Vcat$-ification of an endofunctor $T$ of $\Set$ can be obtained by taking first the $\Pre$-ification $T_\Two:\Pre \to \Pre$ and then the left Kan extension of 
$\d_\ast T_\Two$
along $\d_*$,  
as in 
\begin{equation*}
\xymatrix{
\Vcat \ar[r]^{T_\V} & \Vcat \\
\Pre \ar[r]^{T_\Two}\ar[u]^{\d_\ast} & \Pre\ar[u]_{\d_\ast} \\
\Set \ar[r]^{T}\ar[u]^{D} & \Set\ar[u]_{D} 
}
\end{equation*}

\medskip\noindent Putting things together we now obtain

\begin{thm}\label{thm:beh-vcatification}
Let $\V$ be a non-trivial integral commutative quantale without zero divisors, and $T:\Set \to \Set$ an arbitrary endofunctor, with $\Vcat$-ification $T_\V:\Vcat \to \Vcat$. 
Then the $\Vcat$-adjunctions $C\dashv D:\Set \to \Pre$, $\cv_* \dashv \d_*:\Pre \to \Vcat$ lift to $\Vcat$-adjunctions between the associated $\Vcat$-categories of coalgebras:
$$
\xymatrix@R=33pt{
&
\Coalg(T) 
\ar[d]
\ar@/_8pt/[rr]_{\widetilde D}
\ar@{}[rr]|\bot
\ar@{<-}@/^8pt/[rr]^{\widetilde C}
&&
\Coalg(T_\Two)
\ar[d]
\ar@/_8pt/[rr]_{\widetilde \d_*}
\ar@{}[rr]|\bot
\ar@{<-}@/^8pt/[rr]^{\widetilde \cv_*}
&&
\Coalg(T_\V)\ar[d]
&
\\
&
\Set
\ar@/_8pt/[rr]_{D}
\ar@{}[rr]|\bot
\ar@{<-}@/^8pt/[rr]^{C}
\POS!L(.7),\ar@(ul, dl)_{T}
&&
\Pre
\ar@/_8pt/[rr]_{\d_*}
\ar@{}[rr]|\bot
\ar@{<-}@/^8pt/[rr]^{\cv_*}
&&
\Vcat
\POS!R(.7),\ar@(ur, dr)^{T_\V}
&
}
$$
\end{thm}

\medskip\noindent
Since the $\Vcat$-ification $T_\V$ of an endofunctor $T$ on $\Set$ is supposed to be ``$T$ in the world of $\V$-categories'', the theorem above confirms the expectation that final $T_\V$-coalgebras have a discrete metric. In fact, we can say that the final $T$-coalgebra is the final $T_\V$-coalgebra, if we consider $\Coalg(T)$ as a full (enriched-reflective) subcategory of $\Coalg(T_\V)$.


\medskip\noindent
The next theorem deals with a more general situation where the final metric-coalgebra is the final set-coalgebra with an additional metric. This includes in particular the case where $\ol T$ is $H^\sharp$ for some $H:\Set\to\Vcat$ with $V^\V H_o=T_o$.

\begin{thm}\label{thm:beh-discrete-arities}
Let $\V$ be a commutative quantale, $T:\Set \to \Set$ be an arbitrary endofunctor, $\widehat T:\Pre \to \Pre$ a lifting of $T$ to $\Pre$, and $\ol T:\Vcat \to \Vcat$ be a lifting of $\widehat T$ to $\Vcat$. 
Then the ordinary adjunction $D_o\dashv V: \Pre\to\Set$ and the $\Pre$-adjunction $\d_* \dashv \uv_*:\Vcat\to\Pre$ lift to adjunctions between the associated $\Vcat$-categories of coalgebras
$$
\xymatrix@R=33pt{
&
\Coalg(T) 
\ar[d]
\ar@{<-}@/_8pt/[rr]_{\widetilde V}
\ar@{}[rr]|\bot
\ar@/^8pt/[rr]^{\widetilde D_o}
&&
\Coalg(\widehat T)
\ar[d]
\ar@{<-}@/_8pt/[rr]_{\widetilde \uv_*}
\ar@{}[rr]|\bot
\ar@/^8pt/[rr]^{\widetilde \d_*}
&&
\Coalg(\overline T)\ar[d]
&
\\
&
\Set
\ar@{<-}@/_8pt/[rr]_{V}
\ar@{}[rr]|\bot
\ar@/^8pt/[rr]^{D_o}
\POS!L(.7),\ar@(ul, dl)_{T}
&&
\Pre
\ar@{<-}@/_8pt/[rr]_{\uv_*}
\ar@{}[rr]|\bot
\ar@/^8pt/[rr]^{\d_*}
&&
\Vcat
\POS!R(.7),\ar@(ur, dr)^{\ol T}
&
}
$$
\end{thm}

\noindent
It follows that $\overline T$ and $T$ induce the same notion of behavioural equivalence (bisimilarity). Nevertheless, the final $\overline T$-coalgebra can provide additional information about order and metric of non-bisimilar elements.

\begin{exa}\label{ex:beh-metric-stream}
Recall from Example~\ref{ex:functors-with-V-metric}\eqref{ex:machine} the functor $TX=X^A \times B$ and its lifting $H^\sharp \X = \X^A \otimes \B$. 
Assume that the quantale is integral. 
Then the final $H^\sharp$-coalgebra is the power $\V$-category $\B^{A^*}$ having as objects all functions mapping each finite sequence of inputs to the last observable output in $\B$, with $\V$-distances
\[
\B^{A^*}(h,k) = \bigwedge_{\mathsf l\in A^*} \B(h(\mathsf l),k(\mathsf l))
\] 
for each pair of behaviour functions $h,k:A^* \to B$. 

\medskip\noindent In particular, for the lifting $H^\sharp$ of the stream functor from Example~\ref{ex:functors-with-V-metric}\eqref{ex:stream} we obtain the final $H^\sharp$-coalgebra as the $\V$-category $\B^{\mathbb N} $ of streams over $B$. 
The lifting $H^\sharp$ of the deterministic automata functor from Example~\ref{ex:functors-with-V-metric}\eqref{ex:det-autom} given by $H^\sharp \X = \X^A \otimes \Two_{r,s}$ has now as final coalgebra the ``generalised metric space'' $\Two_{r,s}^{A^*}$ of languages over the alphabet $A$.
\end{exa}



\section{Conclusion}

\medskip\noindent 
This paper is part of a larger endeavour extending set-based coalgebra to $\Vcat$-based coalgebra, see for example \cite{VelebilK11,KurzV13,bkpv:rel-lift,lmcs:bkv,KurzV17,BabusK16,dahlqvist-kurz:calco17}. 
Here, we showed that every functor $H:\Set\to\Vcat$ has a left-Kan extension $H^\sharp:\Vcat \to \Vcat$, and that the final $H^\sharp$-coalgebra is the corresponding final coalgebra over $\Set$ equipped with a $\V$-metric. 

\medskip\noindent 
There are several directions in which to expand our results. 
For example, it would be interesting to move from commutative to general quantales and to quantaloids~\cite{rosenthal, stubbe:quant}. 
There is also the question whether one can extend not only functors but also monads in a uniform way, which could be related to (metric) trace (bi)simulation as in~\cite{bbkk:calco2015-ext}.

\medskip\noindent 
Coalgebraically, it would be interesting to further develop the topic, barely touched upon in the last section, of behavioural pseudo-metrics~\cite{approx-bisim-prob,dgjp:metric-markov-proc,rutten:cmcs98,worrell:cmcs00}, while on the logical side, we aim to combine this paper with \cite{KurzV17,BabusK16} in the pursuit of an equational approach to quantitative reasoning about coalgebras/transition systems, an objective related to recent work on quantitative algebraic reasoning~\cite{mpp:quant-alg-reason}. 


\section*{Acknowledgements} 

\medskip\noindent We thank the anonymous referees of CALCO and LMCS for their valuable comments and patience that allowed us to improve the presentation of our results.




\end{document}